\documentclass{my_m2an} 

\usepackage{lop}
\usepackage{loc}

\hypersetup{
    pdftitle = {A MUSCL-like finite volume approximation of the momentum convection operator for low-order nonconforming face-centred discretizations},
    pdfauthor = {Aubin Brunel, Raphaèle Herbin, Jean-Claude Latché},
    pdfkeywords = {Fluid flows, convection operator, staggered meshes, MUSCL, kinetic energy balance, stability, incompressible flows, compressible flows.}}

\title{A MUSCL-like finite volume approximation of the momentum convection operator for low-order nonconforming face-centred discretizations}
\date{January 2023}
 
\begin{document}

\author{A. Brunel}
\address{Aix-Marseille Universit\'e, Centre de Math\'ematiques et Informatique,  39 rue Joliot-Curie, 13453 Marseille Cedex 13, France (aubin.brunel@univ-amu.fr)}
\author{R. Herbin}
\address{Aix-Marseille Universit\'e, Centre de Math\'ematiques et Informatique, 39 rue Joliot-Curie, 13453 Marseille Cedex 13, France (raphaele.herbin@univ-amu.fr)}
\author{J.-C. Latch\'e}
\address{IRSN, BP 13115, St-Paul-lez-Durance Cedex, France (jean-claude.latche@irsn.fr)}

\begin{abstract}
    We propose in this paper a discretization of the momentum convection operator for fluid flow simulations on quadrangular or hexahedral meshes.
    The space discretization is performed by the low-order nonconforming Rannacher-Turek finite element: the scalar unknowns are associated to the cells of the mesh, while the velocities unknowns are associated to the edges or faces.
    The momentum convection operator is of finite volume type, and its almost second order expression is derived by a MUSCL-like technique.
    The latter is of algebraic type, in the sense that the limitation procedure does not invoke any slope reconstruction, and is independent from the geometry of the cells.
    The derived discrete convection operator applies both to constant or variable density flows, and may thus be implemented in a scheme for incompressible or compressible flows.
    To achieve this goal, we derive a discrete analogue of the computation $u_i\,(\partial_t(\rho u_i) + \dive(\rho u_i \bfu)= \frac 1 2 \partial_t(\rho u_i^2) + \frac 1 2 \dive(\rho u_i^2 \bfu)$ (with $\bfu$ the velocity, $u_i$ one of its component, $\rho$ the density, and assuming that the mass balance holds) and discuss two applications of this result: 
    firstly, we obtain stability results for a semi-implicit in time scheme for incompressible and barotropic compressible flows;
    secondly, we build a  consistent, semi-implicit in time scheme that is based on the discretization of the internal energy balance rather than the total energy.
    The performance of the proposed discrete convection operator is assessed by numerical tests on the incompressible Navier-Stokes equations, the barotropic and the full compressible Navier-Stokes and the compressible Euler equations.
\end{abstract}


\maketitle
%
%
\section{Introduction}

When designing numerical schemes for fluid flow simulations, combining a finite element approximation of diffusion terms with a finite volume discretization of the convection operator is an appealing solution, sometimes found in the literature.
Indeed, the diffusion term may be easily discretized using the finite element method with minimal mesh restrictions while preserving the stability, \ie\ the control of a (possibly discrete) $H^1$-norm, but the discretization of the convection term is less straightforward, since standard finite element methods may yield numerical instabilities, especially in the convection dominated case.
Tackling this problem amounts to introduce some upwinding in the scheme, and, to this purpose, many solutions have been explored in the context of the finite volume; finite-volume convection operators respecting both some monotonicity and $L^2$-stability properties (including, for the latter item, a local discrete entropy or, in the world of fluid flow, a kinetic energy balance) have been obtained in this way.
Several authors have thus proposed discretizations combining finite elements and finite volumes, to take benefit of the best of both worlds, see for instance \cite{ohm-84-tec, ang-91-num,fei-95-com, fei-97-con, eym-06-com, cal-19-com} and references therein.
These works may address convection-diffusion or Navier-Stokes equations, using preferably finite elements approximations of accuracy compatible with finite volumes, \ie\ low-order elements.
For the incompressible Navier-Stokes equations or for low-Mach compressible flows, associating this property with the \emph{inf-sup} stability requirement suggests turning to low-order nonconforming elements, namely the low-order Crouzeix-Raviart element for simplicial meshes \cite{cro-73-con} or the Rannacher-Turek element for quadrangles and hexahedra \cite{ran-92-sim}.
An application of this strategy for the discretization of the stationary incompressible Navier-Stokes equations by Crouzeix-Raviart finite elements may be found in \cite{sch-96-opt}; extension to quasi-incompressible unsteady flows, both with the Crouzeix-Raviart and Rannacher-Turek finite elements, is performed in \cite{ans-11-sta}.

\medskip
In most of the above cited papers, only a first-order upwinding technique is considered, leading to diffusive approximations.
Increasing the order of the scheme and its precision while preserving its stability can be tricky, since naive higher-order methods might lead to spurious oscillations.
As already mentioned, successful methods exist to achieve this goal; such a now well-known method is Van Leer's so-called MUSCL scheme \cite{van-79-tow}.
This technique was firstly used for hyperbolic conservation laws in one space dimension; extending it to multi-dimensional problems on general meshes is a challenging task, due to the so-called slope construction involved in the limitation step, see for instance \cite{cal-10-sta, buf-10-mon, cla-10-sta, let-15-mul}.
A numerical scheme circumventing this problem for the transport operator is proposed in \cite{pia-13-for}; it relies on the observation that the requirements for the scheme to satisfy the maximum principle may be substituted to the usual limitation technique, yielding a limitation step of purely algebraic type, and so free of any geometric consideration.

\medskip
The continuous momentum convection operator that we consider here takes the following generic form:
\begin{equation} \label{eq:cont_conv}
\mathcal C(\rho,u_i)=\partial_t(\rho u_i)+\dive(\rho u_i \bfu)
\end{equation}
where $\rho$ is the density of the fluid and $\bfu$ its velocity (so, for $1 \leq i \leq d$, $u_i$ stands for the $i$-th component of the velocity).
It may be recast under the form of a transport operator provided that a mass balance equation holds, that is 
\begin{equation} \label{eq:cont_mass_balance}
\partial_t \rho + \dive(\rho \bfu)=0.
\end{equation}
Indeed, we have:
\begin{equation}
\partial_t(\rho u_i)+\dive(\rho u_i \bfu)=\underbrace{u_i\big(\partial_t \rho + \dive(\rho \bfu)\big)}_{=0}+\rho \bigl(\partial_t u_i+\bfu \cdot \gradi u_i\bigr).
\end{equation}
This formulation shows that the operator $\mathcal C$ satisfies a discrete maximum principle.
In addition, a standard manipulation of partial derivatives yields:
\begin{equation}\label{eq:kec}
u_i\ \mathcal C(\rho,u_i) = \frac 1 2\,\rho \bigl(\partial_t u_i^2+\bfu \cdot \gradi u_i^2\bigr) = \partial_t(\rho \frac{u_i^2} 2)+\dive(\rho \frac{u_i^2} 2 \bfu).
\end{equation}
A finite volume discretization of the operator $\mathcal{C}$ based on the previously cited algebraic MUSCL method \cite{pia-13-for} was recently derived, first for simplicial or quadrangular (or hexahedral) meshes \cite{gas-18-mus}, and then on more general possibly hybrid meshes \cite{bru-22-sta}.
Here we recall this construction for a space discretization using the unknowns of the Rannacher-Turek finite element (Section \ref{sec:op}) and derive a discrete analogue of Equation \eqref{eq:kec} satisfied by this discrete convection operator (Section \ref{sec:kin}).
The form of $\mathcal C$ is quite general, and the operator built here may be applied as well to incompressible as to compressible flows.
Two results support this issue.
First, for an advection diffusion with an implicit-in-time discretization of the diffusion term (while the MUSCL approximation of the convection term is explicit), integrating the discrete counterpart of \eqref{eq:kec} in space yields a stability estimate, valid for time steps lower than a limit depending on the diffusion coefficient and the mesh regularity, but independent of the space step (Section \ref{sec:stab}); this estimate is the essential argument that is required to control the kinetic energy for incompressible flows or the total energy for barotropic flows.
Second, we show how to build, once again from the discrete version of \eqref{eq:kec}, a consistent scheme for the Euler equations based on the solution of the internal energy balance to preserve the positivity of the latter variable (Section \ref{sec:corr_Euler}).
To this aim, having at hand a local (\ie\ written on each cell and not integrated over the space domain) kinetic energy balance is necessary.
Finally, numerical experiments are performed (Section \ref{sec:num}) to assess the stability, consistency, and accuracy of the proposed scheme for the incompressible and compressible Navier-Stokes equations. 
%
%
\section{Space and time discretizations} \label{sec:mesh}

We first define a primal mesh $\mesh$ by splitting $\Omega$ into a finite family of disjoint quadrangles (if $d=2$) or hexahedra (if $d=3$) denoted by $K$ and called control volumes or cells.
We then denote by $\edges$ the set of faces of the mesh $\mesh$; for $K \in \mesh$, $\edges(K)$ stands for the set of faces of $K$ and we thus have $\partial K = \displaystyle {\cup_{\edge \in \edges(K)} \overline \edge}$. 
Any face $\edge \in \edges$ is either a part of the boundary of $\Omega$, \ie\ $\edge\subset \partial\Omega$, in which case $\edge$ is said to be an external face, or there exists $(K,L)\in \mesh^2$ with $K \neq L$ such that $\overline K \cap \overline L  = \overline \edge$: we denote in this case $\edge = K|L$ and $\edge$ is said to be an internal face.
We denote by $\edgesext$ and $\edgesint$ the set of external and internal faces. 
For $K \in \mesh$ and $\edge \in \edges$, we denote by $|K|$ the measure of $K$ and by $|\edge|$ the $(d-1)$-measure of the face $\edge$. 

\medskip
The discretization is staggered in the sense that the scalar and vector unknowns are not colocated:
\blist
\item the unknowns associated to the density, and to any other scalar variable involved in the problem, as for instance the pressure, are associated with the cells of the primal mesh $\mesh$; limiting the list of set of scalar fields to the density, the pressure $p$ and the internal energy $e$ (which will be sufficient for the numerical applications presented in Section \ref{sec:num}), the corresponding unknowns are denoted by $(\rho_K)_{K\in\mesh}$, $(p_K)_{K\in\mesh}$ and $(e_K)_{K\in\mesh}$;
\item  the degrees of freedom for the velocity are defined on a dual mesh using the Rannacher-Turek non-conforming low-order finite element approximation \cite{ran-92-sim} and are denoted $(\bfu_\edge)_{\edge \in \edges}$ with $\bfu_\edge= (u_{\edge,1},\ldots,u_{\edge,d})$; they are identified with the mean value of the velocity component over the face.
\elist
The dual mesh is constructed as follows (see Figure \ref{fig:mesh}): if $K \in \mesh$ is a rectangle or a rectangular cuboid, we denote by $x_K$ the mass center of $K$ and we construct $D_{K,\edge}$ as the cone with basis $\edge$ and with vertex $x_K$; this definition is extended to a general cell $K$, by supposing that $K$ is split in the same number of sub-cells (the geometry of which does not need to be specified) and with the same connectivity. 
We now define $D_\edge$, the dual cell associated to $\edge$, as $D_\edge=D_{K,\edge} \cup D_{L,\edge}$ if $\edge=K|L \in \edgesint$ and $D_\edge=D_{K,\edge}$ if $\edge\in\edges(K)\cap\edgesext$; its measure is denoted by $|D_\edge|$.
We then denote by $\edgesd(D_\edge)$ the set of dual faces of $D_\edge$, and by $\edged=D_\edge|D_{\edge'}$ the face separating two dual cells $D_\edge$ and $D_{\edge'}$. 

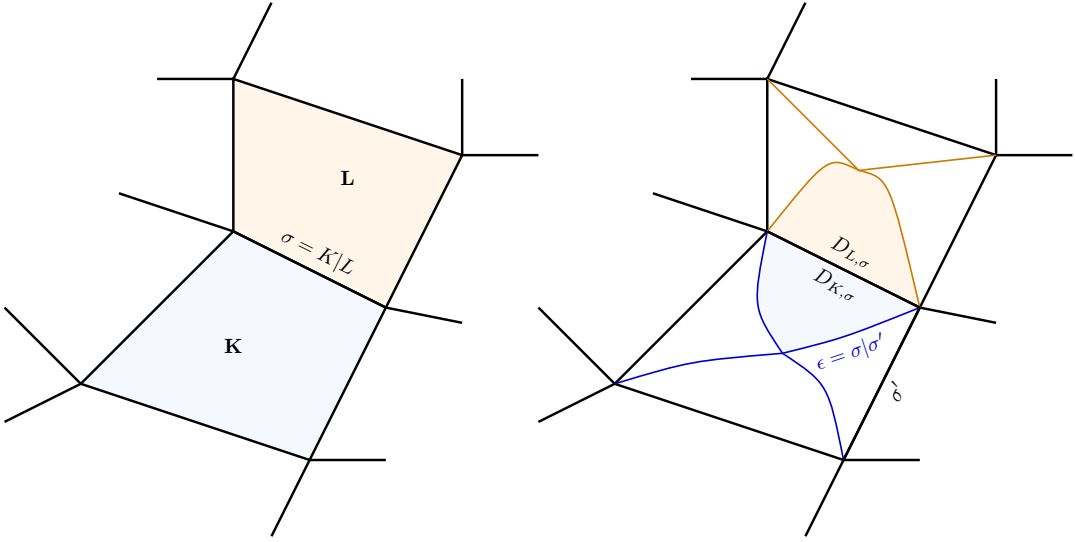
\begin{figure}
\begin{center}
\scalebox{0.78}{
\begin{tikzpicture}[scale=1.3]
\fill[color=bclair,opacity=0.3] (1.,2.) -- (4.,1.) -- (5.,3.) -- (3.,4.) -- (1.,2.);
\fill[color=orangec!30!white,opacity=0.3] (5.,3.) -- (6.,5.) -- (3.,6.) -- (3.,4.);
\draw[very thick, color=black] (1.,2.) -- (4.,1.) -- (5.,3.) -- (3.,4.) -- (1.,2.);
\draw[very thick, color=black] (5.,3.) -- (6.,5.) -- (3.,6.) -- (3.,4.);
\draw[very thick, color=black] (1.,2.) -- (0.,1.5); \draw[very thick, color=black] (1.,2.) -- (0.,3);
\draw[very thick, color=black] (4.,1.) -- (3.5,0.); \draw[very thick, color=black] (4.,1.) -- (5.,1.);
\draw[very thick, color=black] (5.,3.) -- (6.,2.8);
\draw[very thick, color=black] (6.,5.) -- (7.,5.);  \draw[very thick, color=black] (6.,5.) -- (6.,6.);
\draw[very thick, color=black] (3.,6.) -- (3.5,7.); \draw[very thick, color=black] (3.,6.) -- (2.,6.);
\draw[very thick, color=black] (3.,4.) -- (1.5,4.5);
\node at (3., 2.5){$\mathbf K$}; \node at (4.5, 4.7){$\mathbf L$};
\draw[very thick, color=black] (5.,3.) -- (3.,4.) node[midway,sloped,above]{$\edge=K|L$};
\fill[color=bclair,opacity=0.3] (10.,4.) .. controls (9.8,3.) .. (10.2,2.4) .. controls (11.,2.6) .. (12.,3.);
\fill[color=orangec!30!white,opacity=0.3] (10.,4.) .. controls (10.8,5.) .. (11.2,4.8) .. controls (11.6,4.7) .. (12.,3.);
\draw[very thick, color=black] (8.,2.) -- (11.,1.) -- (12.,3.) -- (10.,4.) -- (8.,2.);
\draw[very thick, color=black] (12.,3.) -- (13.,5.) -- (10.,6.) -- (10.,4.);
\draw[very thick, color=black] (8.,2.) -- (7.,1.5); \draw[very thick, color=black] (8.,2.) -- (7.,3);
\draw[very thick, color=black] (11.,1.) -- (10.5,0.); \draw[very thick, color=black] (11.,1.) -- (12.,1.);
\draw[very thick, color=black] (12.,3.) -- (13.,2.8);
\draw[very thick, color=black] (13.,5.) -- (14.,5.);  \draw[very thick, color=black] (13.,5.) -- (13.,6.);
\draw[very thick, color=black] (10.,6.) -- (10.5,7.); \draw[very thick, color=black] (10.,6.) -- (9.,6.);
\draw[very thick, color=black] (10.,4.) -- (8.5,4.5);
\draw[thick, color=bleuf] (10.,4.) .. controls (9.8,3.) .. (10.2,2.4);
\draw[thick, color=bleuf] (10.2,2.4) .. controls (11.,2.6) .. (12.,3.) node[midway,sloped,below]{$\edged=\edge|\edge'$};
\draw[thick, color=bleuf] (8.,2.) .. controls (9.,2.3) .. (10.2,2.4);
\draw[thick, color=bleuf] (11.,1.) .. controls (10.8,2.) .. (10.2,2.4);
\draw[thick, color=orangec!80!black] (10.,4.) .. controls (10.8,5.) .. (11.2,4.8) .. controls (11.6,4.7) .. (12.,3.);
\draw[thick, color=orangec!80!black] (13.,5.) -- (11.2,4.8);
\draw[thick, color=orangec!80!black] (10.,6.) -- (11.2,4.8);
\draw[very thick, color=black] (12.,3.) -- (10.,4.) node[midway,sloped,above]{$D_{L,\edge}$} node[midway,sloped,below]{$D_{K,\edge}$};
\draw[very thick, color=black] (11.,1.) -- (12.,3.) node[midway,sloped,below]{$\edge'$};
\end{tikzpicture}}
\caption{Primal and dual meshes for the Rannacher-Turek elements.}
\label{fig:mesh}
\end{center}
\end{figure}

\medskip
Finally, for the sake of simplicity, a constant time step denoted by $\delta t$ is used for the time discretization, with $\delta t=T/N$.
We define $t_n= n\, \delta t,\ 1\leq n \leq N$, and the notations for the discrete unknowns at step $n$ are obtained from the notations for space discretization introduced above by adding an index $n$, so, finally, the unknowns involved in the definition of the convection operator are $(\rho^n_K)_{K\in\mesh,\ 0\leq n \leq N}$ and $(\bfu_\edge^n)_{\edge \in \edge,\ 0\leq n \leq N }$.

%
%
\section{A second order convection operator}\label{sec:op}

Let us first address the discretization of the mass balance equation \eqref{eq:cont_mass_balance}. 
Since, in the Rannacher-Turek element, the pressure is piecewise constant over the cells, the natural mass balance (or, at least, for incompressible flows, the natural divergence-free constraint) takes a finite volume like formulation, posed over the primal cells.
With an explicit-in-time discretization of the convection flux, this equation thus reads, for $K \in \mesh$:
\[
\frac{|K|}{\delta t}(\rho_K^{n+1}-\rho_K^n)+ |K|\ \dive(\rho\bfu)_K^n=0,\quad \dive(\rho\bfu)_K^n =\frac{1}{|K|}\sum_{\edge \in \edges(K)} F^n_{K,\edge},
\]
where $F^n_{K,\edge}$ stands for the (primal) numerical mass flux across $\edge$ outward $K$ and is defined by:
\[
\forall \edge = K|L \in \edgesint, \quad F^n_{K,\edge}=|\edge|\ \rho_\edge^n \bfu_\edge^n \cdot \bfn_{K,\edge},
\]
with $\bfn_{K,\edge}$ the normal vector to the face $\edge$ outward $K$ and $\rho_\edge^n$ a discretization of the density at the face, which does not need to be specified in this section.
We suppose that the cell densities are positive at all time steps.
When the density is constant, we recover the usual divergence-free constraint for the Rannacher-Turek element.

\medskip
The dual mass fluxes and the face densities are constructed to ensure that a similar discrete mass balance holds over the dual cells, \ie\ to obtain a relation of the form:
\begin{equation} \label{eq:dual_mass_b}
    \forall \edge \in \edges,\qquad \frac{|D_\edge|}{\delta t}(\rho_{D_\edge}^{n+1}-\rho_{D_\edge}^n)+\sum_{\edged \in \edgesd(D_\edge)} F_{\edge,\edged}^n = 0,
\end{equation}
where $\rho^n_{D_\edge}$ is the density at the face $\edge$ and at time step $t_n$, and $F_{\edge,\edged}^n$ a mass flux through $\edged$ outward $D_\edge$.
For the internal faces, the face densities $\rho_{D_\edge}$ are defined as a weighted average of the density unknowns in the cells adjacent to $\edge$:
\begin{equation}
\label{eqdef:rho-Dsigma}
 \forall \edge \in \edgesint,\ \edge=K|L,\qquad |D_\edge|\ \rho_{D_\edge} = |D_{K,\edge}|\ \rho_K + |D_{L,\edge}|\ \rho_L. 
\end{equation}
For an external face $\edge$ of adjacent cell $K$, we just set $\rho_{D_\edge} = \rho_K$.
For $\edged$ included in the primal cell $K$ and $\edge$ a face of $K$, the mass fluxes $F_{\edge,\edged}$ are obtained by a linear combination of the mass fluxes through the primal faces of $K$.
A detailed explanation of the construction process is given in \cite{ans-11-sta} and extended in \cite{bru-22-sta} to more general, possibly hybrid 3D meshes.

\medskip
The mass balance \eqref{eq:dual_mass_b} over the dual meshes is then used for the definition of the discrete momentum convection term $C(\rho,u)_{\edge,i}^{n+1}$, \ie\ the discretization of the continuous term $\mathcal C(\rho,u_i)=\partial_t(\rho u_i) + \dive(\rho u_i \bfu)$. 
For $1 \leq i \leq d$ and $\edge \in \edges$, this discrete term takes the following form:
\begin{equation}\label{eqdef:discrete-convection}
    C(\rho,u)_{\edge,i}^{n+1} = \frac{1}{\delta t}(\rho_{D_\edge}^{n+1} u_{\edge,i}^{n+1}-\rho_{D_\edge}^n u_{\edge,i}^n) + \dive(\rho u_i \bfu)_\edge^n, \mbox{ with } \dive(\rho u_i \bfu)_\edge^n = \frac{1}{|D_\edge|} \sum_{\edged \in \edgesd(D_\edge)} F_{\edge,\edged}^n  u_{\edged,i}^n,
\end{equation}
where $u_{\edged,i}^n$ is an approximation of $u_i$ over the face $\edged$ at the time $t_n$.
For a boundary face $\edge \in \edgesext$, one of the dual faces of $D_\edge$ is the face $\edge$ itself.
If this primal/dual face is included in a part of the boundary where the velocity is prescribed, no equation is written for $\bfu_\edge^{n+1}$ (it is just set to the prescribed value) and no definition is needed for $u_{\edged,i}^n$; in the other case (\ie\ for a Neumann boundary condition), we suppose that the flow leaves the computational domain, and we set $u_{\edged,i}^n$ to the upwind value, \ie\ $u_{\edged,i}^n = u_{\edge,i}^n$.
For an internal dual face, $u_{\edged,i}^n$ is obtained by the algebraic MUSCL-like technique introduced in \cite{pia-13-for}, which implements the following procedure.
Let us recast the convection term $C(\rho,u)_{\edge,i}^{n+1}$ as
\begin{align*}
    C(\rho,u)_{\edge,i}^{n+1} = \frac{1}{\delta t} \rho_{D_\edge}^{n+1}\ \bigl(u_{\edge,i}^{n+1} - \bar u_{\edge,i}^{n+1}),
\end{align*}
with
\begin{align*}
    \bar u_{\edge,i}^{n+1} = \frac 1 {\rho_{D_\edge}^{n+1}}\ \bigl( \rho^n_{D_\edge} u^n_{\edge,i} - \delta t\ \dive(\rho u_i \bfu)_\edge^n \bigr)
= \frac 1 {\rho_{D_\edge}^{n+1}}\ \bigl( \rho^n_{D_\edge} u^n_{\edge,i} - \frac{\delta t}{|D_\edge|} \sum_{\edged \in \edgesd(D_\edge)} F_{\edge,\edged}^n  u_{\edged,i}^n \bigr).
\end{align*}
The discrete convection operator is said to be monotone if the term $\bar u_{\edge,i}^{n+1}$ can be written as a convex combination of degrees of freedom of $u^n_i$; for instance, such a property would ensure a discrete maximum principle for the transport equation, or a convection-diffusion equation with a suitable (only available on specific meshes) discretization of the diffusion term.
Let us recast $\bar u_{\edge,i}^{n+1}$ as
\begin{equation}\label{eq:def_ubar}
\bar u_{\edge,i}^{n+1} = \frac 1 {\rho_{D_\edge}^{n+1}}\ \Bigl[ \bigl( \rho^n_{D_\edge}- \frac{\delta t}{|D_\edge|} \sum_{\edged \in \edgesd(D_\edge)} F_{\edge,\edged}^n \bigr)\ u^n_{\edge,i}
- \frac{\delta t}{|D_\edge|} \sum_{\edged \in \edgesd(D_\edge)} F_{\edge,\edged}^n  (u_{\edged,i}^n-u^n_{\edge,i}) \Bigr].
\end{equation}
The mass balance equation \eqref{eq:dual_mass_b} yields
\[
\frac 1 {\rho_{D_\edge}^{n+1}}\ \bigl( \rho^n_{D_\edge}- \frac{\delta t}{|D_\edge|} \sum_{\edged \in \edgesd(D_\edge)} F_{\edge,\edged}^n) = 1,
\]
and therefore the sum of the coefficients multiplying the velocities $u^n_{\edge,i}$ and $u_{\edged,i}^n$ at the right-hand side of Relation \eqref{eq:def_ubar} is equal to $1$.
The coefficient of $u^n_{\edge,i}$ in \eqref{eq:def_ubar} is non-negative under the $\cfl$ condition
\begin{equation}
\cfl=\max_{\edge \in \edges} \Bigl\{ \frac{\delta t}{\rho_{D_\edge}^n\ |D_\edge|}\sum_{\edged \in \edgesd(D_\edge)}|F_{\edge,\edged}^n| \Bigr\}  \leq 1.
\end{equation}
and we indeed obtain a convex combination at the right-hand side of Equation \eqref{eq:def_ubar} if the following condition holds for each $\edged \in \edgesdint$ such as $\edged=D_\edge|D_{\edge'}$:
\begin{equation} \label{eq:MUSCLfacevalue}
\exists \alpha_\edged^\edge \in [0,1],\ \exists \Tilde{\edge} \in \edges \text{ such that} \quad
u_{\edged,i}^n-u_{\edge,i}^n = \left| \begin{array}{ll}
\alpha_\edged^\edge(u_{\edge,i}^n - u_{\Tilde{\edge},i}^n) & \text{if } F_{\edge,\edged}\geq0, \\[2ex]
\alpha_\edged^\edge(u_{\Tilde{\edge},i}^n -u_{\edge,i}^n)  & \text{otherwise}.
\end{array}\right.
\end{equation}
Of course, in this relation, both the coefficient $\alpha_\edged^\edge$ and the face $\Tilde{\edge}$ have to be determined at each time step.
We now deduce from the relation \eqref{eq:MUSCLfacevalue} a constructive process to compute the quantities $u_{\edged,i}^n$. 
Let $\edged$ be a given internal face, and let $D_{\edge^-}$ (resp. $D_{\edge^+}$) denote the adjacent upwind (resp. downwind) dual cell to the face $\edged$ (\ie\ $F_{\edge^-,\edged}\geq0$).
Let  $\neigh_\edged(D_{\edge^-})$  (resp. $\neigh_\edged(D_{\edge^+})$ be a set of neighbouring dual cells of $D_{\edge^-}$ (resp. $D_{\edge^+}$).
The following assumptions are then a transcription of Condition \eqref{eq:MUSCLfacevalue}: 
\begin{subequations}\label{interval}
\begin{align} & \label{H1}
\exists \ D_{\overline{\edge}} \in \neigh_\edged(D_{\edge^+}) \mbox{ such that } u_{\edged,i}^n \in I^+=\bigl[ u_{\overline{\edge},i}^n, u_{\overline{\edge},i}^n + \dfrac{\xi^+}{2}(u_{\edge^+,i}^n - u_{\overline{\edge},i}^n) \bigr],
\\ & \label{H2}
\exists \ D_{\overline{\edge}} \in \neigh_\edged(D_{\edge^-}) \mbox{ such that } u_{\edged,i}^n \in I^-=\bigl[ u_{\edge^-,i}^n, u_{\edge^-,i}^n + \dfrac{\xi^-}{2}(u_{\edge^-,i}^n - u_{\overline{\edge},i}^n) \bigr],
\end{align}
\end{subequations}
where $\xi^+$ and $\xi^-$ are two numerical parameters lying in the interval $[0,2]$.
These parameters have to be chosen by the user, and are usually kept constant through the whole computation; decreasing their value makes the algorithm limitation more restrictive.
The set $\neigh_\edged(D_{\edge^+})$ is always required to contain $D_{\edge^-}$, with the following two consequences: first, the value $u_{\edge^-,i}^n$ always belongs to both intervals $I^+$ and $I^-$, so their intersection is not void and the scheme is always defined; second, setting $\xi^+=\xi^-=0$ yields the usual upwind scheme.
To make the definition of the scheme complete, we now need to define the sets $\neigh_\edged(D_{\edge^+})$ and $\neigh_\edged(D_{\edge^-})$.
Here we choose $\neigh_\edged(D_{\edge^+})=\{D_{\edge^-}\}$, so that the condition \eqref{H1} implies that $u_{\edged,i}^n$ is a convex combination of $u_{\edge^-,i}^n$ and $u_{\edge^+,i}^n$.
Furthermore, if $\xi^+ \leq 1$, the hypothesis \eqref{H1} yields $u_{\edged,i,\mathrm M} \in [u_{\edged,i,\mathrm U},u_{\edged,i,\mathrm C}]$ where $u_{\edged,i,\mathrm U}$, $u_{\edged,i,\mathrm M}$ and $u_{\edged,i,\mathrm C}$ are the values given by the upwind, MUSCL and centered discretization respectively; the MUSCL discretization thus yields in this case a more diffusive scheme than the centered discretization and less diffusive than the upwind discretization, whatever the choice of $\xi+$ and $\xi^-$ in the $[0,2]$ interval.
Hence, in our numerical experiments, we choose to set  $\xi^+ \leq 1$, for energetic stability reasons; note also that the motivation for considering $\xi^+>1$ is generally to allow a second order interpolation of the unknown at the face, which here does not make sense since the dual mesh cannot be built explicitly.
Concerning $\neigh_\edged(D_{\edge^-})$, several choices are possible: 
\blist
\item a simple choice is to take the neighbouring cells of $D_{\edge^-}$:
\[
\neigh_\edged(D_{\edge^-})=\left\{(D_\edge)_{\edge \in \edges} \text{ such that }D_\edge\text{ shares a face } \tilde{\edged} \text{ with }D_{\edge^-}\right\};
\]
\item the previous set can be restricted to the upstream neighbouring cells of $D_{\edge^-}$:
\[
\neigh_\edged(D_{\edge^-})=\left\{(D_{\edge})_{\edge \in \edges} \text{ such that }D_\edge\text{ shares a face } \tilde{\edged} \text{ with }D_{\edge^-} \text{ and } F_{\edge,\tilde{\edged}}\geq0\right\};
\]
\item another possibility is to take the opposite cell to $D_{\edge^+}$ with respect to $D_{\edge^-}$, \ie 
\[
\neigh_\edged(D_{\edge^-})=\left\{(D_{\edge'})_{\edge' \in \edges} \text{ such that }D_{\edge'} \text{ shares a face } \edged' \text{ with }D_{\edge^-} \text{ and } \edged \cap \edged'=\varnothing \right\}.
\]
\elist
The last choice was selected in our numerical experiments, in the interior of the computational domain.
For dual edges with one of the adjacent cells itself adjacent to the boundary, depending on the sign of the mass fluxes, this choice may be impossible if the opposite cell does not exist; for a smooth flow, in such a case, one may expect that the fluid is entering the domain through the opposite dual face (the face denoted by $\edged'$ in the previous relation), and the value in the opposite cell may be replaced by the Dirichlet value.
Otherwise, the choice for $u_{\edge^-,i}^n$ boils down to the upwind choice.

\begin{figure}
\begin{center}
\scalebox{0.78}{
\begin{tikzpicture}[scale=1.3]
\fill[color=bclair,opacity=0.7] (3.,4.) -- (4.5,4.8) -- (5.,3.) -- (3.2,2.2) -- (3.,4.);
\fill[color=orangec!30!white,opacity=0.7] (5.,3.) -- (3.2,2.2) -- (4.,1.) -- (5.8,1.5) -- (5.,3.);
\fill[color=vertf!30!white,opacity=0.7] (4.,1.) -- (5.8,1.5) -- (6.5,0.) -- (5.,-0.8) -- (4.,1.);
\draw[very thick, color=black] (1.,2.) -- (4.,1.) -- (5.,3.) -- (3.,4.) -- (1.,2.);
\draw[very thick, color=black] (5.,3.) -- (6.,5.) -- (3.,6.) -- (3.,4.);
\draw[very thick, color=black] (4.,1.) -- (6.5,0) -- (7,2.5) -- (5,3);
\draw[very thick, color=black] (4.,1.) -- (6.5,0) -- (6.5,-2) -- (3.5,-1.2) -- (4.,1.);
\draw[very thick, color=black] (1.,2.) -- (0.,1.5); \draw[very thick, color=black] (1.,2.) -- (0.,3);
\draw[very thick, color=black] (6.,5.) -- (7.,5.);  \draw[very thick, color=black] (6.,5.) -- (6.,6.);
\draw[very thick, color=black] (3.,6.) -- (3.5,7.); \draw[very thick, color=black] (3.,6.) -- (2.,6.);
\draw[very thick, color=black] (7.,2.5) -- (7.5,2.4); \draw[very thick, color=black] (7.,2.5) -- (7.1,3.);
\draw[very thick, color=black] (6.5,-2.) -- (7.,-2.); \draw[very thick, color=black] (6.5,-2.) -- (6.5,-2.5);
\draw[very thick, color=black] (3.5,-1.2) -- (3.5,-1.7); \draw[very thick, color=black] (3.5,-1.2) -- (3.,-1.2);
\draw[very thick, color=black] (3.,4.) -- (1.5,4.5);
\draw[very thick, color=black] (6.5,0.) -- (7.,-0.05);
\draw[very thick, color=black] (5.,3.) -- (3.,4.) node[midway,sloped,above]{$\edge$};
\draw[very thick, color=black] (5.,3.) -- (4.,1.) node[midway,sloped,below]{$\edge'$};
\draw[very thick, color=black] (4.,1.) -- (6.5,0.) node[midway,sloped,above]{$\edge''$};
\draw[very thick, color=vertf!60!black] (3.2,2.2) -- (5.,3.) node[near start,sloped,below]{$\edged$};
\draw[very thick, ->, color=vertf!60!black] (4.3,2.15) -- (3.82,3.23) node[near end,sloped,above]{$F_{\edge',\edged}$};
\node at (3.4, 3.4){\textcolor{bleuf}{$\mathbf D_\edge$}};
\node at (4, 1.6){\textcolor{orangec!80!black}{$\mathbf D_{\edge'}$}};
\node at (5.3, 0){\textcolor{vertf!99!black}{$\mathbf D_{\edge''}$}};
\end{tikzpicture}}
\caption{Dual cells involved in the definition of the convection flux.}
\label{fig:flux}
\end{center}
\end{figure}
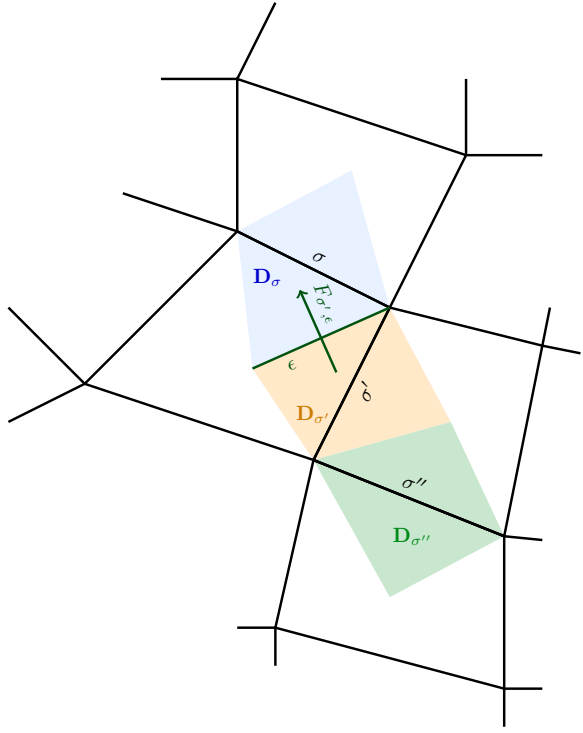

\medskip
We are now in a position to give the algorithm used to compute the quantities $u_{\edged,i}^n$:
\begin{list}{-}{\itemsep=0.5ex \topsep=0.5ex \leftmargin=1.5cm \labelwidth=0.6cm \labelsep=0.5cm \itemindent=0.cm}\item[$(i)$] Compute a tentative value $\overline{u}_{\edged,i}^n$ with a convex combination of the values (e.g. the centered choice) in the surrounding faces.
\item[$(ii)$] The flux $F_{\edge,\edged}^n$ being given, determine the upwind face $D_{\edge^-}$ and the downwind face $D_{\edge^+}$, and choose accordingly the neighbouring sets $\neigh_\edged(D_{\edge^-})$ and $\neigh_\edged(D_{\edge^+})$.
\item[$(iii)$] Compute an admissible interval $I^+ \cap I^-$ for $u_{\edged,i}$ by \eqref{interval}.
\item[$(iv)$] Compute $u_{\edged,i}^n$ by projecting the tentative value $\overline{u}_{\edged,i}^n$ into the interval obtained in the previous step.
\end{list}

\begin{rmrk}[Deriving an implicit MUSCL scheme]
Since this procedure is not linear, we cannot expect to derive an explicit formula to compute the values of the coefficients $a_\edged^\edge$. 
Their evaluation is, however, not necessary in order to define an explicit scheme: the presented algorithm univocally defines the value $u_{\edged,i}^n$. 
But, for this reason, we cannot easily define an implicit-in-time MUSCL scheme. 
However, it is still possible, using one of the following techniques:
\blist
\item a first technique would consist in an iterative process at each time step: in an inner loop, advance the velocity by replacing in the momentum equation the MUSCL convection operator at inner step $k$, $\dive_{\mathrm M}(\rho u_i \bfu)^k_\edge$, by $\dive_{\mathrm U}(\rho u_i \bfu)^{k+1}_\edge-\dive_{\mathrm U}(\rho u_i \bfu)^k_\edge+\dive_{\mathrm M}(\rho u_i \bfu)^k_\edge$, where the subscript $U$ denote the standard upwind convection operator ($\ie$ faces values $u_{\edged,i}$ are obtained through an upwind method) and the superscript $k+1$ indicate an implicit discretization, and then loop until acceptable convergence is reached;
\item an other technique would be to first compute the value $u^n_{\edged,i}$, then use \eqref{eq:MUSCLfacevalue} (or rather \eqref{interval}) to compute the coefficients $a_\edged^\edge$ thanks to $u^n_{\edged,i}$ and the $(u^n_\edge)_{\edge \in \edges}$. 
Then, express an implicit value at the interface $u^{n+1}_{\edged,i}$ as a linear combination of the $(u^{n+1}_\edge)_{\edge \in \edges}$ thanks to the $a_\edged^\edge$.
\elist
Note that both techniques are costlier from a computational point of view.
\end{rmrk}

%
%
\section{A discrete kinetic energy identity and some applications}

In this section, we first focus on the proposed higher-order finite volume convection operator and show that it satisfies an identity which may be seen as a building brick for the derivation of a kinetic energy balance (or, equivalently, an entropy identity for the entropy function $\eta(u_i)=\frac 1 2 u_i^2$).
We then give two applications of this result: first, we establish a stability property for a convection-diffusion problem, with an implicit discretization of the diffusion term, which may readily be extended to obtain stability estimates for incompressible or barotropic flows; second, we build a consistent scheme for the Euler equations based on a discrete solution of a (corrected) internal energy balance.
%
%
\subsection{A local identity for the discrete convection operator}\label{sec:kin}

In the continuous setting, let us assume that the mass balance equation  \eqref{eq:cont_mass_balance} holds.
Let $1\le i \le d$;  for sufficiently regular density and velocity functions, using twice the mass balance to switch from a convection to a transport operator for $u_i$ and then from a transport back to a convection operator for $u_i^2$, leads to:
\begin{equation} \label{eq:ki}
u_i\ \Bigl( \partial_t (\rho u_i) + \dive (\rho u_i \bfu) \Bigr) =
\rho u_i\ \Bigl( \partial_t u_i + \bfu \cdot \gradi u_i \Bigr) =
\frac 1 2 \rho \Bigl( \partial_t (u_i^2) + \bfu \cdot \gradi (u_i^2) \Bigr) =
\partial_t (\rho \frac{u_i^2} 2) + \dive (\rho \frac{u_i^2} 2 \bfu).
\end{equation}
Our aim here is to derive a discrete analogue of this identity.
For the sake of simplicity, we focus on the term $u_i  C(\rho,u)_{\edge,i}^{n+1}$ for the internal faces $\edge \in \edgesint$ of the mesh, where $C(\rho,u)_{\edge,i}^{n+1}$ is  the discrete convection operator defined by \eqref{eqdef:discrete-convection}.
We mimick the derivation of the identity \eqref{eq:ki} and therefore  recast the convection term as a transport one; in order to do so, 
we again suppose that the dual mass fluxes and the face densities are constructed to ensure that a discrete mass balance of the form \eqref{eq:dual_mass_b} holds over the dual cells.

\medskip
We are now in position to state a discrete analogue to Equation \eqref{eq:ki}, which does not feature a null right-hand side but a rest term.
This result can be seen as a direct consequence of \cite[Lemma A1]{her-18-cons}; for the sake of clarity, we reformulate it here in a way that is more convenient for the applications of this paper.

\begin{lemma}[Approximate transport operator for the kinetic energy]\label{lem:kin_en}
Assume that Equation \eqref{eq:dual_mass_b} holds.
Then, for $1 \leq i \leq d$, $\edge \in \edges$ and $0 \leq n \leq N-1$:
\begin{align*}
|D_\edge|u_i  C(\rho,u)_{\edge,i}^{n+1}=
\frac{|D_\edge|}{2\,\delta t}\,  \bigl( \rho^{n+1}_{D_\edge}\, (u_{i,\edge}^{n+1})^2 -\rho^n_{D_\edge}\, (u_{i,\edge}^n)^2 \bigr)
+ \frac 1 2 \sum_{\edged \in \edgesd(D_\edge)} F_{\edge,\edged}^n\,  (u_{i,\edged}^n)^2
+ \sum_{\edged \in \edgesd(D_\edge)} T_{\edge,\edged,i}^{n+1}
+ R_{\edge,i}^{n+1},
\end{align*} 
with
\begin{align} &
\label{eqdef-T}T_{\edge,\edged,i}^{n+1} = - \frac{1}{2} F_{\edge,\edged}^n\, (u_{i,\edged}^n-u_{i,\edge}^n)^2,
\\[1ex] &
R_{\edge,i}^{n+1} = 
\frac{ |D_\edge|}{2\,\delta t}\, \rho^{n+1}_{D_\edge}\ \bigl( u_{i,\edge}^{n+1}-u_{i,\edge}^n \bigr)^2
+ (u^{n+1}_{i,\edge}-u^n_{i,\edge})\sum_{\edged \in \edgesd(D_\edge)} F_{\edge,\edged}^n\, (u_{i,\edged}^n-u_{i,\edge}^n).
\end{align}
\end{lemma}

\begin{proof}
Let $\edge \in \edgesint$ and $0 \leq n < N-1$.
Subtracting the dual mass balance equation \eqref{eq:dual_mass_b} multiplied by $u_{i,\edge}^n$ yields:
\[
\frac 1 {\delta t}\ (\rho^{n+1}_{D_\edge} u_{i,\edge}^{n+1}-\rho^n_{D_\edge} u_{i,\edge}^n)+ \frac{1}{|D_\edge|}\sum_{\edged \in \edgesd(D_\edge)} F_{\edge,\edged}^n u_{i,\edged}^n =
\frac{\rho^{n+1}_{D_\edge}}{\delta t}\ (u_{i,\edge}^{n+1} - u_{i,\edge}^n)
+ \frac 1 {|D_\edge|}\sum_{\edged \in \edgesd(D_\edge)} F_{\edge,\edged}^n (u_{i,\edged}^n-u_{i,\edge}^n).
\]
The left-hand side of this relation is a discretization of the conservative form of the convection operator $\partial_t(\rho u_i) + \dive(\rho u_i \bfu)$, while the right-hand side may be seen as a discretization of the non-conservative form $\rho (\partial_t u_i + \bfu\cdot \gradi u_i)$.
We now multiply the right-hand side of the previous equality (which is precisely $C(\rho,u)_{\edge,i}^{n+1}$) by $|D_\edge|\ u^{n+1}_{i,\edge}$ and use (twice) the identity $2 a(a-b)=a^2-b^2+(a-b)^2$, once for the time derivative term and once for the "velocity gradient term", to obtain:
\begin{align*}
    |D_\edge|u_i  C(\rho,u)_{\edge,i}^{n+1} =
    &
    \frac{|D_\edge|}{2\,\delta t}\, \rho^{n+1}_{D_\edge}\ \bigl( (u_{i,\edge}^{n+1})^2-(u_{i,\edge}^n)^2 \bigr) + \frac{ |D_\edge|}{2\,\delta t}\, \rho^{n+1}_{D_\edge}\ \bigl( u_{i,\edge}^{n+1}-u_{i,\edge}^n \bigr)^2
    \\[1ex] &
    + u^n_{i,\edge} \sum_{\edged \in \edgesd(D_\edge)} F_{\edge,\edged}^n\ (u_{i,\edged}^n-u_{i,\edge}^n)
    + (u^{n+1}_{i,\edge}-u^n_{i,\edge}) \sum_{\edged \in \edgesd(D_\edge)} F_{\edge,\edged}^n\ (u_{i,\edged}^n-u_{i,\edge}^n)
    \\
    =
    &
    \frac{|D_\edge|}{2\,\delta t}\, \rho^{n+1}_{D_\edge}\ \bigl( (u_{i,\edge}^{n+1})^2-(u_{i,\edge}^n)^2 \bigr) + \frac{ |D_\edge|}{2\,\delta t}\, \rho^{n+1}_{D_\edge}\ \bigl( u_{i,\edge}^{n+1}-u_{i,\edge}^n \bigr)^2
    \\[1ex] &
    + \frac 1 2 \sum_{\edged \in \edgesd(D_\edge)} F_{\edge,\edged}^n  \bigl((u_{i,\edged}^n)^2 - (u_{i,\edge}^n)^2 \bigr)
    - \frac{1}{2}\sum_{\edged \in \edgesd(D_\edge)} F_{\edge,\edged}^n (u_{i,\edged}^n-u_{i,\edge}^n)^2 \\[1ex] &
    + (u^{n+1}_{i,\edge}-u^n_{i,\edge})\sum_{\edged \in \edgesd(D_\edge)} F_{\edge,\edged}^n (u_{i,\edged}^n-u_{i,\edge}^n).
\end{align*}
We now reverse the trick used previously to switch from the non-conservative form of the convection operator (this time for $\frac 12 u_i^2$) to the conservative form (which amounts to add this time Equation \eqref{eq:dual_mass_b} multiplied by $\frac 12 |D_\edge|\ (u_{i,\edge}^n)^2$).
This changes the first term of the first and second lines of the right-hand side, and yields the desired identity.
\end{proof}

In the previous lemma, the expression of the approximation of $u_i$ at the dual faces is not specified.
Let us then discuss the properties of the remainder term $T_{\edge,\edged,i}^{n+1}$ defined by \eqref{eqdef-T} for the specific choice of $u_i$  given by the MUSCL scheme introduced in the previous section.
For a dual face $\edged = D_\edge | D_{\edge'}$ for $\edge,\edge' \in \edges$, since the set $\neigh_\edged(D_{\edge^+})$ of neighbours of the dual cell $D_{\edge^+}$ is chosen as $\{D_{\edge^-}\}$, the condition \eqref{H1} yields:
\[
u_{i,\edged}^n= (1 - \frac {\xi_{i,\edged}^n} 2)\, u_{i,\edge^-}^n + \frac{\xi_{i,\edged}^n} 2 \,u_{i,\edge^+}^n,
\]
with $\xi_{i,\edged}^n \in [0,\xi^+]$, so $\xi_{i,\edged}^n \in [0,1]$ if we choose $\xi^+=1$, as in the numerical experiments of Section \ref{sec:num} below.
In this relation, we recall that $D_{\edge^-}$ (resp. $D_{\edge^+}$) is the upwind (resp. downwind) dual cell with respect to $\edged$, \ie\ the dual cell of $\{D_\edge,D_{\edge'}\}$ such that $F_{\edge^-,\edged}^n \geq 0$ (resp. $F_{\edge^+,\edged}^n \leq 0$).
Considering both possible signs of $F_{\edge,\edged}^n$, we obtain the following expression for $u_{i,\edged}^n$:
\[
u_{i,\edged}^n = \frac{u_{i,\edge}^n + u_{i,\edge'}^n} 2 + \frac 1 2 \mathrm{sgn}(F_{\edge,\edged}^n)\ (1-\xi_{i,\edged}^n)\ (u_{i,\edge}^n - u_{i,\edge'}^n).
\]
We recover a classical presentation of the convection scheme as a centered scheme possibly corrected by a diffusion term; indeed, $\xi_{i,\edged}^n=1$ indeed corresponds to the centered scheme, while  $F_{\edge,\edged}^n \mathrm{sgn}(F_{\edge,\edged}^n)\ (1-\xi_{i,\edged}^n) \geq 0$, so that the second term can be seen as a numerical diffusion term.
With this expression of $u_{i,\edged}^n$, the term $T_{\edge,\edged,i}^{n+1}$ reads:
\begin{equation}\label{eq:T}
T_{\edge,\edged,i}^{n+1} = - \frac 1 8\, \bigl(1 + (1-\xi_{i,\edged}^n)^2\bigr)\, F_{\edge,\edged}^n\, (u_{i,\edge}^n-u_{i,\edge'}^n)^2 + \frac 1 4\, (1-\xi_{i,\edged}^n)\ |F_{\edge,\edged}^n|\, (u_{i,\edge}^n-u_{i,\edge'}^n)^2.
\end{equation}
Thanks to the conservativity of the dual mass fluxes, the first part of the right-hand side is also conservative; the second part may be seen as a numerical dissipation.
%
%
\subsection{A stability result}\label{sec:stab}

Suppose, for the sake of simplicity, that a convection-diffusion equation of the form:
\begin{equation} \label{eq:conv_diff}
\partial_t (\rho u_i) + \dive (\rho u_i \bfu)- \mu\lapi u_i = 0,
\end{equation}
holds for the $i$-th component of the velocity, where $\mu$ is a positive parameter.
This equation can be seen as a momentum balance equation with no source term and without the pressure gradient term.
The diffusion term may arise either from a physical fluid viscosity or from a numerical stabilisation term.
Assuming that a mass balance equation holds, multiplying Equation \eqref{eq:conv_diff} by $u_i$ yields, by the same computation for the convection term as in the previous section:
\begin{equation} \label{eq:pre-ec}
\frac 1 2\, \partial_t(\rho u_i^2) + \frac 1 2 \dive(\rho u_i^2 \bfu)-\mu\, \dive (u_i \gradi u_i) + \mu\ \Vert \gradi u_i \Vert^2 = 0.
\end{equation}
Now suppose that the velocity is prescribed to zero on $\partial \Omega$. 
Integrating the previous formula over the domain $\Omega$, then using the divergence theorem for the convection term and Green's identity for the diffusion term yields:
\begin{equation} \label{eq:pre-global_ec}
\frac{1}{2}\int_\Omega \partial_t(\rho u_i^2) \dx + \mu \int_\Omega \Vert \gradi u_i \Vert^2 \dx = 0.
\end{equation}
Integrating in time, this equality yields a control of $\rho^{1/2} u_i$ in the $L^\infty(0,T; L^2(\Omega))$ norm and of $\mu^{1/2} u_i$ in the $L^2(0,T;H^1(\Omega))$ norm.
In addition, we remark that, for $\varphi \in C^\infty_c(\Omega \times[0,T))$,
\begin{multline} \hspace{10ex}
\int_0^T \int_\Omega \mu\, \dive (u_i \gradi u_i)\ \varphi \dx \dt = - \int_0^T \int_\Omega \mu\, u_i \gradi u_i \cdot \gradi \varphi \dx \dt
\\
\leq \mu^{1/2}\ \Vert u_i \Vert_{L^2(\Omega\times(0,T))}\ \Vert \mu^{1/2} u_i \Vert_{L^2(0,T;H^1(\Omega))} \ \Vert \gradi \varphi \Vert_{L^\infty(\Omega\times(0,T))}.
\hspace{10ex} \end{multline}
If we consider a sequence of solutions to Equation \eqref{eq:conv_diff} obtained with a sequence of vanishing viscosities, provided that $\rho$ is bounded by below by a positive real number (so $u_i$ is controlled in $L^2$), this integral thus tends to zero, and Equation \eqref{eq:pre-ec} may be used to obtain an entropy inequality, that is
\[
\frac 1 2\, \partial_t(\rho u_i^2) + \frac 1 2 \dive(\rho u_i^2 \bfu) \leq 0,
\]
in the distributional sense.
Dealing with the real momentum balance equation requires coping with a pressure gradient, which is standard for incompressible and barotropic flows.
In both cases, the estimate of $\gradi p \cdot \bfu$ is obtained thanks to the mass balance equation and the equation of state.
The simplest situation is the incompressible case, where:
\[
\gradi p \cdot \bfu = \dive(p\,\bfu) - p\,\dive \bfu = \dive(p\,\bfu),
\]
so this term yields an entropy flux, and its integral over the computational domain vanishes thanks to the boundary conditions.
The quantity $\frac 1 2 \rho |\bfu|^2$ is now referred to as the kinetic energy balance and the analogues of Equations \eqref{eq:pre-ec} and \eqref{eq:pre-global_ec} as the local and global, respectively, kinetic energy balances.

\medskip
Our goal here is to demonstrate a similar result for higher-order finite volume convection operators, taking the form introduced in the previous section.
It is well-known that such an operator is not $L^2$-stable (while the first-order upwind discretization is, under a $\cfl$ condition), but we show here that the $L^2$-stability is recovered when a non-vanishing diffusion is added, for small enough time steps.
As in the continuous setting in the above introduction, we restrict ourselves to the discretization of the convection-diffusion problem for a component of the velocity, in such a way that the proposed analysis may be used as a building brick for the study of staggered schemes for both incompressible and compressible flows.
We suppose homogeneous Dirichlet boundary conditions on the whole boundary, so the velocity is set to zero on external faces, and the scheme reads, for a given index $i$, $1 \leq i \leq d$:
\begin{equation} \label{eq:conv-diff_scheme}
\frac 1 {\delta t}\, (\rho^{n+1}_{D_\edge} u_{i,\edge}^{n+1}-\rho^n_{D_\edge} u_{i,\edge}^n) + \dive(\rho u_i \bfu)^n_\edge - (\mu \lapi u_i)^{n+1}_\edge=0, \quad \forall \edge \in \edgesint.
\end{equation}
The discrete mass balance equation \eqref{eq:dual_mass_b} over the dual cells is supposed to hold.
The discretization of the diffusion term is implicit in time and does not need to be precisely defined at this point.
We only need to suppose that the following inequality holds:
\begin{equation} \label{eq:centeredstabilitycondition}
- \sum_{\edge \in \edgesint}|D_\edge|\ u^{n+1}_{i,\edge} (\mu \lapi u^{n+1}_i)_\edge \geq  
\sum_{\substack{\edged \in \edgesdint,\\ \edged=D_\edge|D_{\edge'} \\
\edged \subset K }} \mu_\edged^n\,h_K^{d-2}\,(u^{n+1}_{i,\edge}-u^{n+1}_{i,\edge'})^2.
\end{equation}
This relation might be seen as a discrete analogue to the inequality $-\int_\Omega u_i \mu \lapi u_i \dx \geq \int_\Omega \mu \gradi u_i \cdot \gradi u_i \dx$ (recall that we have supposed homogeneous Dirichlet boundary conditions).
The viscosity $\mu_\edged^n$ is supposed to be positive (and therefore, at least for a given discretization, bounded away from zero), and the right-hand side of Inequality \eqref{eq:centeredstabilitycondition} defines a discrete $H^1$ semi-norm (precisely speaking, is equal to the square of a $H^1$ semi-norm), which we denote $|u_i|_\edges$.
If the diffusion operator is given by the Rannacher-Turek finite element, this bound might be obtained thanks to the equivalence between the $|\cdot|_\edges$ norm and the broken $H^1$ semi-norm, which holds under regularity assumptions for the cells.

\medskip
The following result is a global (\ie\ integrated over the computational domain) estimate, which may be seen as a discrete analogue of Equation \eqref{eq:pre-global_ec}.

\begin{theorem}[Stability for a convection-diffusion equation] \label{th:kin}
Assume that Equation \eqref{eq:dual_mass_b} holds, that $\rho^n_{D_\edge} \geq  0$ for $\edge \in \edgesint$ and $0\leq n \leq N-1$, and that the coercivity condition \eqref{eq:centeredstabilitycondition} for the diffusion term holds.
Suppose that the time step satisfies the following set of inequalities:
\begin{multline} \label{eq:dt_stab} \hspace{7ex}
\eta^n =\frac{\delta t}{\tau^n} \leq 1 \mbox{ for } 0 \leq n \leq N-1, \mbox{ with }
\\
\tau^n = \min \Bigl\{
\frac{2^{1-d} \ h_K^{d-2} \ \mu_\edged^n}   { (F_{\edge,\edged}^n)^2\ \bigl( \dfrac 1 {|D_\edge|\ \rho^{n+1}_{D_\edge}} + \dfrac 1 {|D_{\edge'}|\ \rho^{n+1}_{D_{\edge'}}} \bigr)},
\ \edged \in \edgesdint,\ \edged=D_\edge|D_{\edge'},\ \edged \subset K
\Bigr\}. 
\hspace{7ex} \end{multline}
Then the scheme \eqref{eq:conv-diff_scheme}, using the proposed MUSCL scheme with $\xi^+=1$, is stable in the $L^2$-norm, in the sense that its solution satisfies the following inequality:
\begin{align}\label{eq:stab_cent}
\frac 1 2\, \sum_{\edge \in \edgesint} |D_\edge| \ \bigl(\rho^{n+1}_{D_\edge}(u^{n+1}_\edge)^2 - \rho^0_{D_\edge}(u^{0}_\edge)^2 \bigr) \leq \eta^0\ \delta t\ |u^0|_\edges^2.
\end{align}
Note that the right-hand side depends only on the initial conditions for the velocity, for the density, and the density at the end of the first time step
\end{theorem}

\begin{rmrk}[Evaluation of $\tau^n$]
The dual mass fluxes are obtained as a linear combination, with bounded coefficients, of the primal mass fluxes, see \cite{ans-11-sta}.
More specifically, for $\edged$ included in the cell $K$ and $\edge$ a face of $K$,  
\[
F_{\edge,\edged}^n = \sum_{\edge' \in \edges(K)} \alpha_K^{\edged,\edge'} F_{K,\edge'}^n,
\quad \sum_{\edge' \in \edges(K)} |\alpha_K^{\edged,\edge'}| = \alpha \mbox{ with } \alpha =2^{2-d},
\]
where $F_{K,\edge'}^n=|\edge'|\ \rho_{\edge'}^n\, \bfu_{\edge'}^n \cdot \bfn_{K,\edge'}$.
For the sake of simplicity, let us suppose that the density is equal to a constant value, which we denote by $\rho$, and that the velocity is bounded by a quantity $u_{\max}$, which yields $|F_{\edge,\edged}^n| \leq 2^{2-d} |\edge| \ \rho \ u_{\max}$, for $\edged$ included in the cell $K$ and $\edge$ a face of $K$.
Using $|D_\edge|> |D_{K,\edge}| = |K|/(2d)$ and $|\edge| < h_K^{d-1},\ \edge \in \edges(K)$, we get
\[
\tau^n \geq \frac{2^{d-5}} d\ \frac{|K|}{h_K^d}\ \min \bigl\{ \frac{\mu_\edged^n}{u_{\max}^2 \rho }, \ \edged \in \edgesdint \bigr\},
\]
which shows that $\tau^n$ only depends on the viscosity, the density, the velocity and the regularity of the mesh but not on the space step.
\end{rmrk}
\begin{proof}
Let $0 \leq n \leq N-1$ and $1 \leq i \leq d$.
Summing the result of the previous lemma over $\edge \in \edgesint$ and using inequality \eqref{eq:centeredstabilitycondition} yields
\[
\frac 1 {2\, \delta t} \sum_{\edge \in \edgesint} |D_\edge|\ \bigl( \rho^{n+1}_{D_\edge} (u_{i,\edge}^{n+1})^2 - \rho^n_{D_\edge} (u_{i,\edge}^n)^2 \bigr) \leq - \mathcal R_1 - \mathcal R_2 - \mathcal C_1 - \mathcal C_2 - \mathcal D,
\]
where the terms on the right-hand side are defined by
\begin{align*} 
    & \mathcal R_1 = \frac 1 {2\, \delta t} \sum_{\edge \in \edgesint} |D_\edge|\ \rho^{n+1}_{D_\edge}\, ( u_{i,\edge}^{n+1}-u_{i,\edge}^n)^2, 
    \\
    & \mathcal R_2 = \sum_{\edge \in \edgesint}\ (u^{n+1}_{i,\edge}-u^n_{i,\edge}) \sum_{\substack{\edged \in \edgesd(D_\edge),\\ \edged=D_\edge|D_{\edge'}}} F_{\edge,\edged}^n (u_{i,\edged}^n-u_{i,\edge}^n),
    \\ 
    & \mathcal C_1 = \frac 1 2 \sum_{\edge \in \edgesint} \sum_{\edged \in \edgesd(D_\edge)} F_{\edge,\edged}^n\,  (u_{i,\edged}^n)^2, 
    \\
    & \mathcal C_2 = -\frac 1 2 \sum_{\edge \in \edgesint}\ \sum_{\substack{\edged \in \edgesd(D_\edge),\\ \edged=D_\edge|D_{\edge'}}} T_{\edge,\edged,i}^{n+1},
    \\ 
    & \mathcal D = \sum_{\substack{\edged \in \edgesdint,\\ \edged=D_\edge|D_{\edge'},\\ \edged \subset K}} \mu_\edged^n\ h_K^{d-2}\ (u^{n+1}_{i,\edge}-u^{n+1}_{i,\edge'})^2.
\end{align*}
By conservativity, the sums $\mathcal C_1$ vanishes and, using the expression \eqref{eq:T} of $T_{\edge,\edged,i}^{n+1}$, the sum $\mathcal C_2$ is non-negative.
Let us now turn to the term $\mathcal R_2$.
By assumption on the convection scheme, we have $|u_{i,\edged}^n-u_{i,\edge}^n| \leq |u_{i,\edge'}^n-u_{i,\edge}^n|$.
Using the inequality $\displaystyle ab \leq \frac{a^2}{2 \varepsilon}+\frac{\varepsilon b^2}{2}$ for two real numbers $a$ and $b$ and $\varepsilon>0$ yields, with $\displaystyle \varepsilon=\frac{\delta t}{|D_\edge|\ \rho^{n+1}_{D_\edge}}$:
\[
\Bigl| (u^{n+1}_{i,\edge}-u^n_{i,\edge})\ \sum_{\substack{\edged \in \edgesd(D_\edge),\\ \edged=D_\edge|D_{\edge'}}} F_{\edge,\edged}^n \,(u^n_{i,\edge'}-u^n_{i,\edge}) \Bigr| \leq
\frac{|D_\edge|}{2\,\delta t}\, \rho^{n+1}_{D_\edge}\,(u^{n+1}_{i,\edge}-u^n_{i,\edge})^2
+\frac{\delta t}{2\ |D_\edge|\ \rho^{n+1}_{D_\edge}} \Bigl( \sum_{\substack{\edged \in \edgesd(D_\edge),\\ \edged=D_\edge|D_{\edge'}}} |F_{\edge,\edged}^n|\ |u^n_{i,\edge'}-u^n_{i,\edge}| \Bigr)^2.
\]
The sum of the first term over $\edge \in \edges$ is equal to $\mathcal R_1$, whereas using the Cauchy-Schwarz inequality $\left(\sum_{i=0}^nx_i\right)^2\leq n \sum_{i=0}^n(x_i)^2$ in the second term, with $n$ the number of the faces of a dual cell which is equal to $4$ if $d=2$ and $8$ if $d=3$ and thus may be written $2^d$, yields for $\mathcal R_2$:
\begin{multline*} \hspace{10ex}
- \mathcal R_2
\leq \mathcal R_1 + \sum_{\edge \in \edgesint} \frac{2^{d-1}\ \delta t}{|D_\edge|\ \rho^{n+1}_{D_\edge}} \sum_{\substack{\edged \in \edgesd(D_\edge),\\ \edged=D_\edge|D_{\edge'}}} \bigl( F_{\edge,\edged}^n(u^n_{i,\edge'}-u^n_{i,\edge}) \bigr)^2
\\
= \mathcal R_1 + \sum_{\substack{\edged \in \edgesdint,\\ \edged=D_\edge|D_{\edge'}}}  2^{d-1} \ \delta t
\ \bigl( \frac 1 {|D_\edge|\ \rho^{n+1}_{D_\edge}} + \frac 1 {|D_{\edge'}|\ \rho^{n+1}_{D_{\edge'}}} \bigr)\ (F_{\edge,\edged}^n)^2\ (u^n_{i,\edge'}-u^n_{i,\edge} )^2.
\hspace{10ex} \end{multline*}
Gathering all the previous information leads to:
\begin{multline*}
\frac 1 {2\, \delta t} \sum_{\edge \in \edgesint} |D_\edge|\ \bigl( \rho^{n+1}_{D_\edge} (u_{i,\edge}^{n+1})^2 - \rho^n_{D_\edge} (u_{i,\edge}^n)^2 \bigr) \leq
\\
\sum_{\substack{\edged \in \edgesdint,\\ \edged=D_\edge|D_{\edge'}}} 2^{d-1} \ \delta t
\ \bigl( \frac 1 {|D_\edge|\ \rho^{n+1}_{D_\edge}} + \frac 1 {|D_{\edge'}|\ \rho^{n+1}_{D_{\edge'}}} \bigr)\ (F_{\edge,\edged}^n)^2\ (u^n_{i,\edge'}-u^n_{i,\edge} )^2
- \sum_{\substack{\edged \in \edgesdint,\\ \edged=D_\edge|D_{\edge'}, \\ \edged \subset K}} \mu_\edged^n\ h_K^{d-2}\ (u^{n+1}_{i,\edge}-u^{n+1}_{i,\edge'})^2.
\end{multline*}
Summing this inequality over all time steps $t_k$ with $0\leq k \leq n$, we get:
\[
\frac 1 {2\, \delta t}\ \sum_{\edge \in \edgesint} |D_\edge|\ \bigl( \rho^{n+1}_{D_\edge} (u_{i,\edge}^{n+1})^2 - \rho^0_{D_\edge} (u_{i,\edge}^{0})^2 \bigr) \leq - \mathcal T^{n+1} + \mathcal S^n + \mathcal T^0,
\]
with
\begin{align*} &
\mathcal T^{n+1}=\sum_{\substack{\edged \in \edgesdint,\\ \edged=D_\edge|D_{\edge'},\\ \edged \subset K}} \mu_\edged^n h_K^{d-2} (u^{n+1}_{i,\edge}-u^{n+1}_{i,\edge'})^2,
\\ &
\mathcal S^n=\sum_{k=1}^n 
\sum_{\substack{\edged \in \edgesdint,\\ \edged=D_\edge|D_{\edge'},\\ \edged \subset K}} \Bigl( 2^{d-1}\ \delta t
\ \bigl( \frac 1 {|D_\edge|\ \rho^{k+1}_{D_\edge}} + \frac 1 {|D_{\edge'}|\ \rho^{k+1}_{D_{\edge'}}} \bigr)\ (F_{\edge,\edged}^k)^2 -\mu_\edged^k h_K^{d-2}\Bigr)\ (u^k_{i,\edge'}-u^k_{i,\edge} )^2,
\\ &
\mathcal T^0=\sum_{\substack{\edged \in \edgesdint,\\ \edged=D_\edge|D_{\edge'}}} 2^{d-1}\ \delta t
\ \bigl( \frac 1 {|D_\edge|\ \rho^1_{D_\edge}} + \frac 1 {|D_{\edge'}|\ \rho^1_{D_{\edge'}}} \bigr)\ (F_{\edge,\edged}^0)^2\ (u^0_{i,\edge'}-u^0_{i,\edge} )^2.
\end{align*}
The term $\mathcal T^{n+1}$ is obviously positive, the sum $\mathcal S^n$ is negative thanks to the assumption on the time step, and $\mathcal T_0 \leq \eta^0\ |u^0_i|_\edges^2$.
\end{proof}

\begin{rmrk}[Extension of this result to less-limited MUSCL schemes]
In the present case, we have seen that, since no geometrical interpolation for the velocity at the dual faces is possible, the choice $\xi^+=1$ is reasonable.
However, a stability result may still be obtained if, for some reason  only the condition $\xi \leq 2$ (\ie\ $\xi^+=2$) was imposed; in this case, the term $-\mathcal C_2$ is no longer positive, but satisfies
\[
-\mathcal C_2 \leq  \frac 1 2 \sum_{\substack{\edged \in \edgesdint,\\ \edged=D_\edge|D_{\edge'}}} |F_{\edge,\edged}^n|\, (u_{i,\edge'}^n-u_{i,\edge}^n)^2.
\]
To obtain a stability estimate, we need to absorb this term in $\mathcal D$, to obtain (indexing now the terms $\mathcal C_2$ and $\mathcal D$ with respect to time)
\[
-\mathcal C_2^{n+1} - \mathcal D^n \leq - \sum_{\substack{\edged \in \edgesdint,\\ \edged=D_\edge|D_{\edge'},\\ \edged \subset K}} (\mu')_\edged^n\ h_K^{d-2}\ (u^{n+1}_{i,\edge}-u^{n+1}_{i,\edge'})^2,
\]
with
\[
(\mu')_\edged^n\ h_K^{d-2} = \mu_\edged^n\ h_K^{d-2} - \frac 1 2 |F_{\edge,\edged}^{n+1}|,
\]
and to suppose that $(\mu')_\edged^n$ is bounded by below away from zero.
Note that, since $F_{\edge,\edged}^n$ is proportional to the measure of the faces, this assumption is satisfied when the space step is small enough.
The stability condition \eqref{eq:dt_stab} is then rephrased, switching $\mu_\edged^n$ to $(\mu')_\edged^n$.
In addition, the quantity $-C_2^0$ (which only depends on the initial condition) must now be added to the right-hand side of the stability inequality \eqref{eq:stab_cent}; this term may be recast as the $H^1$ semi-norm $|u^0|_\edges^2$ multiplied by a factor proportional to the space and time steps product.
\end{rmrk}
%
%
\subsection{A consistent "internal-energy-based" staggered scheme for the full Euler equations}\label{sec:corr_Euler}

For shock solutions of the Euler equations, only the total energy equation makes sense, because of its conservative character.
This relation reads:
\begin{equation}
\partial_t (\rho\, E) + \dive(\rho \, E \, \bfu) + \dive ( p \, \bfu )=0,
\end{equation}
where $E=\frac 1 2|\bfu|^2+e$, with $e$ the internal energy.
Formally, this equation may be seen as the sum of the kinetic balance:
\[
\partial_t (\rho E_k) + \dive \bigl(\rho\, E_k\, \bfu \bigr) + \gradi p \cdot \bfu = 0,\quad E_k=\frac 1 2|\bfu|^2,
\]
and the internal energy balance:
\begin{equation}\label{eq:e}
\partial_t (\rho e) + \dive(\rho e \bfu)+ p\, \dive \bfu =0.
\end{equation}
Solving this latter equation is appealing since a suitable discretization (both for the convection operator, with a maximum-principle-preserving approximation, and for the term $p\, \dive \bfu$, to take benefit of the fact that $p$ vanishes when $e$ vanishes) leads to a conservation of the positivity of the internal energy; combining this approach with a discretization of the mass balance equation which preserves the positivity of the density, we thus would obtain a scheme which preserves the convex of admissible states ($\rho \geq 0$, $e \geq 0$ and, thanks to the equation of state, $p \geq 0$), which is a non-trivial task (see e.g. \cite{cal-13-pos} and references herein). 
Note also that the total energy is a function of unknowns discretized on both the primal and the dual meshes, and discretizing only the internal energy balance allows circumventing the technical difficulty of building an approximation of such a "composite" unknown.
However, it may be anticipated (and is observed in practice) that a blunt discretization of Equation \eqref{eq:e} would yield a non-consistent scheme, giving solutions that do not respect the Rankine-Hugoniot jump conditions at shocks.
The problem stems from the fact that the discrete kinetic energy balance equation features remainder terms which may be seen as a dissipation associated with numerical diffusion, and which do not tend to zero when the time and space step tend to zero, but to measures borne by the shocks.
The technique initially proposed in \cite{her-14-cons} to solve this problem is to compensate these remainder terms in the internal energy balance, in the following sense.
Let us denote these terms  by $(\mathcal R_\edge^{n+1})_{\edge\in\edges,\ 0\leq n<N}$ and $\mathcal R$, $\Omega \times (0,T) \rightarrow \xR$ be the function defined by
\[
\mathcal R(\bfx,t) = \mathcal R_\edge^{n+1} \quad \mbox{for } \bfx \in D_\edge \mbox{ and } t \in (t_n,t_{n+1}).
\]
The corrective terms in the internal energy balance are denoted by $(\mathcal S_K^{n+1})_{K \in \mesh,\ 0\leq n<N}$, associated with a function
\[
\mathcal S(\bfx,t) = \mathcal S_K^{n+1} \quad \mbox{for } \bfx \in K \mbox{ and } t \in (t_n,t_{n+1}),
\]
and required to be such that the difference $\mathcal S - \mathcal R$ tends to zero in the distributional sense when the space and time steps tend to zero.
The consistency analysis may be found in \cite{her-21-cons}, and semi-explicit or explicit-in-time variants of the scheme may be found in \cite{her-14-cons,her-18-cons,gas-18-mus}.
In all these works, the discrete kinetic energy balance is obtained from a first-order upwind discretization of the convection operator in the momentum balance; we generalize this construction, here.

\medskip
From the consistency analysis \cite{her-21-cons}, it appears that only non-conservative terms have to be kept in the remainder of the discrete kinetic energy balance, the conservative terms being possibly disregarded or not (they vanish in the limit of space and time steps tending to zero).
From Lemma \ref{lem:kin_en}, it thus appears that a candidate for $\mathcal R_\edge^{n+1}$ is obtained by adding to $R_\edge^{n+1}$ the non-conservative part of $\sum_{\edged \in \edgesd(D_\edge)} T_{\edge,\edged,i}^{n+1}$, and summing over the component index:
\begin{multline*}
\mathcal R_\edge^{n+1} =
\frac{ |D_\edge|}{2\,\delta t}\, \rho^{n+1}_{D_\edge}\ \bigl| \bfu_\edge^{n+1}-\bfu_\edge^n \bigr|^2
+ \sum_{i=1}^d (u^{n+1}_{i,\edge}-u^n_{i,\edge})\sum_{\edged \in \edgesd(D_\edge)} F_{\edge,\edged}^n\, (u_{i,\edged}^n-u_{i,\edge}^n) \\
+ \sum_{i=1}^d \sum_{\edged \in \edgesd(D_\edge)} \frac 1 4\, (1-\xi_{i,\edged}^n)\ |F_{\edge,\edged}^n|\, (u_{i,\edge}^n-u_{i,\edge}^n)^2,
\end{multline*}
where $\rho^{n+1}_{D_\edge}$ is a weighted average of the density in the neighbouring cells, defined by \eqref{eqdef:rho-Dsigma}.
For $\edge \in \edgesint$, $\edge=K|L$, the terms of $\mathcal R_\edge^{n+1}$ are distributed to $K$ and $L$ to obtain $\mathcal S_K^{n+1} = \mathcal S_{K,1}^{n+1} + \mathcal S_{K,2}^{n+1} + \mathcal S_{K,3}^{n+1}$ with:
\begin{align*} &
\mathcal S_{K,1}^{n+1} = \frac{\rho_K}{2\ \delta t} \sum_{\edge \in \edges(K)} |D_{K,\edge}|\ \bigl| \bfu_\edge^{n+1}-\bfu_\edge^n \bigr|^2,
\\ &
\mathcal S_{K,2}^{n+1} = \sum_{i=1}^d\ \sum_{\edge \in \edges(K)} (u^{n+1}_{i,\edge}-u^n_{i,\edge})\sum_{\edged \in \edgesd(D_\edge),\ \edged \subset K} F_{\edge,\edged}^n\, (u_{i,\edged}^n-u_{i,\edge}^n),
\\ &
\mathcal S_{K,3}^{n+1} = \sum_{i=1}^d\ \sum_{\edged \subset K,\ \edged=\edge|\edge'} \frac 1 2\, (1-\xi_{i,\edged}^n)\ |F_{\edge,\edged}^n|\, (u_{i,\edge}^n-u_{i,\edge'}^n)^2.
\end{align*}
%
\section{Numerical tests}\label{sec:num}

The discretization of the convection operator presented in the above paragraphs was implemented in the open-source CALIF$^3$S software developed at IRSN \cite{califs}.
We now present the results obtained with CALIF$^3$S, namely a comparison between the upwind, centered, and MUSCL choices, for several classical tests of the literature for incompressible, barotropic, and compressible flows.

%
%
\subsection{Incompressible Navier-Stokes equation}

We first turn to the incompressible Navier-Stokes equations, which read, on a domain $\Omega$:
\begin{subequations} \label{eq:ns_incomp}
\begin{align} &
\partial_t (\rho u_i) + \dive(\rho u_i \bfu) + \partial_i p - \dive(\mu(\gradi \bfu+\gradi \bfu^t))_i= 0,  \quad 1 \leq i \leq d,
\\ &
\dive(\bfu) = 0.
\end{align}
\end{subequations}
Here, we suppose that the density is constant, and we set $\rho = 1$ for the sake of simplicity. 
These equations must be supplemented by initial conditions for the velocity and suitable (especially for stability) boundary conditions, which are specified in the presentation of each of the tests, below.
%
%
\subsubsection{The scheme}

This system is solved using a projection scheme (see \cite{gue-06-ano} for an overview), which consists in the two following steps:
\begin{subequations} \label{eq:inc_sch}
\begin{flalign} \nonumber &
\mbox{\textbf{Prediction step} -- Solve for $\Tilde{\bfu}^{n+1}$: }
\\ & \qquad
\begin{multlined}[b][10.5cm] 
\text{ For } 1 \leq i \leq d,\ \forall \edge \in \edges, \quad
\frac 1 {\delta t}\, \left(\tilde \bfu^{n+1}_{\edge,i}-\bfu_{\edge,i}^n \right) + \dive(\tilde u^n_i \bfu^n)_\edge
+ (\gradi p)^n_{\edge,i} \\
- \dive\bigl(\mu\, (\gradi \tilde \bfu^{n+1} + (\gradi \tilde \bfu^{n+1})^t)\bigr)_{\edge,i}=0.
\end{multlined}
\\[2ex] \nonumber &
\mbox{{\bf Correction step} -- Solve for $p^{n+1}$ and $\bfu^{n+1}$:}
\\ & \qquad
\text{ For } 1 \leq i \leq d,\ \forall \edge \in \edges,\quad
\frac 1 {\delta t}\,(\bfu^{n+1}_{\edge,i}-\tilde \bfu_{\edge,i}^{n+1}) + (\gradi p^{n+1})_{\edge,i}- (\gradi p^n)_{\edge,i}=0,
\\  & \qquad
\ \forall K \in \mesh, \quad \dive(\bfu^{n+1})_K=0.
\end{flalign}
\end{subequations}
The convection terms are those introduced in this paper, with the density set to $1$ in the mass fluxes.
The term $(\gradi p)^n_{\edge,i}$ stands for the $i$-th component of the discrete pressure gradient built at the face $\edge$, given by:
\begin{equation} \label{eq:def_grad}
\forall \edge \in \edgesint,\ \edge = K|L,\quad (\gradi p)^n_{\edge,i}=\frac{|\edge|}{|D_\edge|}(p_L-p_K)\ \bfn_{K,\edge} \cdot \bfe^{(i)},
\end{equation}
with $\bfe^{(i)}$ the $i$-th vector of the orthonormal basis of $\xR^d$, and $\bfn_{K,\edge}$ the normal vector to the face $\edge$ outward the cell $K$.
We use the usual finite element discretization for the viscous term, which reads:
\begin{equation} \label{eq:def_diff}
-\dive \bigl(\mu (\gradi \tilde \bfu^{n+1} +  (\gradi \tilde \bfu^{n+1})^t) \bigr)_{\edge,i} 
= - \frac 1 {|D_\edge|} \sum_{K \in \mesh}\ \int_K \bigl(\mu (\gradi \tilde \bfu^{n+1} + (\gradi \tilde \bfu^{n+1})^t)) \bigr) : \gradi \bvphi\ei_\edge \dx,
\end{equation}
where $\bvphi\ei_\edge$ stands for the vector-valued Rannacher-Turek finite element shape function associated with the $i$-th component of the velocity and to the face $\edge$ (with the version of the element where the mean value of the shape function over the face is equal to 1) and the operator $:$ is defined by $A:B= \sum_{i,j=1}^d A_{i,j} B_{i,j}$ for two matrices $A$ and $B$ of $\xR^{d\times d}$. 
Finally, the discretization of the divergence of the velocity on the primal mesh reads:
\[
\dive(\bfu^{n+1})_K=\frac{1}{|K|}\sum_{\edge \in \edges(K)}|\edge|\ \bfu_\edge.\bfn_{K,\edge},
\]
which, together with Equation \eqref{eq:def_grad}, ensures the usual discrete $\gradi - \dive$ $L^2$-duality.

\medskip
The initial values of the unknowns are given by an average of the initial data:
\[
\text{For } 1\leq i \leq d,\ \forall \edge \in \edges, \quad \bfu^0_{\edge,i} = \frac 1 {|\edge|} \int_\edge  \bfu_{0,i}(\bfx) \dedge(\bfx),
\]
where $\dedge$ stands for the $d-1$-dimensional Lebesgue measure and $\bfu_0 = (u_{0,1},\dots,u_{0,d})^t$ is the initial condition for the velocity, supposed to be regular enough for the integral over the faces to be defined (for instance, $\bfu_0 \in H^1(\Omega)^d$).
Note that, if $\bfu_0$ is divergence-free, then the discrete divergence of $\bfu^0$ vanishes.

\medskip
For the scheme \eqref{eq:inc_sch}, a control of the kinetic energy (or, in other words, of the predicted velocity in discrete $L^2(H^1)$ norm and in $L^\infty(L^2)$ norm, and of the end of step velocity in $L^\infty(L^2)$ norm) may be derived using Theorem \ref{th:kin}, following known techniques \cite{she-92-err,gue-96-som}.
%
%
\subsubsection{Flow past a cylinder}

We compute here a two-dimensional flow past a cylinder, inspired from a literature benchmark (\textbf{Test Case 2D-2} of \cite{sch-96-opt}).
The computational domain is the same as in \cite{sch-96-opt}, and consists of a rectangular channel with a cylindrical obstacle near the inlet (left) boundary; we refer to \cite[Figure 1]{sch-96-opt} for the exact definition of the domain. 
At the time $t=0$, the fluid is at rest.
The velocity satisfies a homogeneous Dirichlet condition at the top and bottom sides, and the flow leaves freely the domain through the right-hand side. 
It enters the domain on the left boundary with an imposed velocity profile:
\[
u_x(0,y)=4 \ u_m \ \frac{y\,(H-y)}{H^2}, \ u_y(0,y)=0, \ \ \forall y \in [0,H],
\]
where $H=0.41$ is the height of the domain and $u_m = 1.5$.
The robustness of the scheme for strongly convection dominated flow is assessed by changing the Reynolds number value $Re$ chosen in \cite{sch-96-opt} ($Re=100$) to a larger value, namely $Re=500$ (with $Re = (\rho \overline{u} D)/\mu$ where $\overline{u} = 2 u_x(0, H/2)/3 = 1$).
To this purpose, the density is fixed at $\rho=1$ and the viscosity is equal to $\mu = 0.0002$ .
The computations are first performed using a very coarse grid with 4033 cells (see Figure \ref{fig:cylinder_mesh}), representative of what is often encountered in complex 3D industrial simulations. 
The time step is $\delta t = 0.002$.

\begin{figure} \centering
\includegraphics[width=0.95\textwidth]{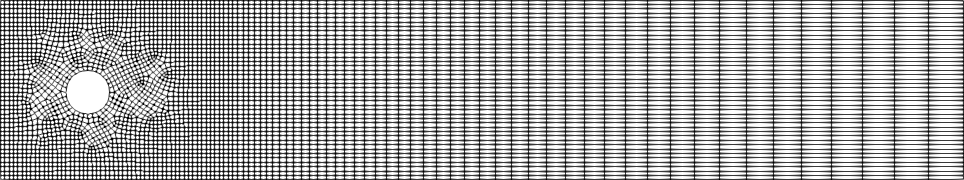}
\caption{Flow past a cylinder - Coarse mesh.}
\label{fig:cylinder_mesh}
\end{figure}

\medskip
The results are plotted in Figure \ref{fig:NSI_FPAC_R500_coarse}, together with the results obtained with (implicit-in-time) upwind and centered convection operators.
For all the schemes, the flow is unsteady.
As expected, the upwind operator introduces a large numerical diffusion; this is not the case for the other operators. 
The centered scheme yields an unrealistic large recirculation zone.
The computation is then run on refined grids (12913 cells and 43009 cells), with an adjusted time step ($\delta t = 0.000625$ and $\delta t = 0.000187$ respectively).
On these grids, the centered scheme seems to yield results more in line with the ones  obtained with the upwind and MUSCL discretizations, as can be seen in Figure \ref{fig:NSI_FPAC_R500_fine}.
This confirms that, on the coarsest grid, the solution obtained with the MUSCL scheme is much more accurate than with the other discretizations.

\begin{figure}
    \centering
    \includegraphics[width=0.95\textwidth]{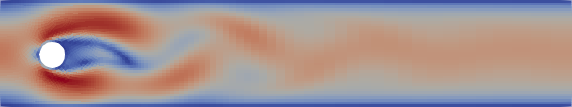} \\[1ex]
    \includegraphics[width=0.95\textwidth]{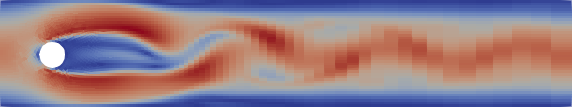} \\[1ex]
    \includegraphics[width=0.95\textwidth]{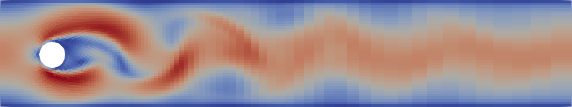}
    \caption{Flow past a cylinder - Magnitude of the velocity at time $t=5$ (coarse mesh). From top to bottom: upwind scheme, centered scheme, MUSCL scheme.}
    \label{fig:NSI_FPAC_R500_coarse}
\end{figure}

\begin{figure}
    \centering
    \includegraphics[width=0.95\textwidth]{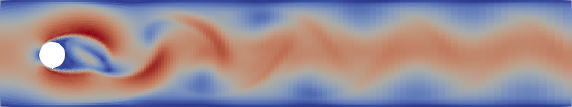} 
    \caption{Flow past a cylinder - Magnitude of the velocity at time $t=5$ for the most refined mesh, with the centered scheme.}
    \label{fig:NSI_FPAC_R500_fine}
\end{figure}

\medskip
To further assess the quality of the different schemes, we turn to the other outputs studied in of \cite[Test Case 2D-2]{sch-96-opt} (even though our aim is not to compare our results with those of \cite{sch-96-opt}, since the viscosity is different).
The main quantities of interest are the pressure difference $\Delta P$ between the front and end points of the cylinder (\ie\ the points $(0.15, 0.20)$ and $(0.25, 0.20)$ respectively), the Strouhal number, the maximum drag coefficient, and the maximal and minimal lift coefficients (see \cite{sch-96-opt} for a definition).
They are gathered in Tables \ref{tab:NSI_FPAC_R500_data_upwind}, \ref{tab:NSI_FPAC_R500_data_MUSCL} and \ref{tab:NSI_FPAC_R500_data_centered}, and the computed drag and lift coefficients are plotted as a function of time on Figures \ref{fig:NSI_FPAC_R500_DragLiftCoeff} and \ref{fig:NSI_FPAC_R500_DP}.
With the centered scheme, the computed flow does not seem to tend to a periodic flow, contrary to what happens with the MUSCL and upwind schemes.
Even if the convergence is far from being reached with the (intentionally) very coarse mesh used in this study, the MUSCL scheme seems able to capture at least the order of magnitude of the recorded quantities (see in particular the lift coefficient in Table \ref{tab:NSI_FPAC_R500_data_MUSCL}).

\medskip
To sum up, the conclusion of this test is that, for the simulation of such convection-dominated flow, the MUSCL scheme seems to be a better alternative than the upwind and centered schemes on coarse meshes (representative of industrial simulations): indeed the upwind and centered schemes respectively suffer from an over-diffusion and a lack of stability.

\begin{table} 
    \centering 
    \begin{tabular}{c|c|c|c} 
        Number of cells & 4033 & 12913 & 43009 \\ \hline
        Min. mesh area & $3.43\times10^{-5}$ & $1.11 \times 10^{-5}$ & $2.55 \times 10^{-6}$ \\ \hline
        $\Delta P$ & 2.29620 & 2.37170 & 2.53970\\
        Strouhal number & 0.22257 & 0.25077 & 0.27523\\
        Max. drag coeff. & 3.23134 & 3.01118 & 2.81112\\
        Max. lift coeff. & 0.51332 & 1.11934 & 1.50993\\
        Min. lift coeff. & -0.50646 & -0.95858 & -1.44269 \\
    \end{tabular} 
    \caption{Flow past a cylinder - Quantitative results for the upwind scheme.} 
    \label{tab:NSI_FPAC_R500_data_upwind} 
\end{table}

\begin{table} 
    \centering 
    \begin{tabular}{c|c|c|c} 
        Number of cells & 4033 & 12913 & 43009 \\ \hline
        Min. mesh area & $3.43\times10^{-5}$ & $1.11 \times 10^{-5}$ & $2.55 \times 10^{-6}$ \\ \hline
        $\Delta P$ & 2.38970 & 2.52830 & 2.76460\\
        Strouhal number & 0.25112 & 0.27822 & 0.29464 \\
        Max. drag coeff. & 3.38864 & 3.19996 & 2.99350\\
        Max. lift coeff. & 0.96980 & 1.75976 & 2.21766 \\
        Min. lift coeff. & -0.98392 & -1.43585 & -1.91979 \\
    \end{tabular} 
    \caption{Flow past a cylinder - Quantitative results for the MUSCL scheme.} 
    \label{tab:NSI_FPAC_R500_data_MUSCL} 
\end{table}

\begin{table} 
    \centering 
    \begin{tabular}{c|c|c|c} 
        Number of cells & 4033 & 12913 & 43009 \\ \hline
        Min. mesh area & $3.43\times10^{-5}$ & $1.11 \times 10^{-5}$ & $2.55 \times 10^{-6}$ \\ \hline
        $\Delta P$ & - & 2.34780 & 3.07140\\
        Strouhal number & - & 0.26484 & 0.30252 \\
        Max. drag coeff. & 3.20972 & 3.42892 & 3.51592\\
        Max. lift coeff. & 0.15683 & 1.36092 & 2.50430\\
        Min. lift coeff. & -0.14332 & -1.23030 & -2.42746\\
    \end{tabular} 
    \caption{Flow past a cylinder - Quantitative results for the centered scheme.} 
    \label{tab:NSI_FPAC_R500_data_centered} 
\end{table}

\begin{figure}
    \centering
    \begin{tabular}{cc}
        \includegraphics[trim=0 0 0 5, clip, width=0.47\textwidth]{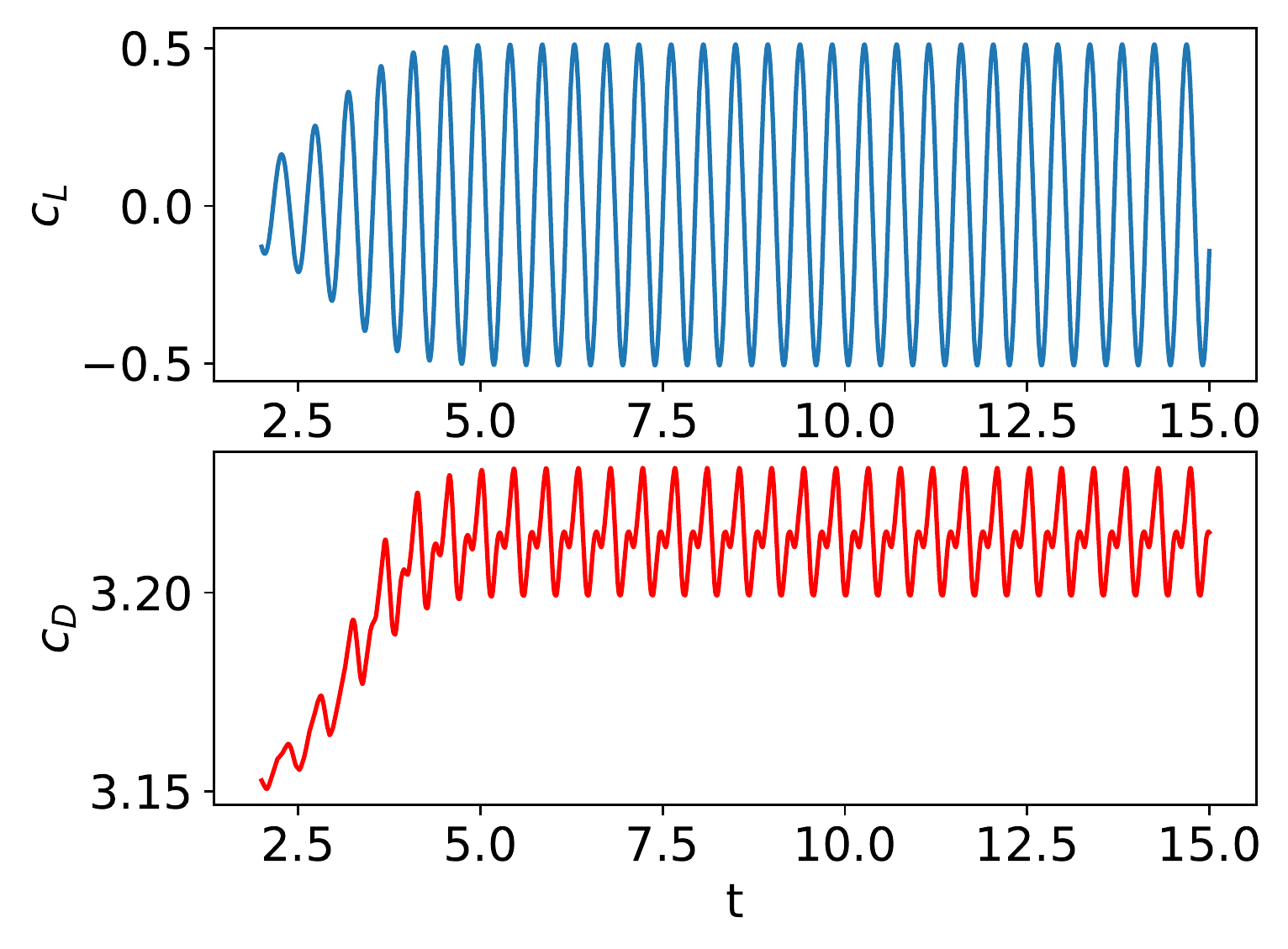} &
        \includegraphics[trim=0 0 0 5, clip, width=0.455\textwidth]{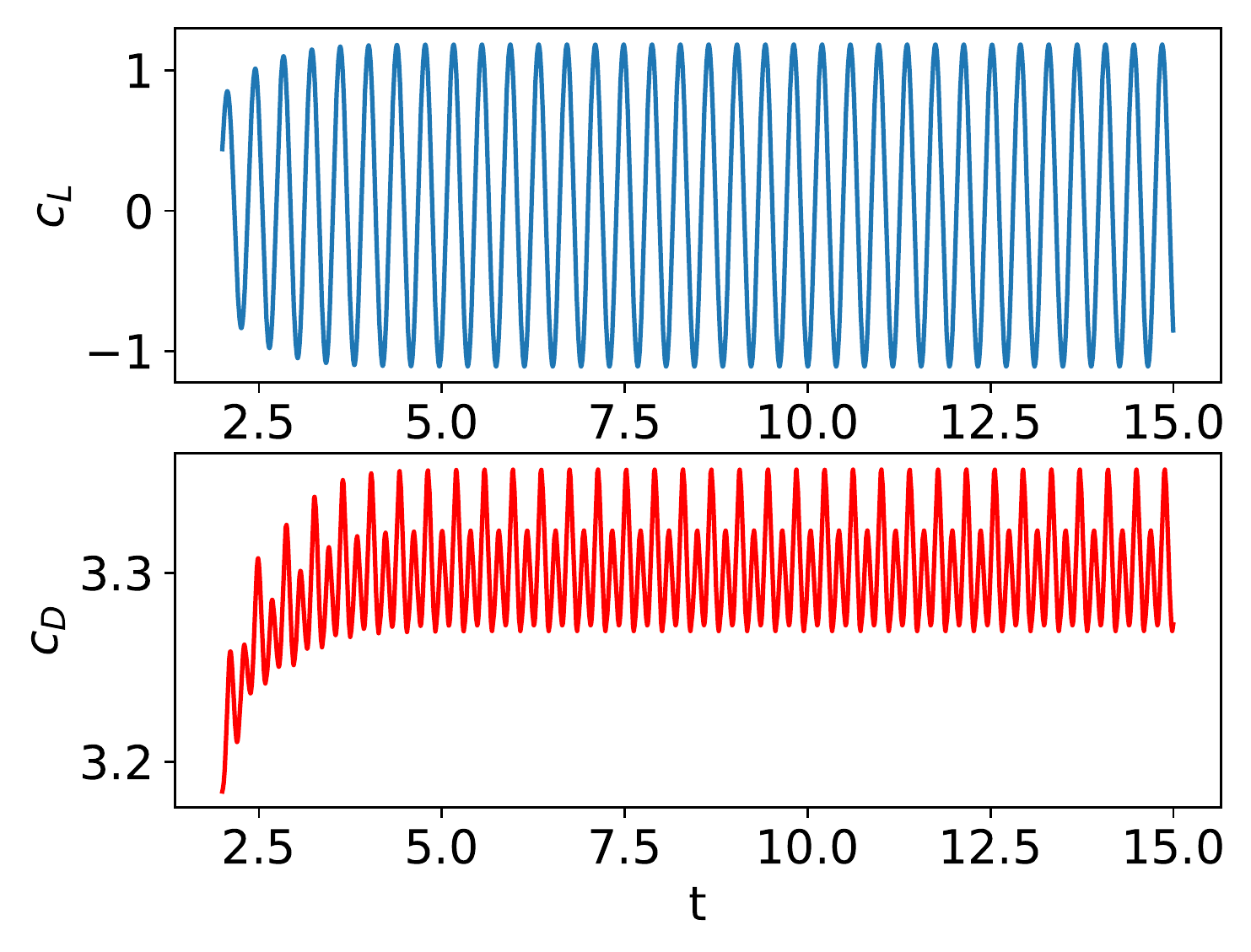}
    \end{tabular}
    \includegraphics[trim=0 0 0 5, clip, width=0.47\textwidth]{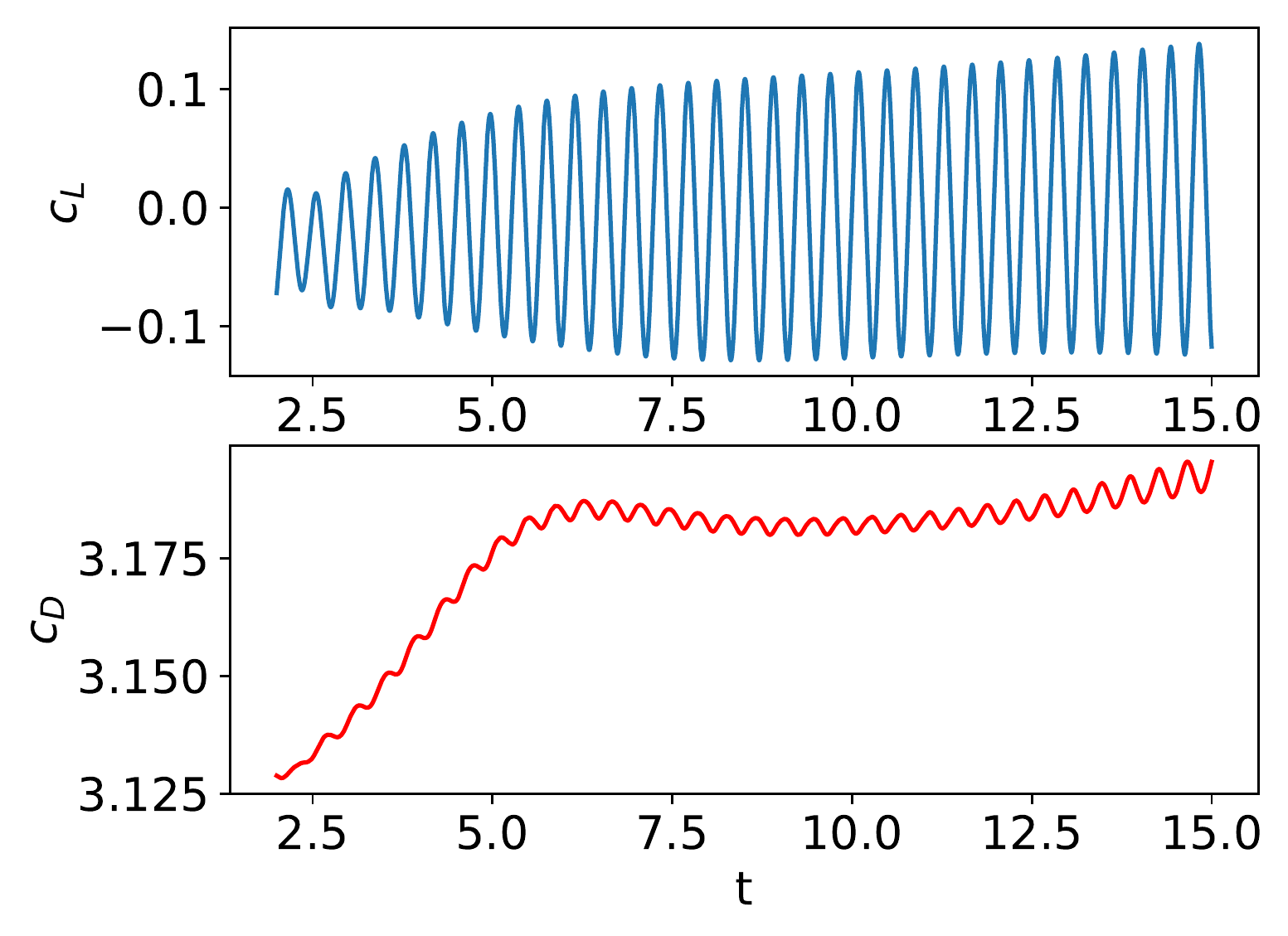}    
    \caption{Flow past a cylinder - Drag and lift coefficient as a function of time (coarse mesh). Top-left: upwind scheme. Top-right: MUSCL scheme. Bottom: centered scheme.}
    \label{fig:NSI_FPAC_R500_DragLiftCoeff}
\end{figure}

\begin{figure}
    \centering
    \begin{tabular}{cc}
        \includegraphics[trim=0 0 0 5, clip, width=0.47\textwidth]{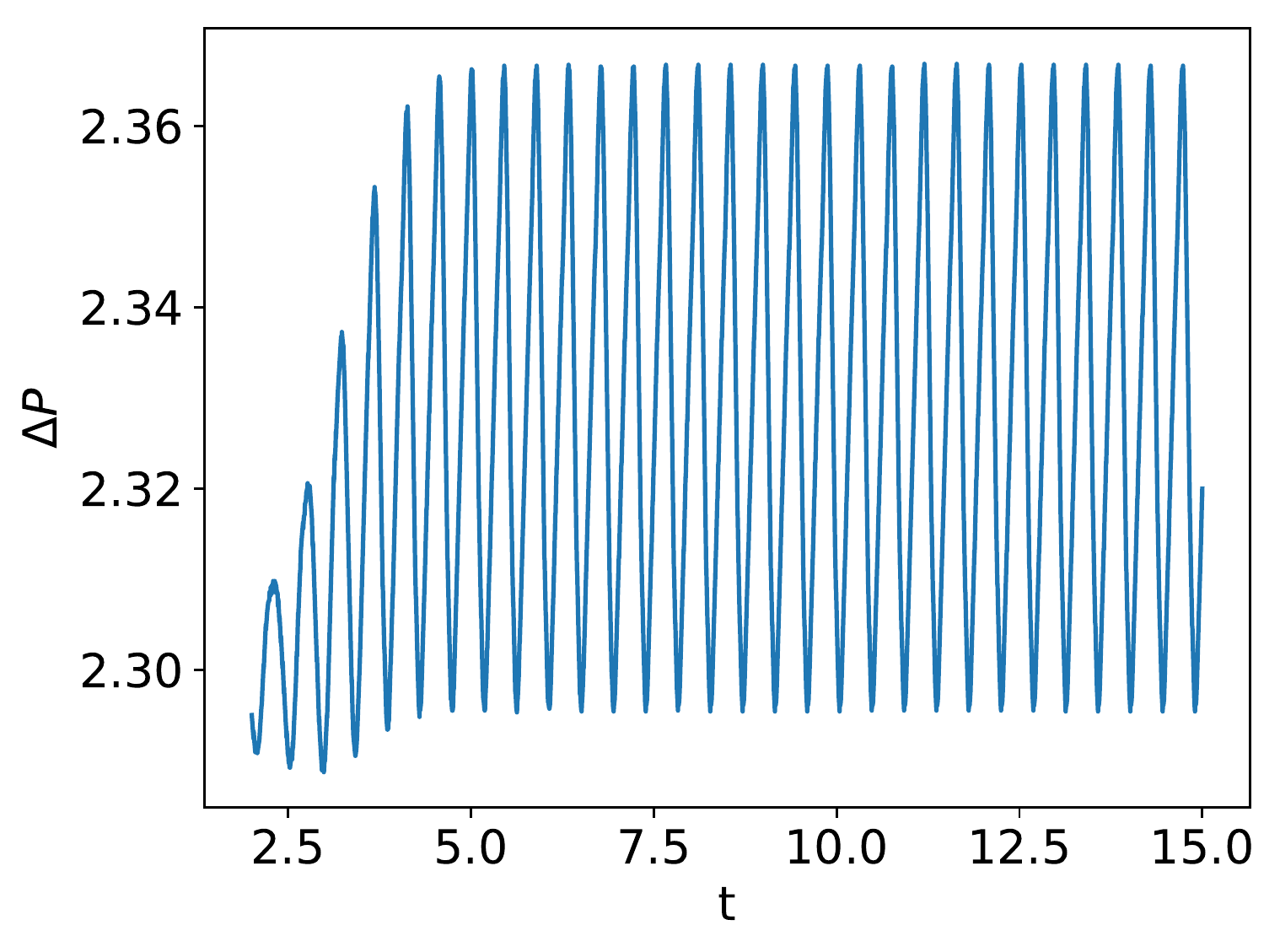} &
        \includegraphics[trim=0 0 0 5, clip, width=0.468\textwidth]{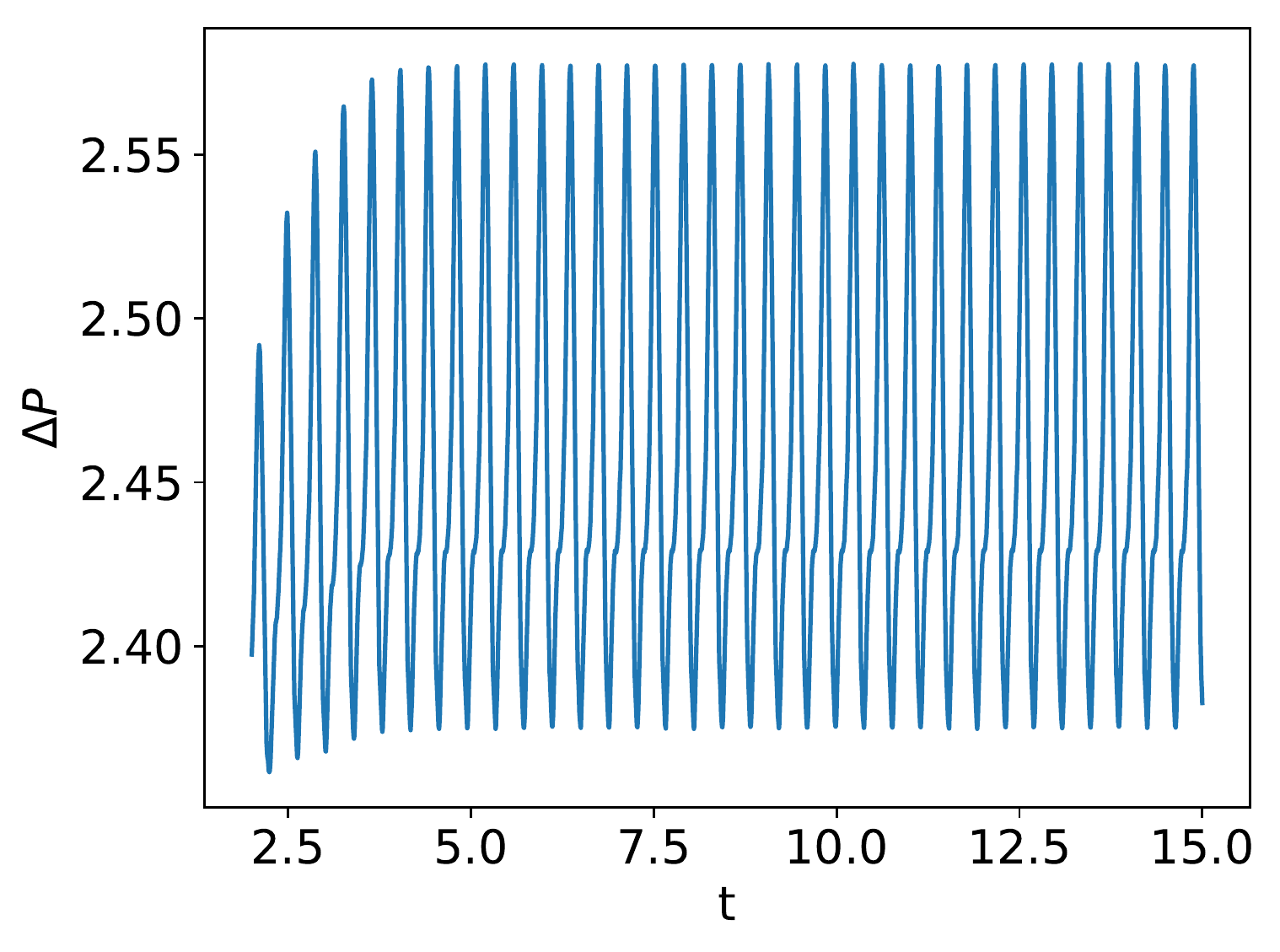}\\
    \end{tabular}
    \includegraphics[trim=0 0 0 5, clip, width=0.47\textwidth]{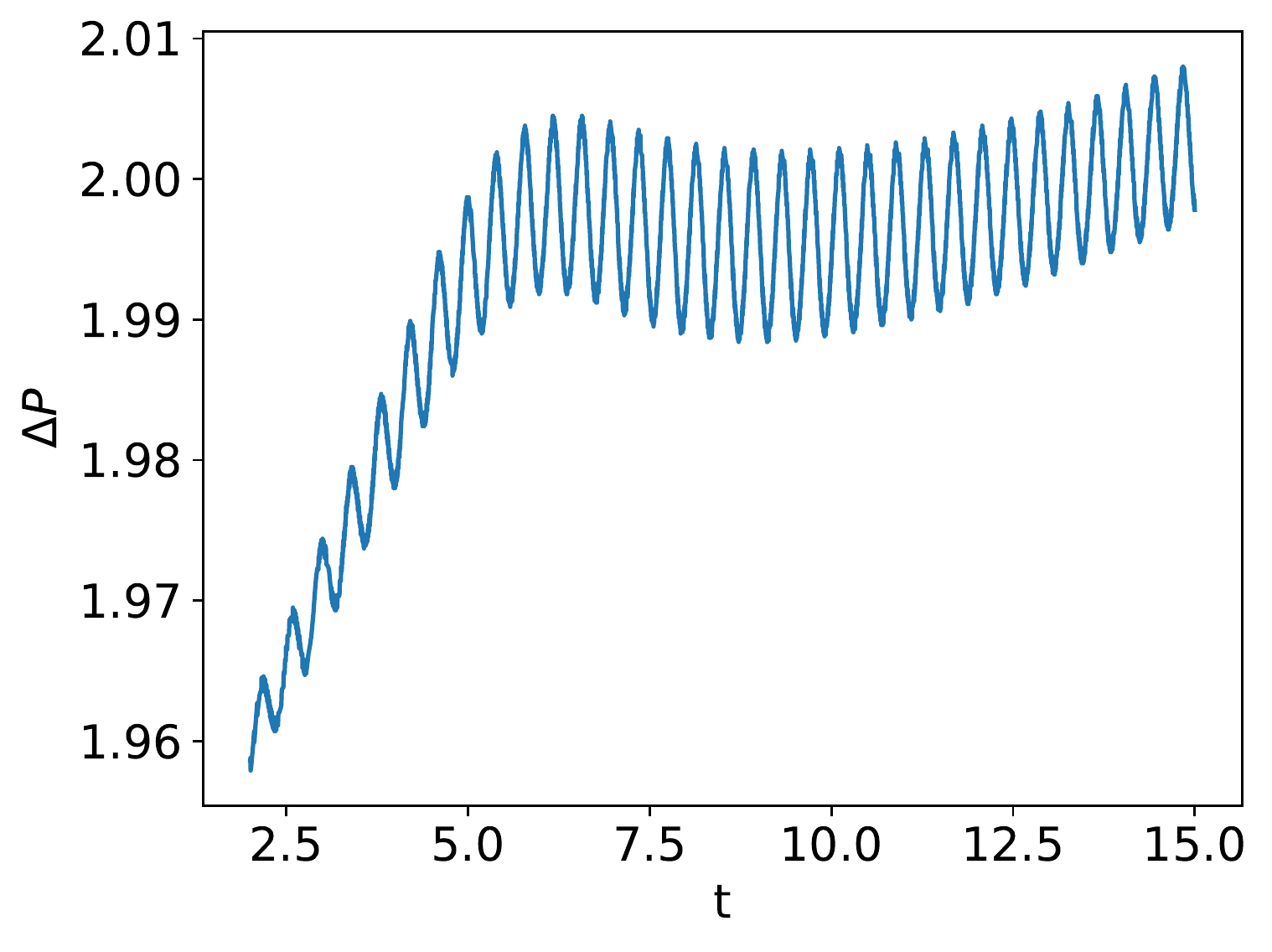}
\caption{Flow past a cylinder - $\Delta P$ as a function of time on the coarsest mesh. Top-left: upwind scheme. Top-right: MUSCL scheme. Bottom: centered scheme.}
    \label{fig:NSI_FPAC_R500_DP}
\end{figure}

%
%
\subsubsection{Lid-driven cavity} \label{subsec:LidDrivenCavity}

We now turn to the well-known $2D$ lid-driven cavity flow test case, which is a classical test problem for the validation of Navier-Stokes schemes. 
It consists in the study of a flow in the square $[0,1] \times [0,1]$. 
Homogeneous Dirichlet boundary conditions are prescribed to the velocity on the left, right, and bottom sides. 
The velocity is tangential to the top side, and its norm is set to 1, \ie:
\begin{align}
u_x(x,1)=1, \ \ u_y(x,1)=0 \ \ \forall x \in [0,1].
\end{align}
The value of the viscosity is chosen to obtain a Reynolds number $Re$ equal to $5000$, with $Re = \rho \overline{u} D/\mu$ with $\rho = 1$, $D = 1$, $\overline{u}=1$ and $\mu = 0.0002$.
With this value of the Reynolds number, the problem is known to converge to a steady state.
To reach this state, we let the computation run up to a final time of $T=200$ seconds (with a time step of $\delta t = 0.0025$), which is enough to obtain a relative difference between the velocity at two successive time steps in the range of $10^{-6}$.
This test is classical, and numerous computations are available (see e.g. \cite{ghi-82-hig,bot-98-ben,bru-06-2d}); the reference used in this paper is a converged-in-space computation that can be found in \cite{bru-06-2d}. 

\medskip
We perform two computations, with uniform $128 \times 128$ and $256 \times 256$ grids, respectively. 
The amplitude of the variations of the streamline function and the location of the center of the primary and bottom right secondary vortices obtained with the upwind, centered, and MUSCL schemes are reported in Table \ref{tab:AmpStreamline} and \ref{tab:LocVortexR5000} respectively. 
The location of the center of the vortices is defined as the point where the streamline function reaches an extremum: the primary vortex corresponds to the minimum of the streamline function, while the secondary vortex corresponds to a maximum.
On both grids, the amplitude of the streamline function variations seems to be overvalued with the upwind discretization, and undervalued with the centered one, while the MUSCL discretization yields a better agreement with the reference value.
Concerning the location of the vortices, all methods seem to give close outcomes, and the results are in reasonable agreement with the reference ones; with the upwind discretization, both vortices seem to be however slightly shifted upward compared to the higher-order methods.
Slight differences may also be observed on the shape of the vortices (Figures \ref{fig:PrimVortR5000} and \ref{fig:SecVortR5000} for the primary and secondary vortex, respectively).

\begin{table} 
    \centering 
    Ref \cite{bru-06-2d} on refined mesh: 0.1249994 \\ 
    \begin{tabular}{c|c|c} 
        Scheme & Grid $128\times 128$ & Grid $256\times 256$\\ \hline
        Upwind & 0.1539811 & 0.1551107\\
        Centered & 0.0877255 & 0.1081858\\
        MUSCL & 0.1066407 & 0.1155036\\
    \end{tabular} 
    \caption{Lid-driven cavity - Amplitude of the streamline function variations ($\psi_{max}-\psi_{min}$).} 
    \label{tab:AmpStreamline} 
\end{table}
\begin{table} 
    \centering 
    \begin{tabular}{c|c|c|c|c|c} 
        Scheme & Grid & $x_{pv}$ & $y_{pv}$ & $x_{sv}$ & $y_{sv}$\\ \hline
        Ref \cite{bru-06-2d} & $1024 \times 1024$ & 0.51465 & 0.53516 & 0.80566 & 0.073242\\ \hline
        Upwind & $128 \times 128$ & 0.516 & 0.547 & 0.820 & 0.086\\
        Centered & $128 \times 128$ & 0.516 & 0.539 & 0.820 & 0.078\\
        MUSCL & $128 \times 128$ & 0.516 & 0.539 & 0.812 & 0.078\\ \hline
        Upwind & $256 \times 256$ & 0.516 & 0.543 & 0.812 & 0.078\\
        Centered & $256 \times 256$ & 0.512 & 0.539 & 0.812 & 0.074\\
        MUSCL & $256 \times 256$ & 0.512 & 0.535 & 0.809 & 0.074\\
    \end{tabular} 
    \caption{Lid-driven cavity - Location of the primary vortex $(x_{pv},y_{pv})$ and the lower right secondary vortex $(x_{sv},y_{sv})$.} 
    \label{tab:LocVortexR5000} 
\end{table}

\begin{figure}
    \centering
    \fbox{\includegraphics[clip=true, trim= 130 135 120 110, width=0.3\textwidth]{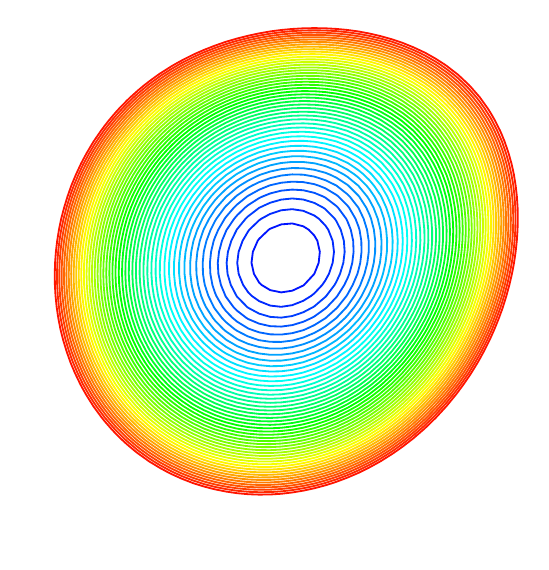}}
    \fbox{\includegraphics[clip=true, trim= 130 135 120 110, width=0.3\textwidth]{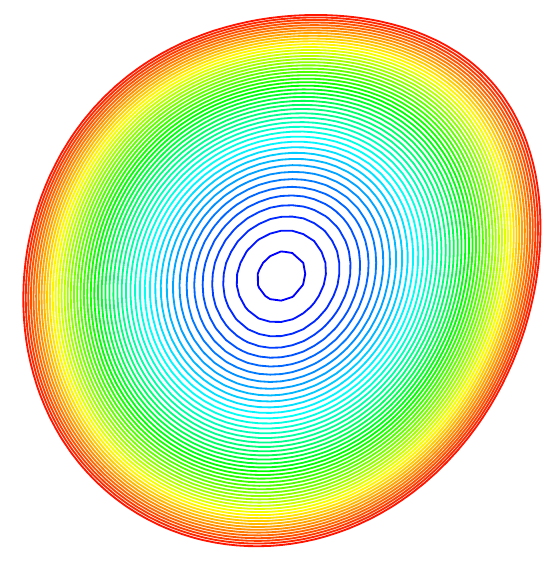}}
    \fbox{\includegraphics[clip=true, trim= 130 135 120 110, width=0.3\textwidth]{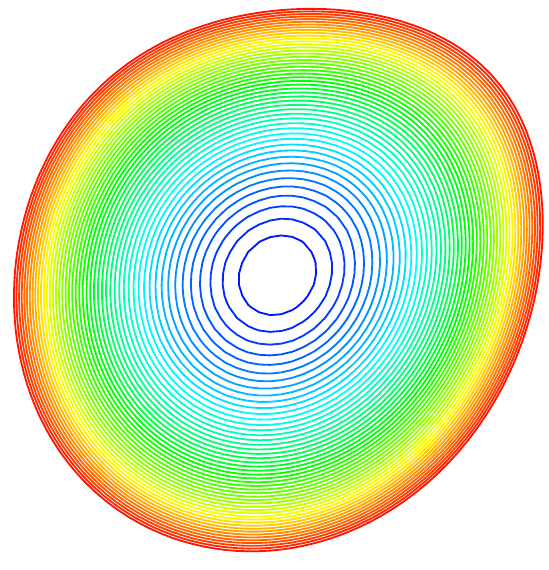}}
    \caption{Lid-driven cavity - Primary vortex. From left to right: results with the upwind, centered, and MUSCL scheme.}
    \label{fig:PrimVortR5000}
\end{figure}

\begin{figure}
    \centering
    \fbox{\includegraphics[clip=true, trim= 280 3 2 280, width=0.3\textwidth]{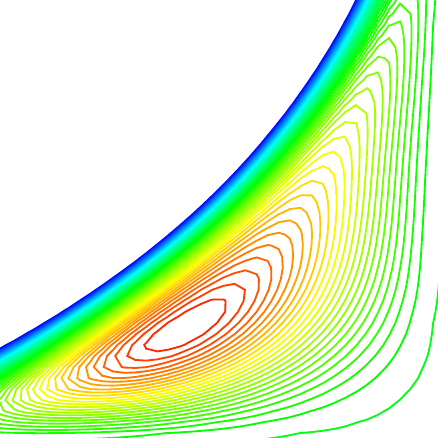}}
    \fbox{\includegraphics[clip=true, trim= 280 3 2 280, width=0.3\textwidth]{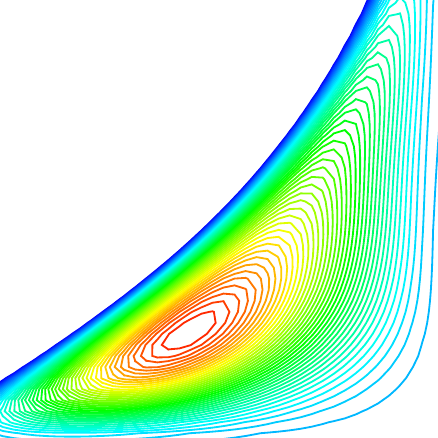}}
    \fbox{\includegraphics[clip=true, trim= 280 3 2 280, width=0.3\textwidth]{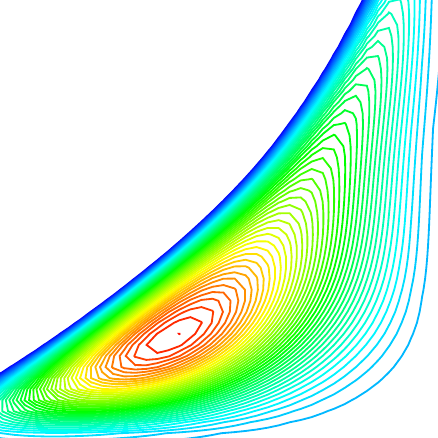}}
    \caption{Lid-driven cavity - Secondary vortex. From left to right: results with the upwind, centered, and MUSCL scheme.}
    \label{fig:SecVortR5000}
\end{figure}

%
%
\subsubsection{Backward-facing step}\label{subsec:BackwardFacingStep}

We finally address the so-called backward-facing step problem, introduced in \cite{arm-83-exp} and also addressed in \cite{chi-99-num, ans-11-sta}. 
The domain is rectangular, its length is set to $L=20$ and its height to $H=1.9423$.
The flow enters the domain $\Omega$ through its left boundary and a step of height $h=0.9423$ is considered at the left of the computational domain, outside and adjacent to $\Omega$; the step is thus only modelled by the boundary conditions, and a parabolic velocity profile above it is assumed.
Consequently, Dirichlet conditions are prescribed at the left, top, and bottom boundaries, the velocity being set to zero except in the inlet part of the boundary, \ie\ the part of the left side located above $h=0.9423$; homogeneous Neumann conditions are imposed on the right side of the domain.
The fluid density is $\rho=1$, the viscosity is $\mu=0.001$ and the peak velocity in the inlet boundary is equal to $1$, which corresponds to a Reynolds number $Re=1000$ (with respect to this maximum inlet velocity).
The mesh used here is a rather coarse $250\times50$ grid, and the time step is $\delta t = 0.01$.

\medskip
The streamlines vortices at time $t=20$ are plotted on Figure \ref{fig:BFS_R5000_50x250}. 
As expected, the upwind scheme is the most diffusive: all the vortices are damped, with a quasi-complete disappearance of the one located at the right of the reattachment point. 
Both centered and MUSCL schemes yield qualitatively similar results.

\begin{figure}
    \centering
    \includegraphics[width=0.9\textwidth]{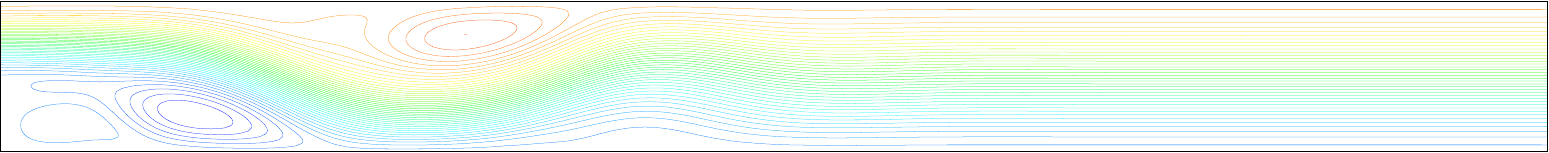}
    \includegraphics[width=0.9\textwidth]{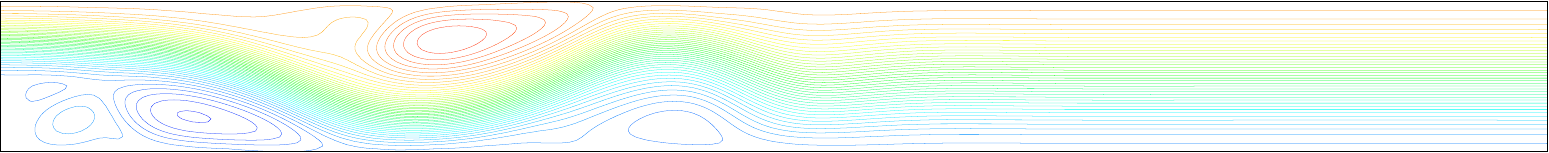}
    \includegraphics[width=0.9\textwidth]{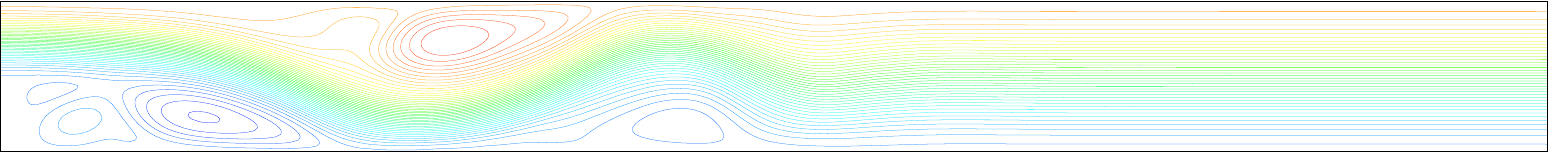}
    \caption{Backward-facing step - Streamlines at time $t=20$. From top to bottom: upwind, centered, and MUSCL schemes.}
    \label{fig:BFS_R5000_50x250}
\end{figure}
%
%
\subsection{Compressible barotropic Navier-Stokes equations}

We now show applications to the barotropic (isentropic) compressible Navier-Stokes equations:
\begin{subequations} \label{eq:ns_bar}
\begin{align} 
\label{eq:mass_balance} &
\partial_t \rho + \dive(\rho\bfu)  = 0, \\
\label{eq:mom_balance} &
\partial_t (\rho u_i) + \dive(\rho u_i \bfu) + \partial_i p - \dive(\mu(\gradi \bfu+\gradi \bfu^t))_i = 0, \qquad 1\leq i \leq d,
\\ &
p  = a \rho^\gamma,\quad a >0,\ \gamma \geq 1.
\end{align}    
\end{subequations}
%
%
\subsubsection{The scheme}

A first-order forward Euler time-discretization of System \eqref{eq:ns_bar} reads:
\begin{subequations} \label{eq:ns_bar_scheme}
\begin{flalign} \label{eq:scheme_mass} &
\forall K \in \mesh,\quad \dfrac{1}{\delta t}(\rho^{n+1}_K-\rho^n_K) + \dive( \rho^n \bfu^n)_K=0,
\\ \nonumber &
\text{For } 1\leq i\leq d,\ \forall \edge \in \edges,
\\ \label{eq:scheme_mom} &
\phantom{\forall K \in \mesh, \quad}
\dfrac{1}{\delta t}(\rho^{n+1}_{D_\edge} u^{n+1}_{\edge,i}-\rho^n_{D_\edge} u^n_{\edge,i}) + \dive(\rho^n u_i^n\bfu^n)_\edge
+ (\gradi p)^n_{\edge,i} - \dive(\mu (\gradi \bfu^n + \gradi (\bfu^n)^t))_{\edge,i}=0,
\\ \label{eq:scheme_eos} &
\forall K \in \mesh, p^{n+1}_K=a\ (\rho^{n+1}_K)^\gamma.
\end{flalign} 
\end{subequations}
The momentum discrete convection terms are described in the previous section, the mass fluxes in \eqref{eq:scheme_mass} are approximated with a MUSCL scheme \cite{pia-13-for} and the other terms of the system are the same as those used for the incompressible Navier-Stokes equations. 
The first-order scheme \eqref{eq:ns_bar_scheme} is then extended to second-order in time using a second order Runge-Kutta scheme (or Heun scheme), which reads, with $\boldsymbol{W}^n=(\rho^n,u^n,p^n)$ the unknowns at step $n$ and $S(\boldsymbol W)$ the new unknowns resulting from the application of \eqref{eq:ns_bar_scheme} to the unknown vector $\boldsymbol W$:
\begin{equation}\label{eq:heun_algo}
\boldsymbol{W}^{n+\frac{1}{3}}=S(\boldsymbol{W}^n),\quad
\boldsymbol{W}^{n+\frac{2}{3}}=S(\boldsymbol{W}^{n+\frac{1}{3}}), \quad
\boldsymbol{W}^{n+1}=\frac{1}{2}(\boldsymbol{W}^n+\boldsymbol{W}^{n+\frac{2}{3}}).
\end{equation}
Unless specified, the Heun scheme is used in the following numerical tests.
%
%
\subsubsection{Travelling vortex}

Here we assess the convergence rate of the proposed scheme on a test case built for this purpose. 
We first derive an analytical solution of the steady isentropic Euler equations, consisting in a standing vortex; then this solution is made unsteady by adding a constant velocity translation. 
A solution to the Navier-Stokes equations is finally derived by compensating the viscous forces (that appear on the left-hand side of equation \eqref{eq:mom_balance}) with a source term.
We refer to \cite{gal-20-sec} for the exact expression of this solution.
We take $a=9.81/2$ and $\gamma = 2$, so System \eqref{eq:ns_bar} is identical to the (viscous) shallow-water equations without bathymetry.
The viscosity $\mu$ is chosen so that the Reynolds number is equal to $50$.
The domain is the square $\Omega=[-1.2,2]^2$ and the computation is run on the time interval $[0,0.8]$.

\medskip
The meshes are uniform $n \times n$ grids, starting from a $32 \times 32$ one and then doubling the number of control volumes in each direction until we reach a $256 \times 256$ mesh.
The time step is set to $0.03125\times h$, with $h=3.2/n$, which yields a $\cfl$ number with respect to the celerity of the fastest wave close to $0.07$ (the material velocity and the speed of sound are in the range of $1.45$ and $0.76$ respectively), this low value of the $\cfl$ number being imposed by the explicit discretization of the diffusion term (the constraint stems from the necessity to be stable up to the finest mesh).

\medskip
In Figure \ref{fig:err_NS}, we draw the $L^1$ norm of the numerical error for the velocity and the density as a function of the mesh step.
This error is obtained by taking the difference between the computed velocity or density at the final time and the piecewise constant function defined by taking the value of the continuous solution at the diamond or primal cell center.
The measured orders of convergence are close to $1.8$ and $2$ for the velocity and the density respectively, which corresponds to the properties which are expected for the scheme.
In this respect, note that we work here with uniform meshes, so the slope limitation $\xi^+=1$ does not prevent to choose the face value given by a second order interpolation; with non-uniform meshes, a limitation of the order of convergence would probably be observed (unless relaxing the limitation to $\xi^+=2$, which is possible).

\begin{figure}
    \centering
    \includegraphics[width=0.47\textwidth]{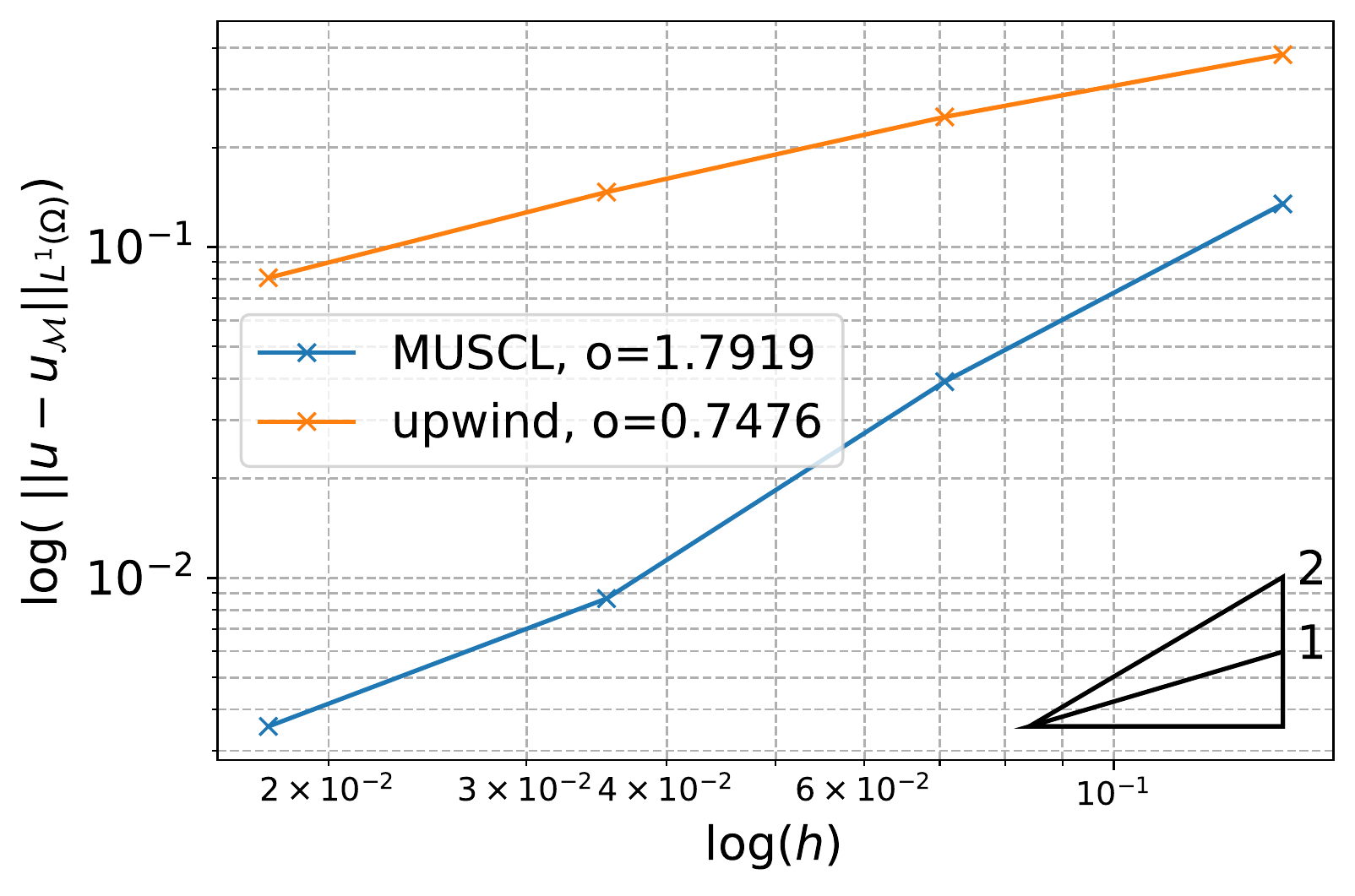}
    \includegraphics[width=0.47\textwidth]{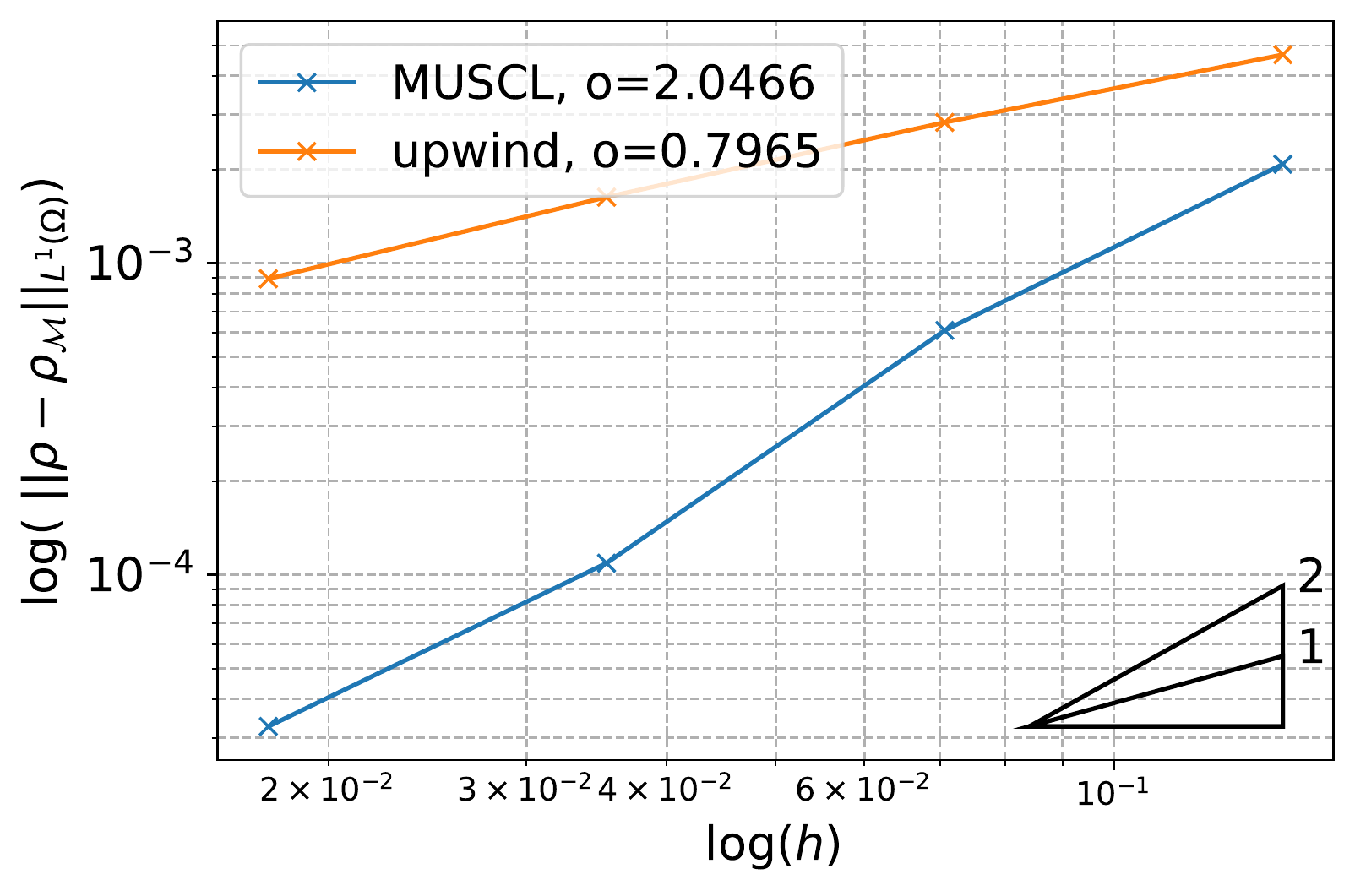}
    \caption{Barotropic travelling vortex - $L^1$ norm error for the MUSCL scheme and the upwind scheme for the velocity and the density. Here $h'=\max_{K \in \mesh}\text{diam}(K)=\sqrt{2}h$.}
    \label{fig:err_NS}
\end{figure}
%
%
\subsubsection{Flow past a cylinder} \label{sec:baro_cyl}

We now turn to a two-dimensional problem, namely an adaptation to the compressible case of the flow past a cylinder problem already studied in the incompressible context.
The geometry of the domain is thus once again given in \cite[Figure 1]{sch-96-opt}, up to the fact that the left part of the domain is lengthened, to keep the reflected shocks travelling to the left inside the computational domain, up to the final time (see below).
Here the viscosity is set to $\mu=0$, and we keep $a=9.81/2$ and $\gamma=2$, to recover once again the shallow-water equations.
We take as initial data a given homogeneous state $u_0=0$ and $\rho_0=0.2$ over the whole domain, and prescribe the velocity and the density at the left boundary in such a way to generate a shock travelling from the left to the right. 
This shock is supposed to satisfy $M=2$, where $M=\omega/c$ is the so-called Mach number associated with the shock, \ie\ the ratio of the speed of the shock wave $\omega$ to the speed of sound $c$ in the initial medium (or, equivalently, in the right state of the shock), given by $c=\sqrt{2 a \rho_0}$ (so $w=2\ \sqrt{2 a \rho_0}$). 
Using the Rankine-Hugoniot jump relations, we obtain the inlet conditions at the left boundary $x=0$:
\begin{align}
    &\forall y \in [0,H], \quad u_x(0,y)=\omega\left(1-\frac{2}{\sqrt{1+8M^2}-1}\right), \quad u_y(0,y)=0, \\
    &\forall y \in [0,H] \quad \rho(0,y)=0.1\left( \frac{\sqrt{1+8M^2}-1}{2} \right).
\end{align}
Impermeability and perfect slip boundary conditions are prescribed on the other boundaries except on the right side of the domain; here, we let the flow leave the domain "freely"; this is numerically obtained by using a first-order upwind approximation for the convection fluxes (the computed $y$-component of the velocity is positive at any time and all along the boundary) and supposing that the pressure gradient vanishes.

\medskip
The computation is performed on a mesh consisting of 106897 control volumes (which yields a minimum area of the cells equal to $4.44 \times 10^{-7}$), and the time step is equal to $\delta t = 4.10^{-6}$.
For the MUSCL scheme, we observe spurious wiggles which need to be damped with an artificial diffusion term $T_{dif}$ of the form:
\begin{equation}\label{eq:stab_mom}
T_{dif} =\sum_{\substack{\edged \in \edgesd(D_\edge),\\ \edged=D_\edge|D_{\edge'}}}\nu_{\edged}^{n+1}(u^n_{\edge,i}-u^n_{\edge',i}),
\end{equation}
which is added to the left-hand side of Equation \eqref{eq:scheme_mom}. 
The artificial viscosity parameter is constant and equal to:
\[
\nu= \frac 1 {10}\ \sqrt{2 a \rho_\ell}/\ h,
\]
with $h$ the space step and $\rho_\ell=0.65$ the maximal density obtained after reflection of the shock wave on the cylinder.
This yields a viscosity significantly lower than the numerical viscosity which would be introduced by a Godunov scheme (note that $\sqrt{2 a \rho_\ell}$ is the maximum celerity of the sound wave).
The necessity of such a stabilization was already observed in \cite{her-18-cons}; it is probably due to the fact that the scheme numerical diffusion depends linearly (at most, \ie\ with the upwind scheme) on the material velocity (and not the waves celerity, as would be the case for a Godunov scheme), which here moreover vanishes in the right state of the shock.
In our numerical experiments, no reasonable diffusive parameter was sufficient to ensure the stability of the centered scheme, so no result with this discretization is presented here.

\begin{figure}
\centering
\includegraphics[width=0.95\textwidth, viewport = 0cm 0cm 50cm 6cm, clip]{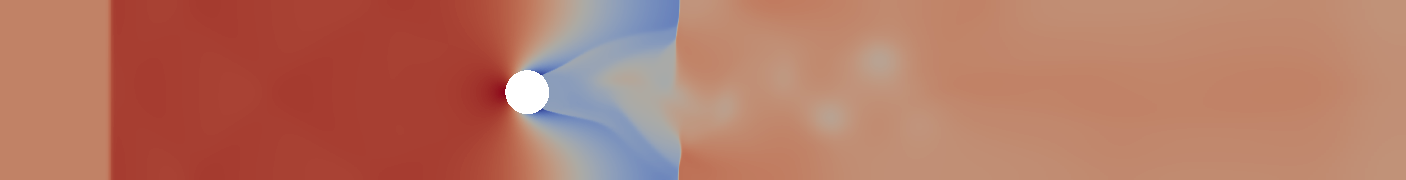} \\[2ex]
\includegraphics[width=0.95\textwidth, viewport = 0cm 0cm 50cm 6cm, clip]{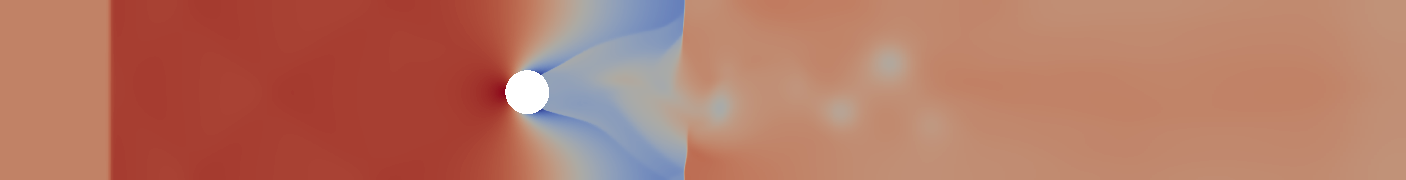}
\caption{Barotropic flow past a cylinder - density at time $t=1$. Top: upwind scheme. Bottom: MUSCL scheme.}
\label{fig:swe}
\end{figure}

\medskip
The computations show a reflection of the shock on the obstacle, which generates a reflected shock (first curved then tending to a plane wave) travelling to the left, together with some complex structures in the obstacle wake, including vortex shedding phenomena, however with a small amplitude.
Density fields obtained at $t=1$ with the scheme proposed here and with an upwind discretization of the momentum balance (while the mass balance is still discretized by a MUSCL scheme) are plotted in Figure \ref{fig:swe}.
These results look qualitatively similar, which is because the governing structures in the flow for the velocity are shocks, where the diffusion brought by the upwind discretization is controlled by the compressive character of the velocity field.
Note also that the Heun scheme is observed to be more diffusive for shock solutions than the first-order forward Euler time marching algorithm \cite{gas-18-mus}, the diffusion being probably generated by the last averaging step of the algorithm (when written under the form \eqref{eq:heun_algo}).
%
%
\subsection{Euler equations}

We now turn to an application of the MUSCL discretization to the compressible Euler equations, which read
\begin{subequations}
\begin{align}\label{eu:eq:pb_mass} &
\partial_t \rho + \dive( \rho\, \bfu) = 0,
\\[1ex] \label{eu:eq:pb_mom} &
\partial_t (\rho\, \bfu) + \dive(\rho\, \bfu \otimes \bfu) + \gradi p= 0,
\\[1ex] \label{eu:eq:pb_Etot} &
\partial_t (\rho\, E) + \dive(\rho \, E \, \bfu) + \dive ( p \, \bfu )=0,
\\ \label{eu:eq:pb_etat} &
 p=(\gamma-1)\, \rho\, e, \qquad E=\frac 1 2|\bfu|^2+e,
\end{align} \label{eu:eq:pb}\end{subequations}
where $\gamma>1$ is a coefficient specific to the fluid under consideration.
As explained in Section \ref{sec:corr_Euler}, while preserving the consistency with the total energy balance \eqref{eu:eq:pb_Etot}, we choose to base the scheme on the internal energy balance equation, which formally takes the following form:
\begin{align}\label{eu:eq:pb_Eint}
\partial_t (\rho e) + \dive(\rho e \bfu)+ p\, \dive \bfu =0.
\end{align}
For shock solutions, this equality becomes an inequality (the left-hand side is non-negative).
%
%
\subsubsection{The scheme}

The discrete unknowns for the internal energy are associated with the primal mesh, and the scheme reads:
\begin{subequations}
\begin{align}
\displaystyle \label{eu:eq:scheme_mass} &
\forall K \in \mesh, \;
\dfrac{|K|}{\delta t}(\rho^{n+1}_K-\rho^n_K) + \dive(\rho^n \bfu^n)_K=0,
\displaybreak[1]\\ \label{eu:eq:scheme_Eint} &
\forall K \in \mesh, \;
\dfrac{|K|}{\delta t}(\rho^{n+1}_K e^{n+1}_K-\rho^n_K e^n_K) + \dive(\rho^n e^n \bfu^n )_{K}
 +|K|\ p^n_K \,(\dive \bfu^n)_K =S^n_K,
\displaybreak[1]\\[2ex] \label{eu:eq:scheme_eos} &
\forall K \in \mesh, \;
 p^{n+1}_K=(\gamma-1)\ \rho^{n+1}_K\ e^{n+1}_K,
\displaybreak[1]\\[2ex]\displaystyle \nonumber &
\mbox{For } 1 \leq i \leq d,\ \forall \edge \in \edges,
\\[1ex] \label{eu:eq:scheme_mom} & \displaystyle \phantom{\forall K \in \mesh, \;}
\dfrac{|D_\edge|}{\delta t}(\rho^{n+1}_{D_\edge} u^{n+1}_{\edge,i}-\rho^n_{D_\edge} u^n_{\edge,i})
+ \dive(\rho^n u^n_i \bfu^n)_\edge
+ |D_\edge|\, (\gradi p)^{n+1}_{\edge,i}
=0.
\end{align}\label{eu:eq:scheme} 
\end{subequations}

All the terms have been previously introduced, except the convection term of the discrete internal energy Equation \eqref{eu:eq:scheme_Eint} which reads:
\[
\dive(\rho^n e^n \bfu^n)_{K} = \sum_{\edge \in \edges(K)} F_{K,\edge}^n e^n_\edge,
\]
where the face value $e^n_\edge$ is given by a monotone approximation, \ie\ either a first-order upwind (with respect to the mass flux $F_{K,\edge}^n$) or a MUSCL-like approximation \cite{pia-13-for}; unless specified, this latter choice is made here.
The corrective term $S^n_K$ of the internal energy balance \eqref{eu:eq:scheme_Eint} is derived in Section \ref{sec:corr_Euler}.
As in the barotropic case, a stabilization of the form \eqref{eq:stab_mom} may be introduced in the discrete momentum balance equation \eqref{eu:eq:scheme_mom}; in this case, the corresponding dissipation must be added to $S^n_K$  (see \cite{her-18-cons}).
%
%
\subsubsection{A one-dimensional Riemann problem...}

We assess the behaviour of the scheme on a Riemann problem, known as \textbf{Test case 3} of \cite{tor-13-rie}.
The left and right states are given by:
\[
\mbox{left state: } \begin{bmatrix} \rho_L=1 \\ u_L = 0 \\ p_L = 1000 \end{bmatrix}; \qquad
\mbox{right state: } \begin{bmatrix} \rho_R=1 \\ u_R = 0 \\ p_R = 0.001 \end{bmatrix}.
\]
The computational domain is $\Omega=(0,1)$ and the final time is $T=0.012$.
At the time $t=0$, the unknowns are given for $x<0.5$ by the left state, and by the right state otherwise.
The boundary conditions are Dirichlet conditions, the prescribed values being once again given by the left and right states.
The structure of the solution to this problem is the following \cite{tor-13-rie}: on the left side of the domain, a rarefaction wave travels to the left; it is separated by a contact discontinuity from a shock wave on the right side of the domain, travelling to the right.

\medskip
\paragraph{... on a really one-dimensional domain.}

First, we choose to discretize the domain as a real one-dimensional domain, in which case the space discretization with the Rannacher-Turek element is equivalent to the usual MAC scheme \cite{har-65-num,har-71-num}.
The space step $h$ is uniform, and its value is $h=1/1000$ for the results plotted in this section; the time step is equal to $\delta t = h / 100$.
Here, no stabilization term needs to be added to the discrete momentum balance equation.
We illustrate the effect of the corrective term on the density, the energy, the pressure, and the velocity in Figure \ref{fig:Euler_Riemann_w_wo}.
As expected, without correction, the scheme is not consistent, because the computed (approximate) jump at the shock does not satisfy the Rankine-Hugoniot jump relations (this error propagating to the whole solution).
A convergence study would show that the solution obtained without corrective terms converges, but to a limit that is not a weak solution to Euler equations.
On the opposite, with the correction, the discontinuities position and the constant states are correctly (exactly, up to rounding errors, for the latter) computed; when refining the mesh, the convergence is achieved essentially by sharpening the "approximate discontinuities", and we observe a first -order convergence for the velocity and the pressure (the unknowns which are constant through the contact discontinuity) and of order slightly greater than 0.8 for the density.

\medskip
We compare in Figure \ref{fig:Euler_Riemann_up_muscl} the results obtained with the proposed MUSCL scheme with the scheme of \cite{her-18-cons}, which uses a first-order upwind discretization of the convection term (in the three equations of the system).
As expected, the high-order approximation notably reduces the numerical diffusion, which essentially plagues the contact discontinuity.

\begin{figure}
    \centering
    \begin{minipage}{0.47\textwidth}
        \includegraphics[width=0.95\textwidth]{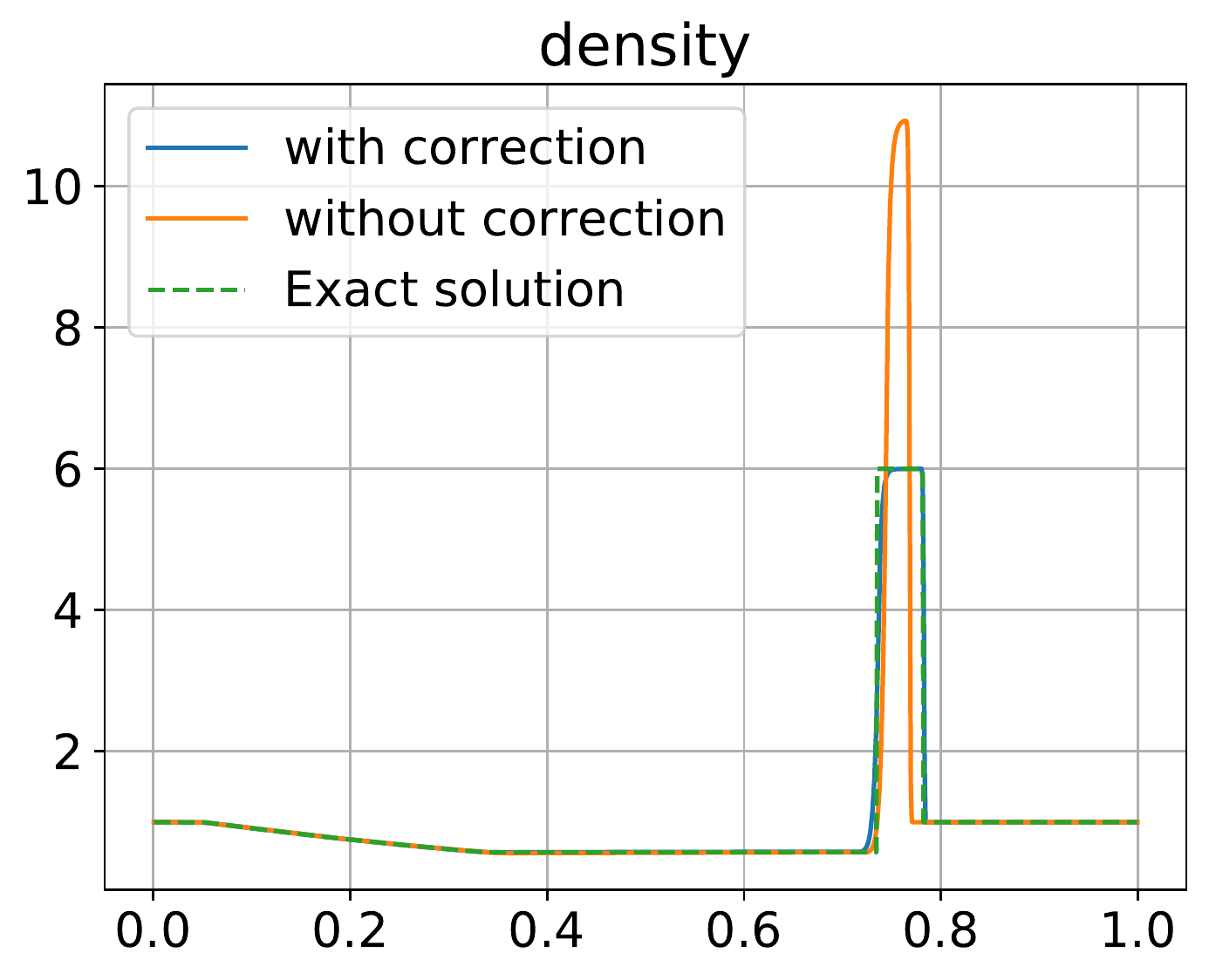}
    \end{minipage}%
    \begin{minipage}{0.47\textwidth}
        \includegraphics[width=\textwidth]{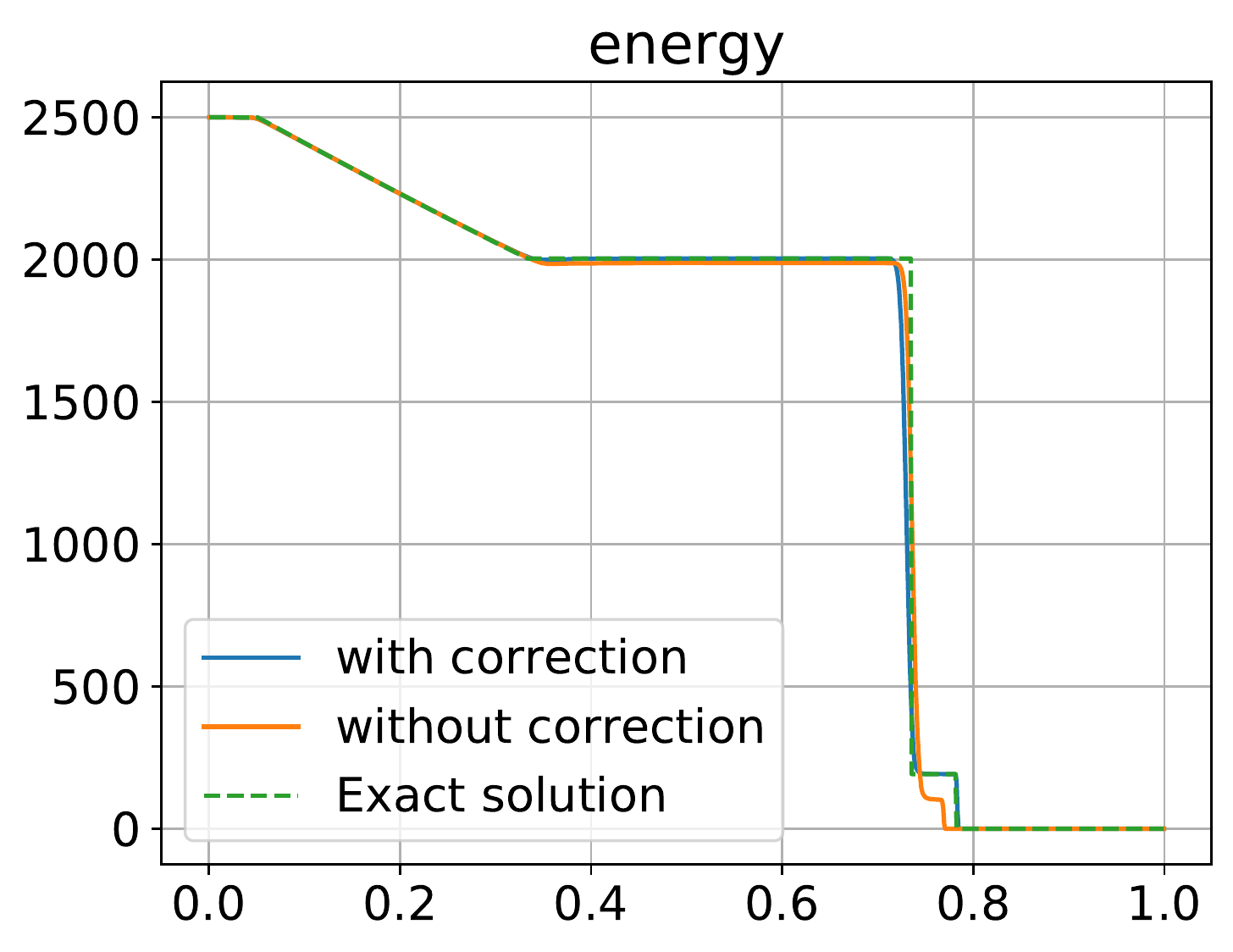}
    \end{minipage}
    \begin{minipage}{0.47\textwidth}
        \includegraphics[width=0.95\textwidth]{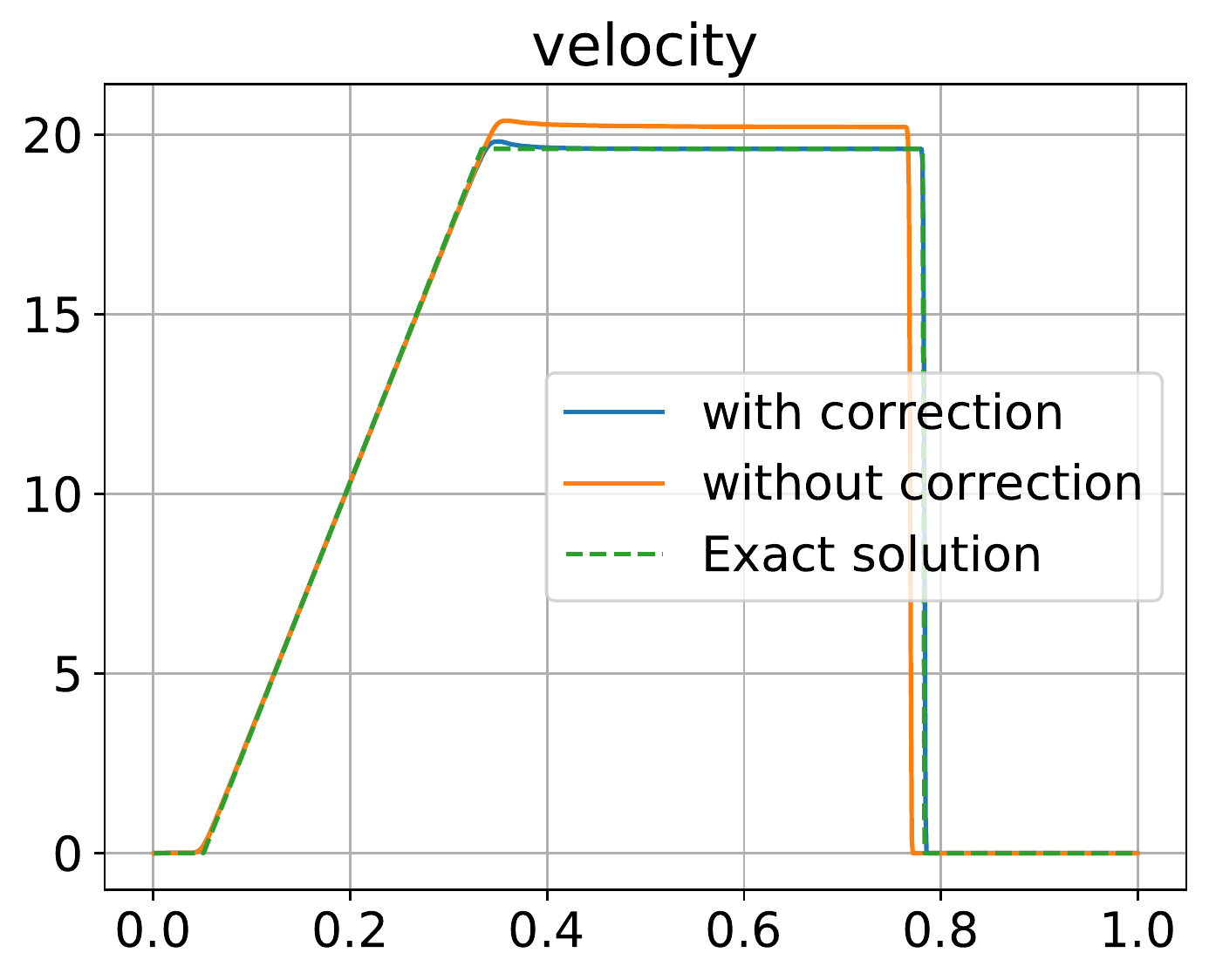}
    \end{minipage}
    \begin{minipage}{0.47\textwidth}
        \includegraphics[width=\textwidth]{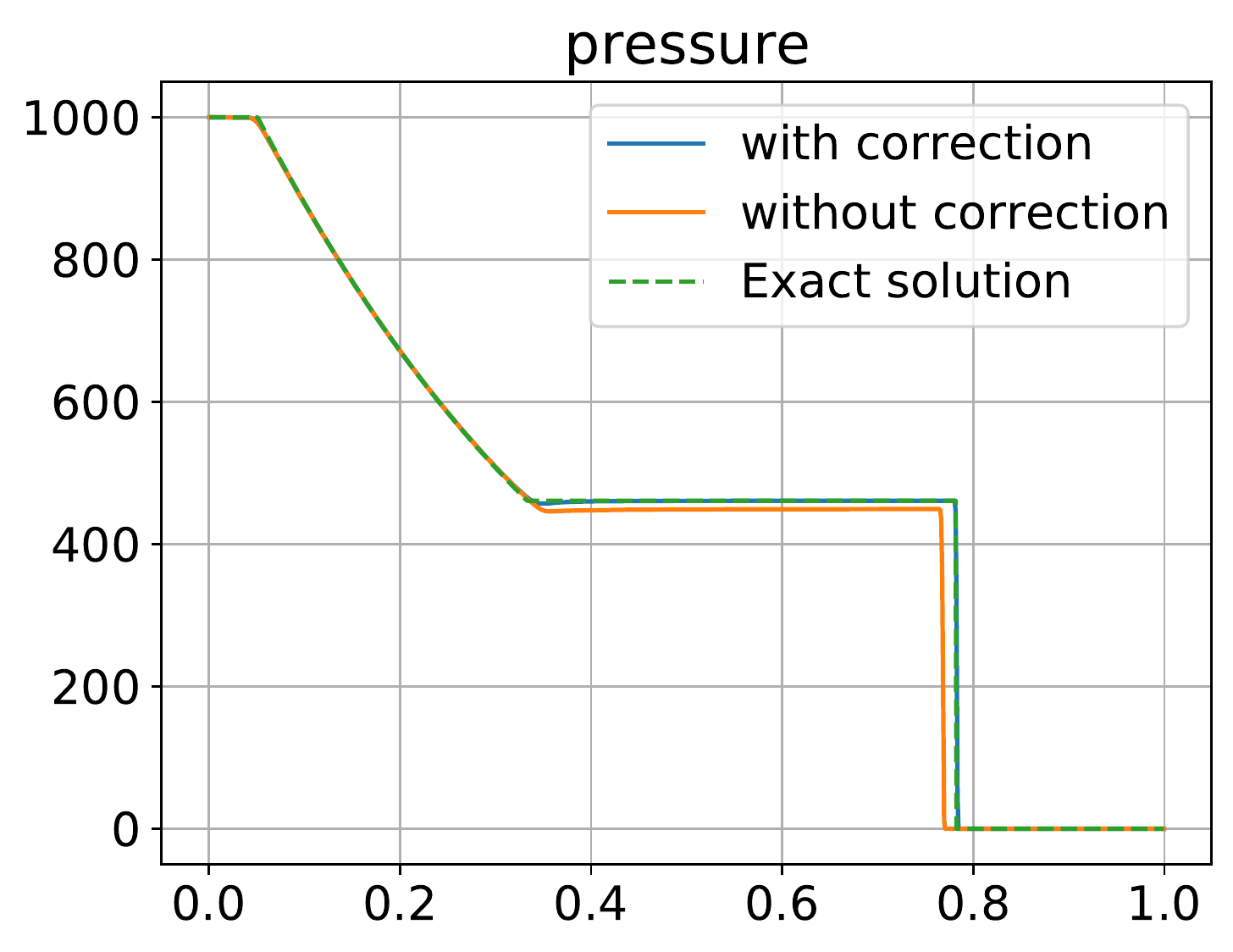}
    \end{minipage}
    \caption{Riemann problem for the Euler equations - Comparison of the results of the Test case 3 of \cite{tor-13-rie} for a MUSCL discretization of the convection term, with the corrective term of Section \ref{sec:corr_Euler} (in blue) or without it (in orange). Exact solution is plotted in green.}
    \label{fig:Euler_Riemann_w_wo}
\end{figure}

\begin{figure}
    \centering
    \begin{minipage}{0.47\textwidth}
        \includegraphics[width=0.95\textwidth]{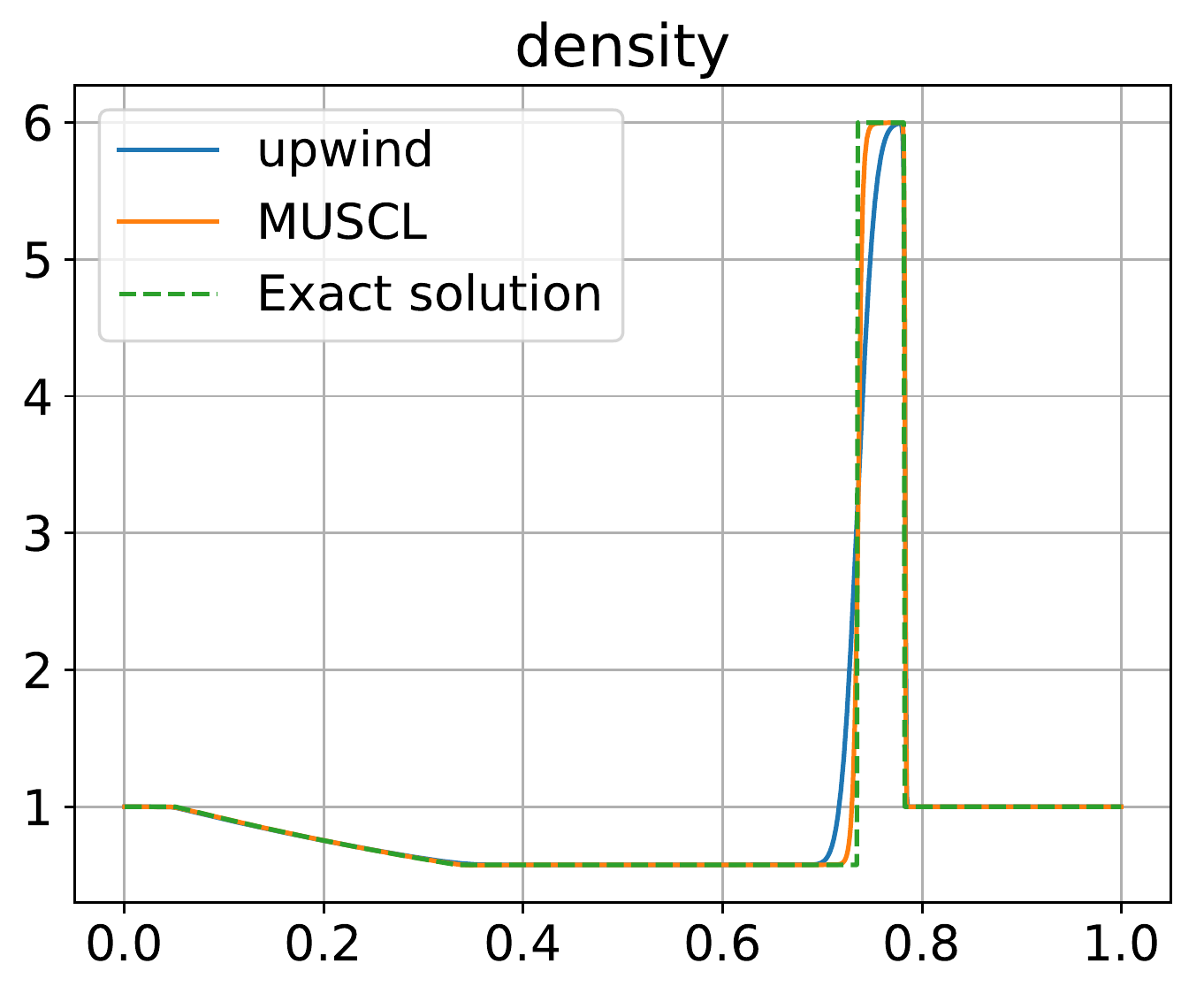}
    \end{minipage}%
    \begin{minipage}{0.47\textwidth}
        \includegraphics[width=\textwidth]{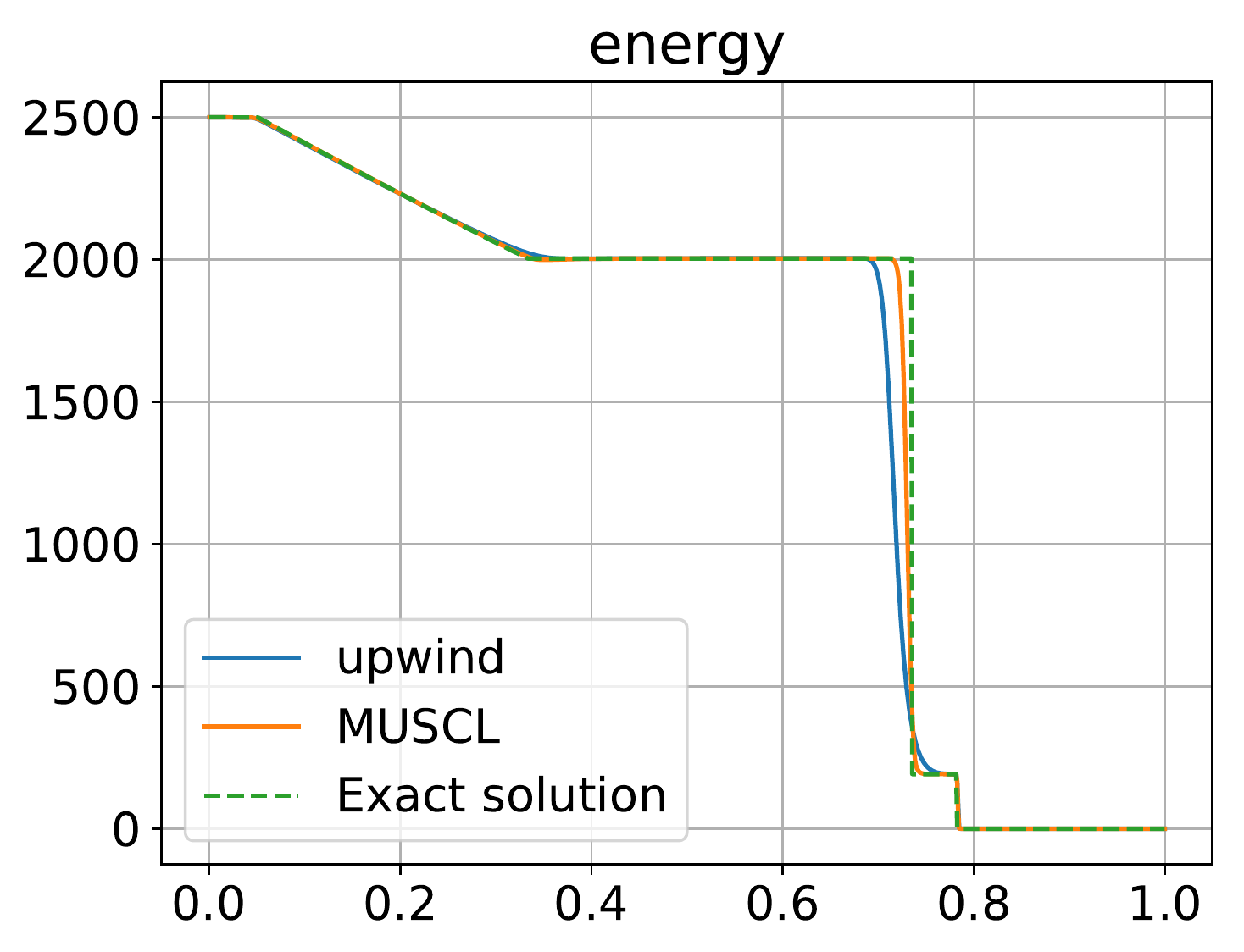}
    \end{minipage}
    \begin{minipage}{0.47\textwidth}
        \includegraphics[width=0.95\textwidth]{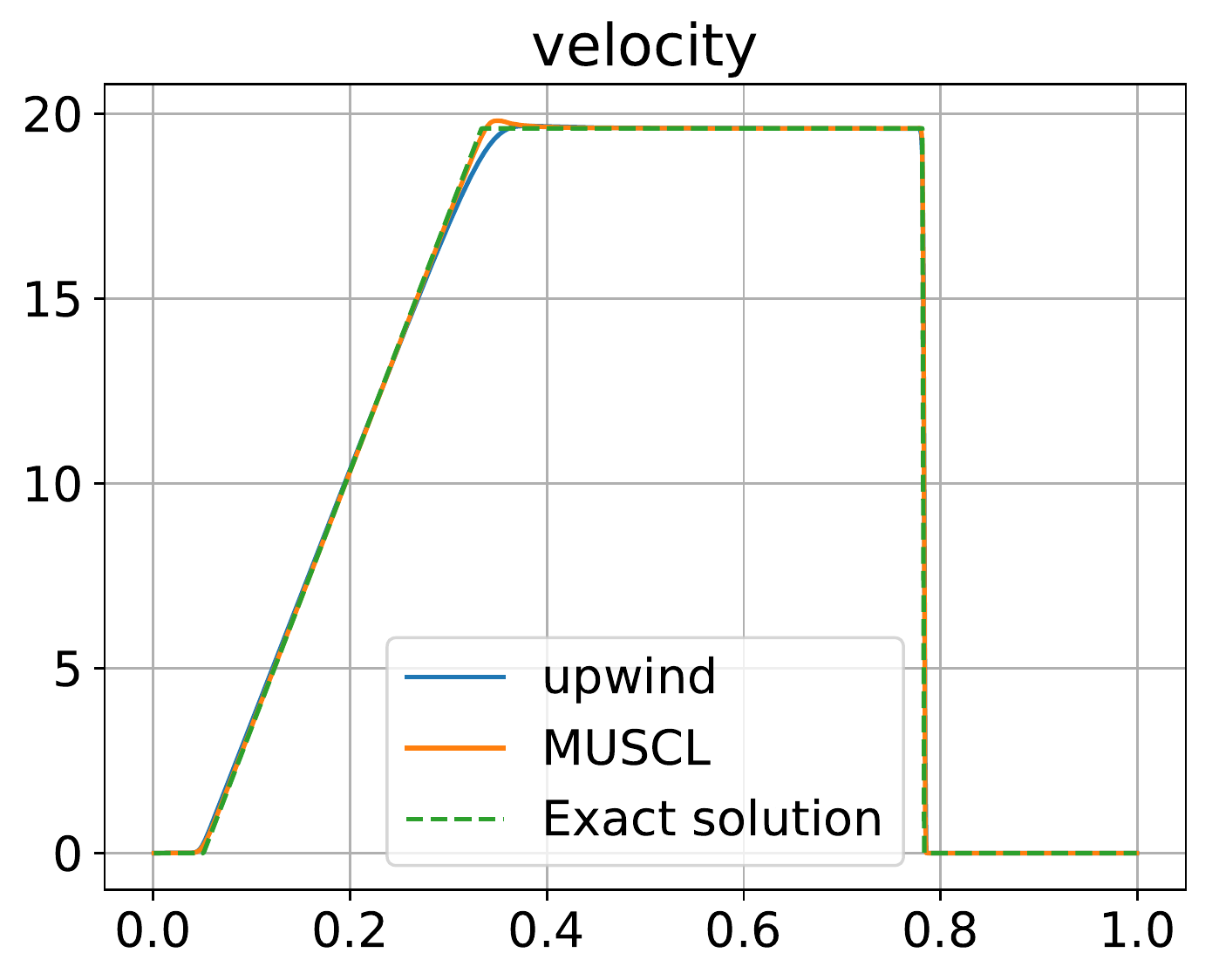}
    \end{minipage}
    \begin{minipage}{0.47\textwidth}
        \includegraphics[width=\textwidth]{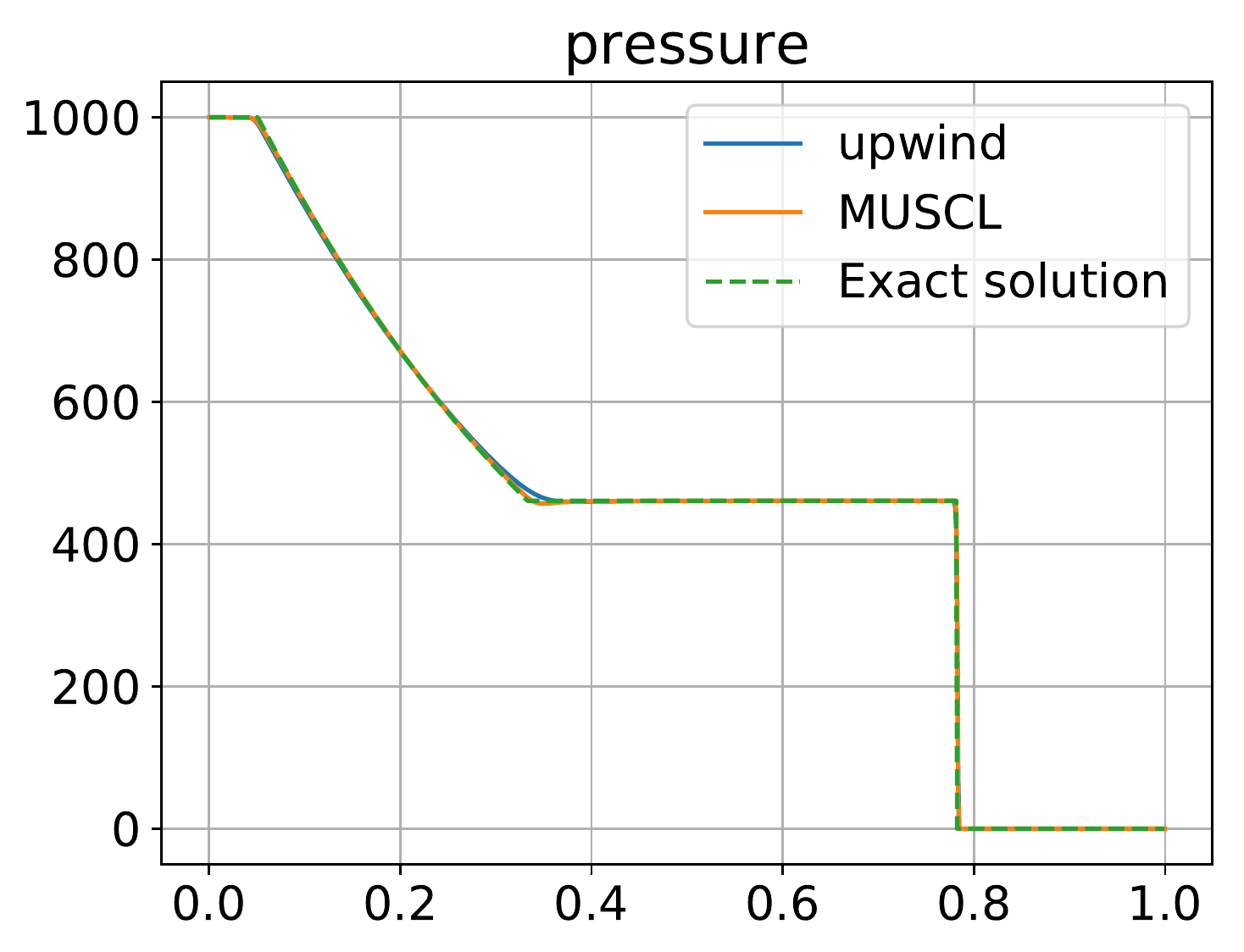}
    \end{minipage}
    \caption{Riemann problem for the Euler equations - Comparison of the results of the Test case 3 of \cite{tor-13-rie} with MUSCL (in orange) and upwind (in blue) discretization of the convection terms. Exact solution is plotted in green.}
    \label{fig:Euler_Riemann_up_muscl}
\end{figure}

\paragraph{... on a fictitious two-dimensional domain.}

In the one-dimensional case, the space discretization for the Rannacher-Turek element is quite different from the multi-dimensional case; in particular, in 1D, all the degrees of freedom of the velocity correspond to the normal component to the face; moreover, in 2D, the convection fluxes involve unknowns which are associated to non-aligned face centers.
Therefore, we reproduce the test with a "fictitious" two-dimensional domain.
This domain is now chosen as $\tilde{\Omega} = \Omega \times [0,h]$, where $h$ is the space step in the $x$- and $y$-direction (so the mesh consists of only one horizontal stripe of meshes), with again $h = 0.001$. 
Symmetry (or impermeability and perfect slip) boundaries condition are prescribed at the top and the bottom sides of the domain.
Now, as already observed in \cite{khe-13-pre}, the stabilization term given by \eqref{eq:stab_mom} has to be introduced in the momentum balance equation \eqref{eu:eq:scheme_mom}, to avoid an odd-even decoupling between the normal (\ie\ associated to vertical faces) and the tangential (\ie\ associated to horizontal external faces) $x$-components of the velocity.
The viscosity coefficient featured in \eqref{eq:stab_mom} is constant and set to $\nu = 50 h \ u_{\max} \ \rho_{\max}$ where $u_{\max}=19.6$ and $\rho_{\max}=6$ are the maximum velocity and density, respectively, given by the analytical solution.
Due to the explicit-in-time approximation of the viscosity term, the stability of the scheme is conditioned to a $\cfl$ criterion of the form $\delta t \leq c h^2$, so the time step is reduced and set to $\delta t = h / 200$.

\medskip
Results are compared to the ones obtained in the previous paragraph in Figure \ref{fig:Euler_Riemann_1D_2D}. 
A good agreement is observed, even though the introduction of the stabilization term leads to a slightly more diffusive scheme, as one could expect.

\begin{figure}
    \centering
    \begin{minipage}{0.47\textwidth}
        \includegraphics[width=0.95\textwidth]{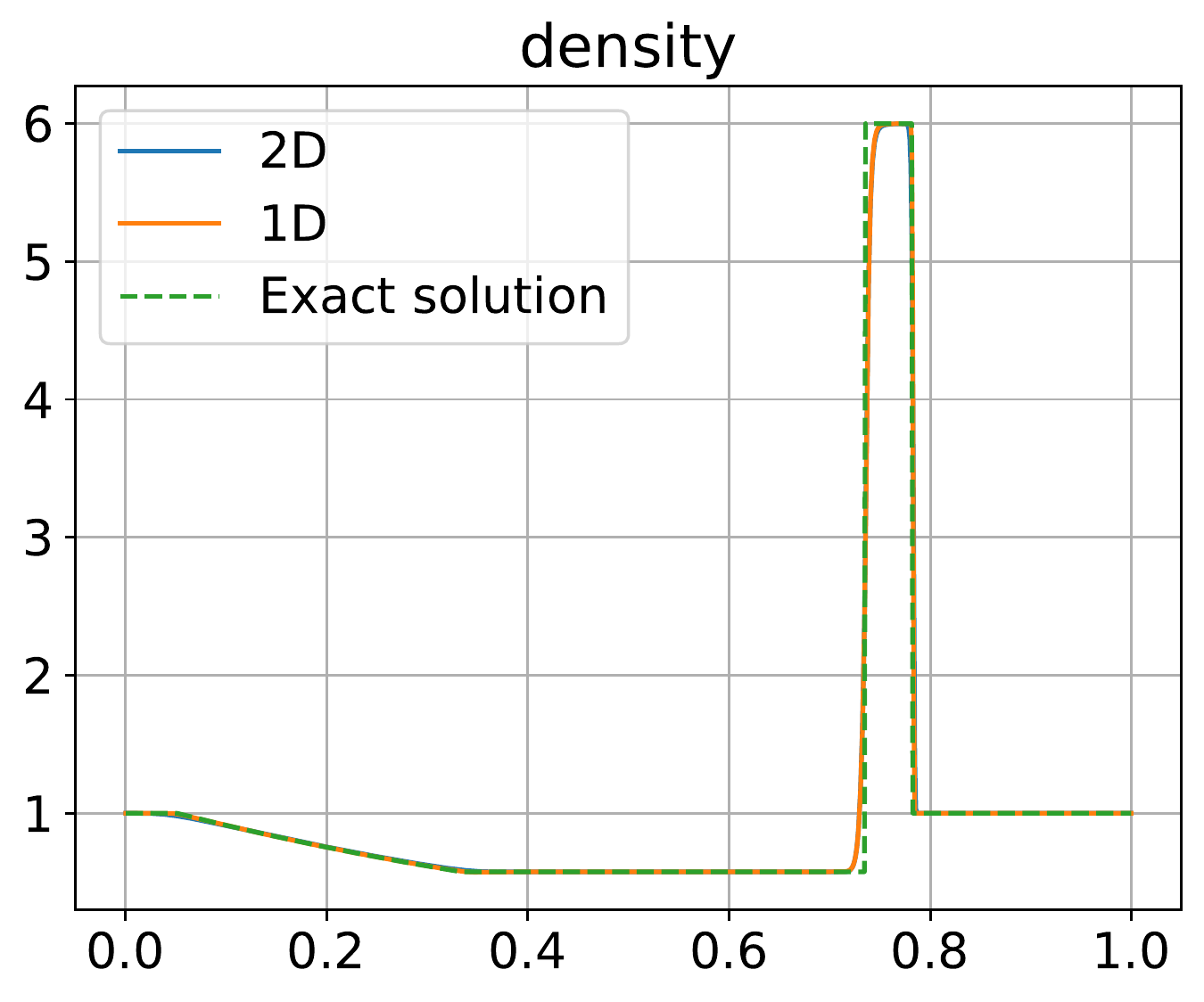}
    \end{minipage}%
    \begin{minipage}{0.47\textwidth}
        \includegraphics[width=\textwidth]{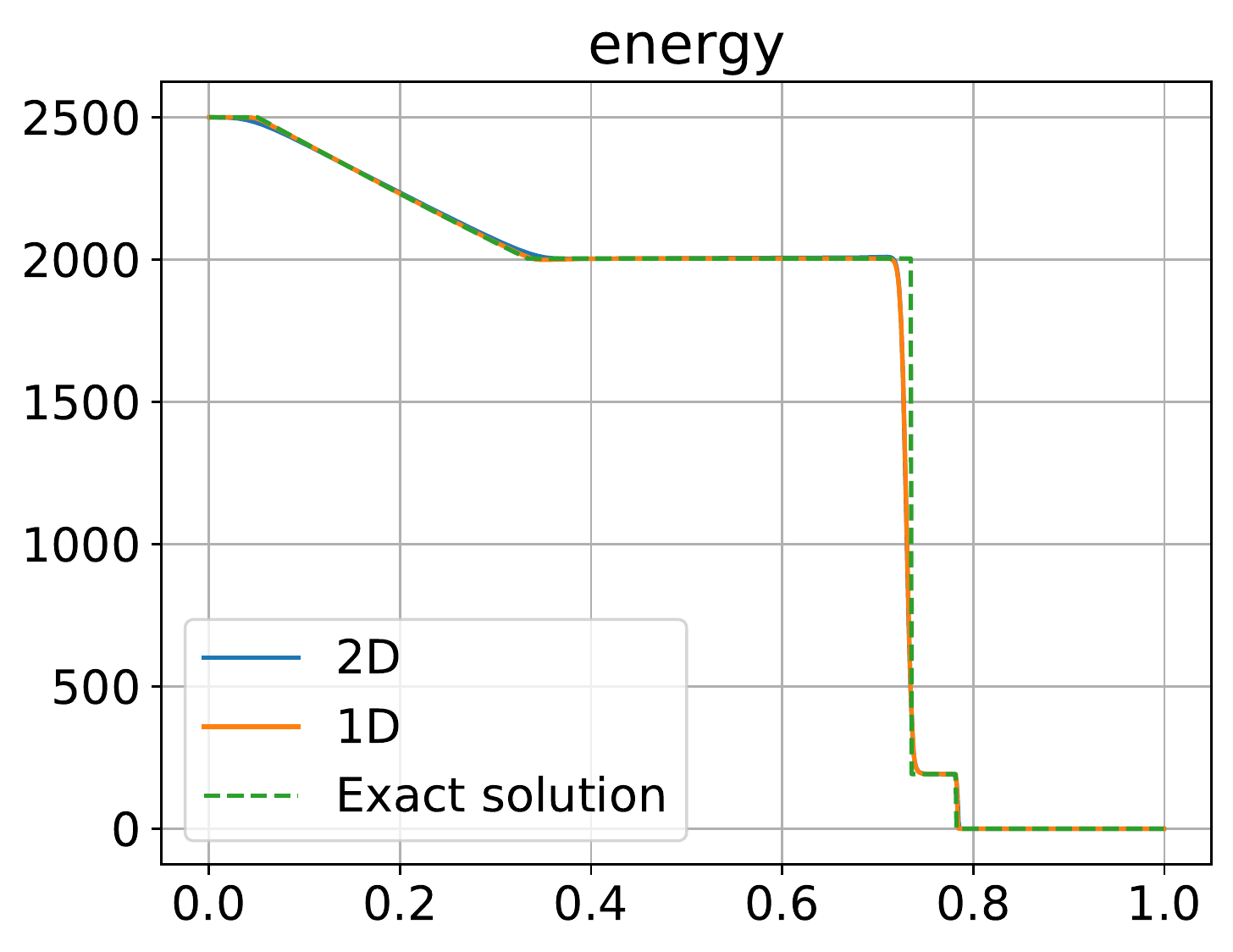}
    \end{minipage}
    \begin{minipage}{0.47\textwidth}
        \includegraphics[width=0.95\textwidth]{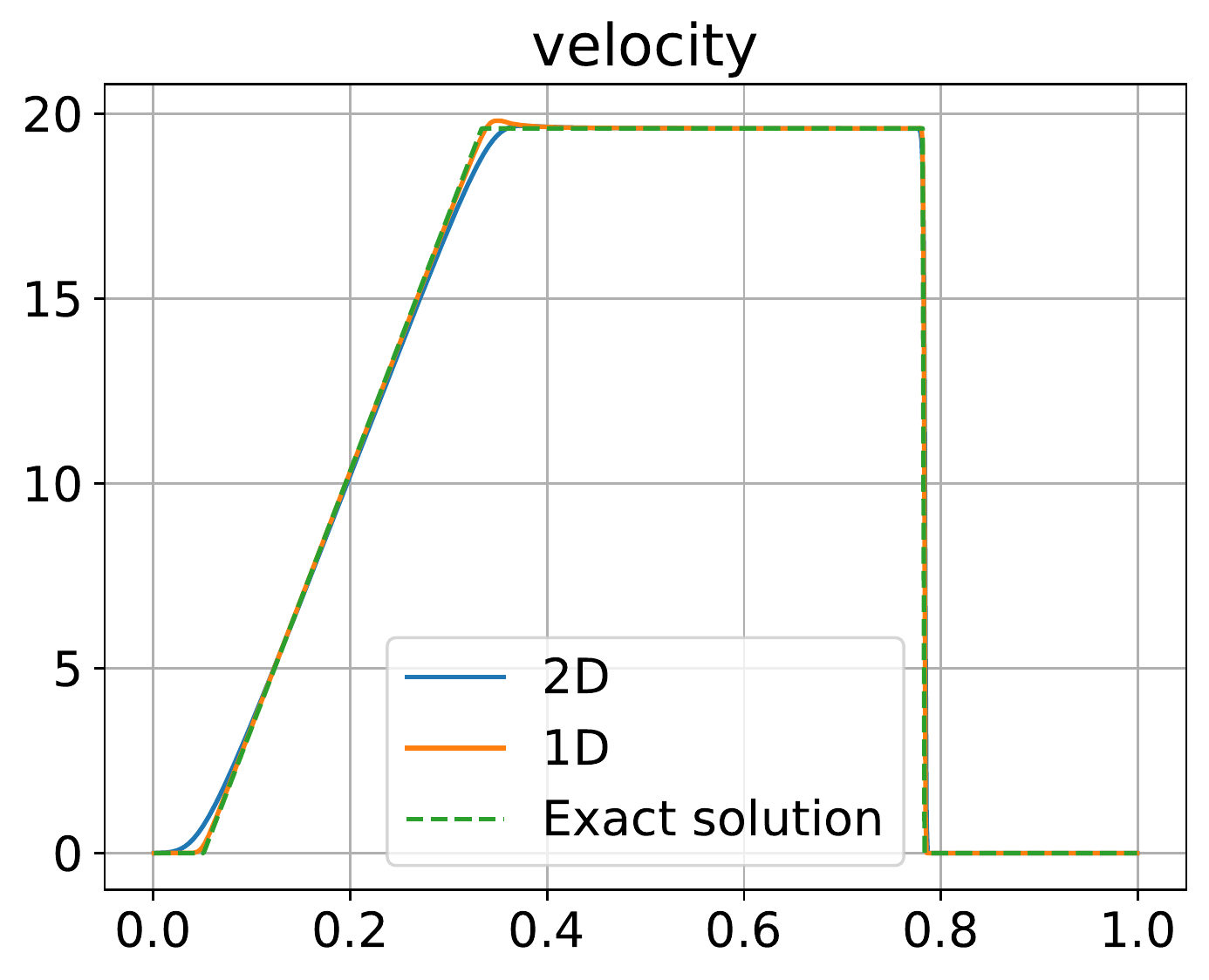}
    \end{minipage}
    \begin{minipage}{0.47\textwidth}
        \includegraphics[width=\textwidth]{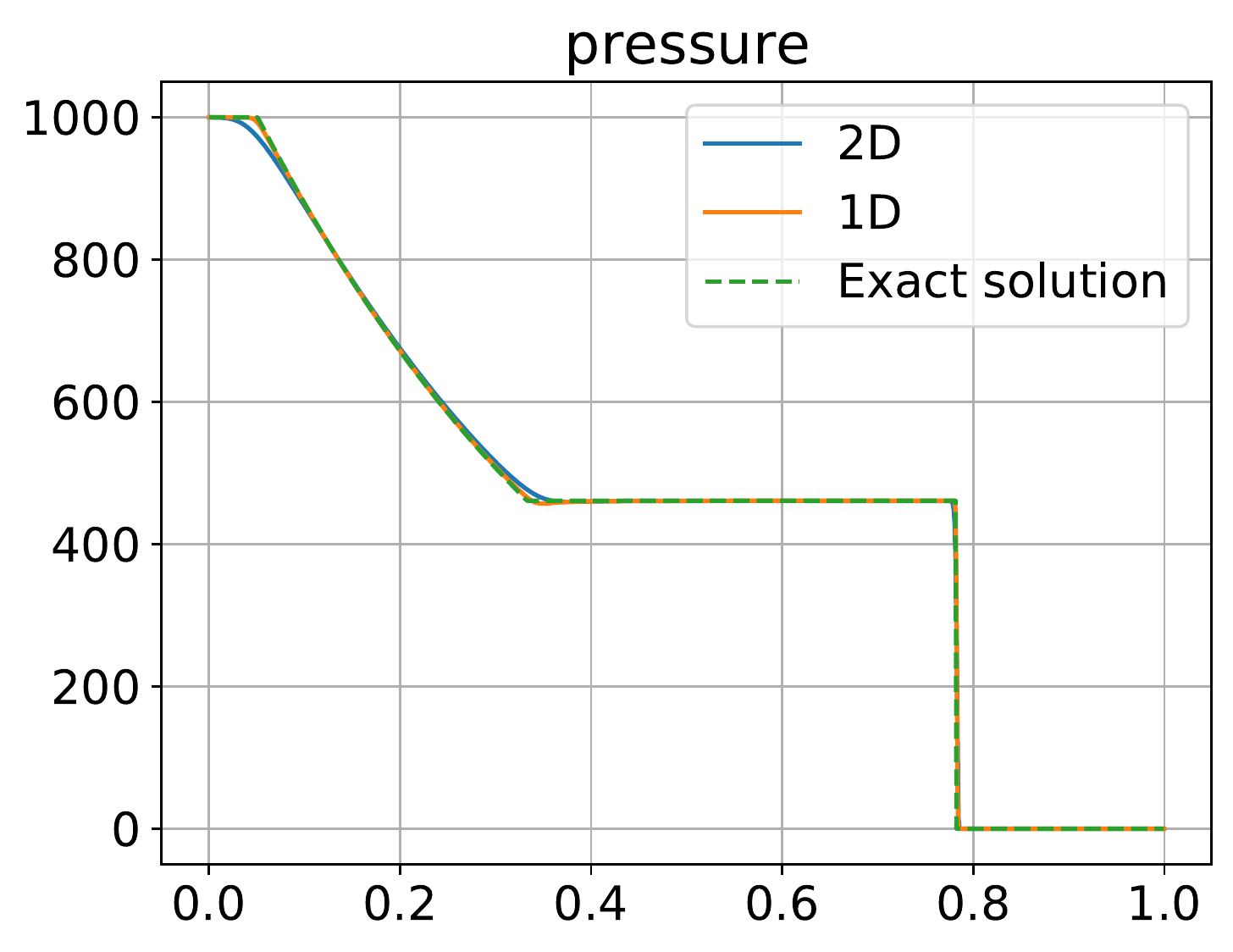}
    \end{minipage}
    \caption{Riemann problem for the Euler equations - Comparison of the results of the Test case 3 of \cite{tor-13-rie} for a MUSCL discretization of the convection term, on a one dimensional domain (in orange) or a fictitious two dimensional domain (in blue). Exact solution is plotted in green.}
    \label{fig:Euler_Riemann_1D_2D}
\end{figure}
%
%
\subsubsection{Flow past a cylinder}

We now address once again the problem of a flow past a cylinder, with the same domain as for the barotropic case.
Once again, we suppose that the initial data is a given homogeneous state with a fluid at rest, and we generate a shock travelling to the right by choosing suitable boundary conditions on the left side of the domain.
In addition, we tune the data to obtain a "non-isentropic analogue" of the case presented in Section \ref{sec:baro_cyl}.
We take $\gamma=2$, so that the usual entropy for the Euler equations reads $s=e/\rho$.
If the entropy were constant, the equation of state $p=(\gamma-1)\rho e$ would yield $p=s \rho^2$, and we would obtain the same problem as in Section \ref{sec:baro_cyl} provided that $s=a$.
We thus choose for the density the same value as in Section \ref{sec:baro_cyl}, \ie\ $\rho_0=0.2$, and the initial internal energy is given by $e_0=a \rho_0$.
The Mach number characterizing the shock is still $M=2$, its celerity is $\omega=M\,(\gamma p_0/\rho_0)^{1/2}$, and the Rankine-Hugoniot condition yields the values of the unknowns $(\rho_b,u_b,p_b)$ to be prescribed at the left boundary:
\[
\rho_b= \frac{\gamma +1}{\gamma -1+\dfrac 2 {M^2}}\ \rho_o,\qquad
u_b=\omega\ \bigl(1-\frac{\rho_o}{\rho_b}\bigr),\qquad
p=p_0+\omega^2 \bigl(1-\frac{\rho_0}{\rho_b}\bigr)\ \rho_0.
\]
Impermeability and perfect slip conditions are prescribed at the other boundaries, except the right one where we let the flow leave the domain, with the same technique as for the barotropic case.

\medskip
As in Section \ref{sec:baro_cyl}, we use a mesh that consists of $106897$ control volumes, and the time step is equal to $\delta t = 4.10^{-6}$.
The simulation is run until the final time $T=1$.
A stabilisation is once again needed, and the viscosity coefficient is chosen constant and equal to be roughly equal to $\overline{c}/10$, where $\overline{c}$ is the approximated sound of speed in the medium $\overline{c} := (\gamma p_{\max}/\rho_{\max})^{1/2}$ where $p_{\max} = 1.7$ and $\rho_{\max} = 0.5$.

\medskip
As in the barotropic case, the computations show a reflection of the shock on the obstacle, which generates a reflected shock (first curved then tending to a plane wave) travelling to the left (at a speed similar to the barotropic case), together with some complex structures in the obstacle wake, including vortex sheddings.
However, here, this latter phenomenon is much more visible (Figure \ref{fig:Euler_FPAC}).

\begin{figure}
\centering
\includegraphics[width=\textwidth]{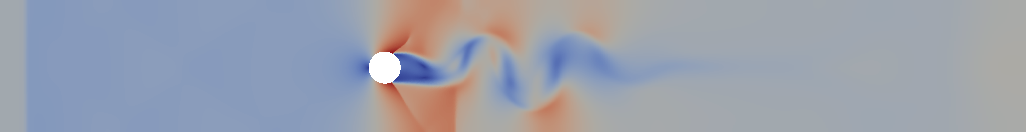} \\[2ex]
\includegraphics[width=\textwidth]{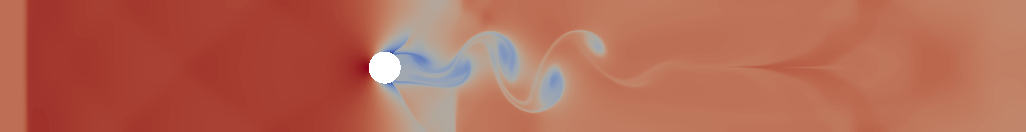} \\[2ex]
\includegraphics[width=\textwidth]{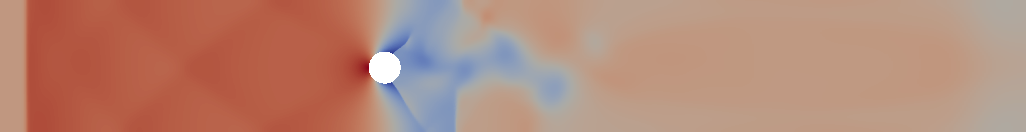}
\caption{Flow past a cylinder, Euler equations - From top to bottom, velocity, density, and pressure at time $t=1$.}
\label{fig:Euler_FPAC}
\end{figure}

\section{Conclusion}

In this work, we presented a discretization of the momentum convection operator for quadrilaterals or hexahedral meshes.
The discrete operator is based on a low-order finite-volume-like formulation on staggered discretization, and, due to its generic form, is valid for the simulations of both compressible and incompressible flows.
The computation of the interpolation of the velocity is done through an algebraic MUSCL procedure, designed to get a higher-order convection operator (that is, less diffusive than the classical upwind method) that does not yield spurious oscillations.
The limitation process is algebraic in the sense that it does not require a slope reconstruction of any kind, but rather hinges on stability conditions that are originally derived to yield a maximum principle for a transport equation.
Furthermore, we showed that it is possible to derive an approximate transport operator for the kinetic energy from this convection operator, which might be used as a primary step to prove a kinetic energy inequality for incompressible or barotropic flows or to derive consistent schemes for the Euler equations.
Finally, we presented numerical results for incompressible, barotropic or compressible flows.
In all these tests, we checked that the MUSCL method brought an enhancement compared with classical interpolation techniques: it appears to be more stable than the centered scheme and less diffusive than the upwind one.
On a Cartesian sequence of meshes, we also verified that the MUSCL scheme is higher-order than the upwind method, but is only almost second-order.
In the present formulation of the scheme, a second-order interpolation of the velocity at the faces is not reachable in general cases, since we chose not to precisely define the geometry of the dual mesh associated to the velocity.
Indeed, only the volume of these dual cells as well as the mass fluxes on these fake control volumes are needed to write the scheme.
These values are then computed from algebraic constraints, thought to verify a local discrete mass balance required for the derivation of the kinetic energy transport operator.
For a cell of a given polygon or polyhedron type, it is then possible to determine once and for all an explicit expression for these quantities, since these conditions are unique.
This brings two outcomes: first, we obtain an efficient computation of the dual mass fluxes; second, this construction is readily extendable to more general cells (such as prisms or pyramids for three-dimensional problems).
Such work is, for instance, conducted in \cite{bru-22-sta}, where we also prove that the construction of the dual mass fluxes from the stability requirements also implies their consistency.

\bibliographystyle{abbrv}
\bibliography{muscl}

\begin{thebibliography}{10}

\bibitem{ang-91-num}
L.~Angermann.
\newblock Numerical solution of second-order elliptic equations on plane
  domains.
\newblock {\em Mathematical Modelling and Numerical Analysis}, 25:169--191,
  1991.

\bibitem{ans-11-sta}
G.~Ansanay-Alex, F.~Babik, J.-C. Latch{\'e}, and D.~Vola.
\newblock An {L2}-stable approximation of the {N}avier--{S}tokes convection
  operator for low-order non-conforming finite elements.
\newblock {\em International Journal for Numerical Methods in Fluids},
  66:555--580, 2011.

\bibitem{arm-83-exp}
B.~Armaly, F.~Durst, J.~Pereira, and B.~Sch{\"o}nung.
\newblock Experimental and theoretical investigation of backward-facing step
  flow.
\newblock {\em Journal of fluid Mechanics}, 127:473--496, 1983.

\bibitem{bot-98-ben}
O.~Botella and R.~Peyret.
\newblock Benchmark spectral results on the lid-driven cavity flow.
\newblock {\em Computers \& Fluids}, 27:421--433, 1998.

\bibitem{bru-06-2d}
C.-H. Bruneau and M.~Saad.
\newblock The 2d lid-driven cavity problem revisited.
\newblock {\em Computers \& Fluids}, 35:326--348, 2006.

\bibitem{bru-22-sta}
A.~Brunel, R.~Herbin, and J.-C. Latché.
\newblock A staggered scheme for the compressible euler equations on general 3d
  meshes.
\newblock submitted, \url{https://arxiv.org/abs/2209.06474}, 2022.

\bibitem{buf-10-mon}
T.~Buffard and S.~Clain.
\newblock Monoslope and multislope {MUSCL} methods for unstructured meshes.
\newblock {\em Journal of Computational Physics}, 229:3745--3776, 2010.

\bibitem{cal-10-sta}
C.~Calgaro, E.~Chane-Kane, E.~Creus{\'e}, and T.~Goudon.
\newblock $l^\infty$-stability of vertex-based {MUSCL} finite volume schemes on
  unstructured grids: simulation of incompressible flows with high density
  ratios.
\newblock {\em Journal of Computational Physics}, 229:6027--6046, 2010.

\bibitem{cal-19-com}
C.~Calgaro, C.~Colin, and E.~Creusé.
\newblock A combined finite volumes - finite elements method for a low-mach
  model.
\newblock {\em International Journal for Numerical Methods in Fluids},
  90(1):1--21, 2019.

\bibitem{cal-13-pos}
C.~Calgaro, E.~Creus\'e, T.~Goudon, and Y.~Penel.
\newblock Positivity-preserving schemes for {E}uler equations: sharp and
  practical {CFL} conditions.
\newblock {\em Journal of Computational Physics}, 234:417--438, 2013.

\bibitem{califs}
CALIF$^3$S.
\newblock A software components library for the computation of fluid flows.
\newblock \\ \texttt{https://gforge.irsn.fr/gf/project/califs}.

\bibitem{chi-99-num}
T.~Chiang, T.~Sheu, and C.~Fang.
\newblock Numerical investigation of vortical evolution in a backward-facing
  step expansion flow.
\newblock {\em Applied Mathematical Modelling}, 23:915--932, 1999.

\bibitem{cla-10-sta}
S.~Clain and V.~Clauzon.
\newblock {$L^\infty$} stability of the {MUSCL} methods.
\newblock {\em Numerische Mathematik}, 116:31--64, 2010.

\bibitem{cro-73-con}
M.~Crouzeix and P.~Raviart.
\newblock Conforming and nonconforming finite element methods for solving the
  stationary {S}tokes equations.
\newblock {\em RAIRO S\'erie Rouge}, 7:33--75, 1973.

\bibitem{eym-06-com}
R.~Eymard, D.~Hilhorst, and M.~Vohral\'ik.
\newblock A combined finite volume--nonconforming/mixed-hybrid finite element
  scheme for degenerate parabolic problems.
\newblock {\em Numerische Mathematik}, 105:73--131, 2006.

\bibitem{fei-95-com}
M.~Feistauer, J.~Felcman, and M.~Luk{\'a}{\v{c}}ov{\'a}-Medvid'ov{\'a}.
\newblock Combined finite element-finite volume solution of compressible flow.
\newblock {\em Journal of computational and applied mathematics},
  63(1-3):179--199, 1995.

\bibitem{fei-97-con}
M.~Feistauer, J.~Felcman, and M.~Luk{\'a}{\v{c}}ov{\'a}-Medvid'ov{\'a}.
\newblock On the convergence of a combined finite volume-finite element method
  for nonlinear convection-diffusion problems.
\newblock {\em Numerical Methods for Partial Differential Equations},
  13(2):163--190, 1997.

\bibitem{gal-20-sec}
T.~Gallou{\"e}t, R.~Herbin, J.-C. Latch{\'e}, and Y.~Nasseri.
\newblock A second order consistent {MAC} scheme for the shallow water
  equations on non uniform grids.
\newblock In {\em International Conference on Finite Volumes for Complex
  Applications}, pages 123--131. Springer, 2020.

\bibitem{gas-18-mus}
L.~Gastaldo, R.~Herbin, J.-C. Latch{\'e}, and N.~Therme.
\newblock A muscl-type segregated--explicit staggered scheme for the euler
  equations.
\newblock {\em Computers \& Fluids}, 175:91--110, 2018.

\bibitem{ghi-82-hig}
U.~Ghia, K.~Ghia, and C.~Shin.
\newblock {H}igh-{R}e solutions for incompressible flow using the
  {N}avier-{S}tokes equations and a multigrid method.
\newblock {\em Journal of Computational Physics}, 48:387--411, 1982.

\bibitem{gue-06-ano}
J.~Guermond, P.~Minev, and J.~Shen.
\newblock An overview of projection methods for incompressible flows.
\newblock {\em Computer Methods in Applied Mechanics and Engineering},
  195:6011--6045, 2006.

\bibitem{gue-96-som}
J.-L. Guermond.
\newblock Some implementations of projection methods for {N}avier-{S}tokes
  equations.
\newblock {\em Mathematical Modelling and Numerical Analysis}, 30:637--667,
  1996.

\bibitem{har-71-num}
F.~Harlow and A.~Amsden.
\newblock A numerical fluid dynamics calculation method for all flow speeds.
\newblock {\em Journal of Computational Physics}, 8:197--213, 1971.

\bibitem{har-65-num}
F.~Harlow and J.~Welsh.
\newblock Numerical calculation of time-dependent viscous incompressible flow
  of fluid with free surface.
\newblock {\em Physics of Fluids}, 8:2182--2189, 1965.

\bibitem{her-14-cons}
R.~Herbin, W.~Kheriji, and J.-C. Latch\'e.
\newblock On some implicit and semi-implicit staggered schemes for the shallow
  water and {E}uler equations.
\newblock {\em Mathematical Modelling and Numerical Analysis}, 48:1807--1857,
  2013.

\bibitem{her-21-cons}
R.~Herbin, J.-C. Latch\'{e}, S.~Minjeaud, and N.~Therme.
\newblock Conservativity and weak consistency of a class of staggered finite
  volume methods for the {E}uler equations.
\newblock {\em Mathematics of Computation}, 90:1155--1177, 2021.

\bibitem{her-18-cons}
R.~Herbin, J.-C. Latch{\'e}, and T.~Nguyen.
\newblock Consistent segregated staggered schemes with explicit steps for the
  isentropic and full {E}uler equations.
\newblock {\em Mathematical Modelling and Numerical Analysis}, 52:893--944,
  2018.

\bibitem{khe-13-pre}
W.~Kheriji, R.~Herbin, and J.-C. Latch\'e.
\newblock Pressure correction staggered schemes for barotropic monophasic and
  two-phase flows.
\newblock {\em Computers \& Fluids}, 88:524--542, 2013.

\bibitem{let-15-mul}
C.~Le~Touze, A.~Murrone, and H.~Guillard.
\newblock Multislope {MUSCL} method for general unstructured meshes.
\newblock {\em Journal of Computational Physics}, 284:389--418, 2015.

\bibitem{ohm-84-tec}
K.~Ohmori and T.~Ushijima.
\newblock A technique of upstream type applied to a linear nonconforming finite
  element approximation of convective diffusion equations.
\newblock {\em RAIRO. Analyse num{\'e}rique}, 18:309--332, 1984.

\bibitem{pia-13-for}
L.~Piar, F.~Babik, R.~Herbin, and J.-C. Latch{\'e}.
\newblock A formally second-order cell centred scheme for convection--diffusion
  equations on general grids.
\newblock {\em International Journal for Numerical Methods in Fluids},
  71:873--890, 2013.

\bibitem{ran-92-sim}
R.~Rannacher and S.~Turek.
\newblock Simple nonconforming quadrilateral {S}tokes element.
\newblock {\em Numerical Methods for Partial Differential Equations},
  8:97--111, 1992.

\bibitem{sch-96-opt}
F.~Schieweck and L.~Tobiska.
\newblock An optimal order error estimate for an upwind discretization of the
  {N}avier-{S}tokes equations.
\newblock {\em Numerical Methods for Partial Differential Equations},
  12:407--421, 1996.

\bibitem{she-92-err}
J.~Shen.
\newblock On error estimates of projection methods for {N}avier-{S}tokes
  equations: First-order schemes.
\newblock {\em SIAM Journal on Numerical Analysis}, 29:57--77, 1992.

\bibitem{tor-13-rie}
E.~F. Toro.
\newblock {\em Riemann solvers and numerical methods for fluid dynamics: a
  practical introduction}.
\newblock Springer Science \& Business Media, 2013.

\bibitem{van-79-tow}
B.~Van~Leer.
\newblock Towards the ultimate conservative difference scheme. v. a
  second-order sequel to godunov's method.
\newblock {\em Journal of computational Physics}, 32:101--136, 1979.

\end{thebibliography}

\end{document}